\documentclass[10.5pt,reqno]{amsart}
\usepackage{amssymb,mathrsfs,graphicx}
\usepackage{epsfig,subfigure}
\usepackage{ifthen}
\usepackage{hyperref}
\usepackage{float}
\usepackage{colortbl}

\usepackage{ifthen} %  for  boolean variables
\provideboolean{shownotes} % define a boolean variable
\setboolean{shownotes}{true} % defined as ``true'' or ``false''
\newcommand{\margnote}[1]{
\ifthenelse{\boolean{shownotes}}%
{\marginpar{\raggedright\tiny\texttt{#1}}}%
{}%
}

\newcommand{\dem}{\frac{1}{2}}
\newcommand{\hole}[1]{
\ifthenelse{\boolean{shownotes}}%
{\begin{center} \fbox{ \rule {.25cm}{0cm} \rule[-.1cm]{0cm}{.4cm}
\parbox{.85\textwidth}{\begin{center} \texttt{#1}\end{center}} \rule
{.25cm}{0cm}}\end{center}} {} }
\title[Derivation of the Boltzmann equation with soft potentials from a particles system]{Derivation of the Boltzmann equation with moderately soft potentials from a perturbed  Nanbu particles system}

\author[Salem]{Samir Salem}
\address[Samir Salem]{\newline Ecole Polytechnique, CMLS 91128 Palaiseau Cedex, France}
\email{samir.salem@polytechnique.edu}

\numberwithin{equation}{section}

\newtheorem{theorem}{Theorem}[section]
\newtheorem{lemma}{Lemma}[section]

\newtheorem{proposition}{Proposition}[section]

\newcommand{\R}{\mathbb R}

\newcommand{\B}{\mathbb{B}}
\newcommand{\Sd}{\mathbb{S}^{2}}

\newcommand{\II}{\mathcal{I}}
\newcommand{\VV}{\mathcal{V}}

\newcommand{\BB}{\mathcal{B}}
\newcommand{\MM}{\mathcal{M}}
\newcommand{\NN}{\mathcal{N}}
\newcommand{\WW}{\mathcal{W}}

\newcommand{\KK}{\mathcal{K}}
\newcommand{\PP}{\mathcal{P}}
\newcommand{\FF}{\mathcal{F}}
\newcommand{\EE}{\mathcal{E}}

\newcommand{\CC}{\mathcal{C}}
\newcommand{\DD}{\mathcal{D}}
\newcommand{\ZZ}{\mathcal{Z}}

\newcommand{\ee}{\mathbf{e}}

\newcommand{\mei}{\frac{1}{N}\sum_{i=1}^N}
\newcommand{\mej}{\frac{1}{N}\sum_{j=1}^N}
\newcommand{\meij}{\frac{1}{N^2}\sum_{i\neq j}^N}

\newcommand{\e}{\varepsilon}

\newcommand{\lal}{\langle}
\newcommand{\ral}{\rangle}

\newcommand{\lt}{\left}
\newcommand{\rt}{\right}
\newcommand{\pa}{\partial}
\newcommand{\Ps}{\mathcal{P}_{\textit{sym}}(\R^{dN})}
 
\newcommand{\mb}{\mathbf{1}}

\newcommand{\bq}{\begin{equation}}
\newcommand{\eq}{\end{equation}}

\newcommand{\LL}{\mathcal{L}}

\newcommand{\E}{\mathbb{E}}
\newcommand*{\rom}[1]{\expandafter\@slowromancap\romannumeral #1@}
\def\charf {\mbox{{\text 1}\kern-.30em {\text l}}}
\begin{document}
%%%%%%%%%%%%%%%%
\allowdisplaybreaks

\date{\today}



\begin{abstract} 
We derive the 3D spatially homogeneous Boltzmann's equation with moderately soft potentials and singular angular interaction, from an interacting particles system. The collision kernel is of the form $B(z,\sigma)=|z|^{\gamma}b\lt( \frac{z}{|z|}\cdot \sigma\rt)$ and for $K>0$, $\sin(\theta)b\lt(\cos(\theta)\rt)\sim K\theta^{-1-\nu}$, with $\gamma \in (-2,-1)$ and $\nu\in(1,2)$ satisfying $\gamma+\nu>0$. We use at the particle level the regularizing effects of the grazing collisions, in order to control the singularity of the soft potential. This enables to use a classical compactness argument, and provide a qualitative convergence result from the interacting particles system toward the solution of the limit macroscopic equation. 
\end{abstract}

\maketitle \centerline{\date}

%\tableofcontents

%%%%%%%%%%%%%%%%%%%%%%%%%%%%%%%%%%%%%%%%%%%%%%%%%%%%%%%%%%%%%%%%%%%%%%%%%%%%%%%%%5
%
%
%                        Section: Introduction
%
%
%%%%%%%%%%%%%%%%%%%%%%%%%%%%%%%%%%%%%%%%%%%%%%%%%%%%%%%%%%%%%%%%%%%%%%%%%%%%%%%%%

\section{Introduction}

The Boltzmann equation is a fundamental model of statistical physic. It describes the time evolution of the kinetic distribution of particles in a perfect diluted gas. The particles move with constant velocity until their paths reach another particle, in which case a collision happens and the pre-collisional velocities are changed. We are here interested in the spatially homogeneous case, and if we denote $g_t(v)$ the particle density at time $t$ at point $v$, it solves
\bq
\label{eq:Bol}
\pa_t g_t(v)=\int_{\R^3\times\mathbb{S}^{2}}B(v-v_*,\sigma)\lt(g_t(v')g_t(v'_*)-g_t(v)g_t(v_*)\rt)dv_*d\sigma,
\eq
where $'$ denotes a post-collisional velocity defined as
\begin{align*}
&v'=\frac{v+v_*}{2}+\frac{|v-v_*|}{2}\sigma, \ 
v'_*=\frac{v+v_*}{2}-\frac{|v-v_*|}{2}\sigma.\\
\end{align*}
The coefficient $B$ is a nonnegative function on $\R^3\times\Sd$ called the \textit{collision kerenl}, which depends on the nature of the interaction between particles. We are interested here in a collision kernel $B$ satisfying  for some $\gamma \in [-3,1]$, some $\nu\in (0,2)$ and $K>0$
	\bq
	\label{eq:H1}
	B(v-v_*,\sigma)=|v-v_*|^\gamma b\lt( \frac{v-v_*}{|v-v_*|} \cdot \sigma\rt),
	\eq
	and
	\bq
	\label{eq:bGraz}
	\cos(\theta)=  \frac{v-v_*}{|v-v_*|} \cdot \sigma, \quad \beta(\theta):= \sin(\theta)b(\cos(\theta))\underset{\theta\sim 0}{\sim} K \theta^{-(1+\nu)}.
	\eq 
	The angle $\theta$ is the angle between the pre and post-collisional relative velocities, and is called \textit{the angle of deviation}. We thus talk about \textit{grazing collisions} kernel, since the collisions with small angle of deviation are more weighted by the kernel. When $\gamma>0$, we are in the case of \textit{hard potentials}, for $\gamma=0$ in the case of \textit{Maxwellian molecules} and for $\gamma<0$ in the case of \textit{soft potentials}. In this paper, we focus on the \textit{moderately soft potentials}, that is the special case $\gamma \in (-2,0)$. At the limit $\nu\rightarrow 2^{-}$, that is the grazing collision limit, equation \eqref{eq:Bol} becomes the Landau equation  (see for instance on this topic \cite{Des},\cite{VilGraz}). In the physically relevant cases of the inverse power-law potentials (see for instance \cite{A},\cite{VilRev}), these parameters are given with respect to some parameter $s$ describing the repulsion potential operating between two particles, as
\[
\gamma=\frac{s-5}{s-1}, \ \nu=\frac{2}{s-1}.
\]

In this paper, we address the question of the propagation of chaos for equation \eqref{eq:Bol}. This question almost goes back to Boltzmann's historical paper \cite{Bol}, in which he derives the equation which now bares his name, from the dynamic of the atoms composing the gas. To do so, some assumptions coming from the classical mechanics law are used, such as the elasticity of the collisions between atoms, or the reversibility of the atomic dynamic. But he also uses an assumption of statistical nature, namely the \textit{Stosszahlansatz }. Literally, "assumption about counting of the chocks", also known as \textit{molecular chaos},  this assumption means that the correlations between two particles among all the particles composing the gas, are negligible. We refer the reader interested in historical and heuristical considerations about this topic to the french note \cite{Barbe}. \newline
In one of the founding papers of the mathematical kinetic theory \cite{Kac}, Kac introduces a probabilistic framework to formalize the Boltzmann's idea of the molecular chaos, namely the chaos property (we refer to the lecture notes \cite{Szn}). Defining a $N$ component Markov process, he showed that along this process the chaos property is propagated in time. This enables to justify the molecular chaos assumption, and thus to rigorously derive a toy model for the Boltzmann equation.  \newline

The rest of the paper is organized as follows. In Section 2, we introduce the particles system from which we derive equation \eqref{eq:Bol}, make some comments about the more or less recent literature on the topic. We state the main result of the paper Theorem \ref{thm:main1}, and give a sketch of proof. In Section 3, we establish the proof of Theorem \ref{thm:Fish1}, which contains the bounds, uniform in the number of particles in the system, of the quantities of interest. Finally, in Section 4, we use these bounds to complete the proof of Theorem \ref{thm:main1}, using a classical martingale method (see for instance \cite{SznBol}).\newline
Two appendices are dedicated to gathering some properties of some coefficients of the particles system. Two others, to the proof of the key Propositions \ref{prop:HLS} and \ref{thm:Fish2} respectfully.

\textbf{Notation} : 
\begin{itemize}
	\item Lebesgue and Sobolev's norm: For $s\in (0,1)$ the fractional Sobolev's norm as
	\[
	|f|^2_{H^s(\R^d)}= \int_{\R^{2d}} \frac{(f(x)-f(y))^2}{|x-y|^{d+2s}}dxdy. 
	\]
	Denoting the Fourier transform as $\FF(f)(\xi)=\int_{\R^d} e^{-i v\cdot \xi}f(v)dv$, we may also define it as	
	\[
	|f|^2_{H^s(\R^d)}=\int_{\R^d} \lt|\FF(f)(\xi)\rt|^2|\xi|^{2s}d\xi.
	\]
	We denote $\lal x \ral=\sqrt{1+|x|^2}$, and for $k\geq 0$
	\[
	L^1_k=\{ f\in L^1, \ s.t. \ \lal \cdot \ral^k f \in L^1 \},
	\]
	and 
	\[
	L\ln L=\bigl\{ f\in L^1(\R^d), \ f> 0 \ s.t. \  H(f):=\int_{\R^d} f \ln f <\infty \}.
	\]
	
	\item Probability measures : For a functional $F$ on $\R^{dN}$, and $i=1,\cdots,N$ we note $\nabla_i F=\mathbf{a}_i\cdot \nabla F^N \in \R^d$, where $\mathbf{a}_i=\lt(\underbrace{0,\cdots 0}_{d(i-1)} ,
	\underbrace{1,\cdots, 1}_{d},\underbrace{0\cdots,0}_{d(N-i)}\rt) \in \R^{dN}$. The notation $V$ will stand for the integration variable in $\R^{dN}$, and for $V\in \R^{dN}$, $V^{N-1}_i$ stands for $(v_1,\cdots,v_{i-1},v_{i+1},\cdots,v_N)$.\\
	The notation $\PP(E)$ stands for the set of probability measures on $E$, $\PP_{\mbox{sym}}(E^N)$ stands for the set of sequences of symmetric probabilities on $E^N$, i.e. invariant by permutation. $W_p$ is the Wasserstein metric on $\PP(\R^d)$ of order $p\geq 1$. For $T>0$, the notation $\mathbb{D}([0,T];\R^d)$ stands for the \textit{c\`adl\`ag} (right continuous with left limits) paths on $\R^d$, or equivalently the Skorokhod space from $[0,T]$ to $\R^d$. We define $\mathbf{e}_t:\gamma\in \mathbb{D}(0,T;\R^d)\mapsto \gamma(t)\in \R^d$ the evaluation map at time $t$, and for $\rho\in\PP\lt(\mathbb{D}(0,T;\R^d)\rt)$ we implicitly associate the family of probability measures $(\rho_t\in \PP(\R^d))_{t\in[0,T]}$ defined as $\rho_t=\mathbf{e}_t\# \rho$. \\
	All the probability measure at stake in the paper are assumed to admit a density with respect to the Lebesgue measure, and the confusion will be abused between measures and their densities.

	\item Miscellaneous : For the sake of simplicity, numerical constants are all denoted $C$, and when constants depend on parameters of the problem, this dependence is expressed in the index.  For $\gamma\geq 0$, we denote $K_\gamma:x\in \R^3\mapsto |x|^\gamma x\in \R^3$.  We consider $(\rho_{\e})_{\e>0}$ a family of even mollifying kernels, such that  for any $k\geq 2$
	\[
	\sup_{\e \in (0,1)} \int_{\R^3} |w|^k\rho_\e(dw)\leq C_k, \ \forall v \in \R^3, \ \sup_{\e \in (0,1)}|K_\gamma*\rho_\e(v)|\leq C(1+|K_\gamma(v)|). 
	\]
	Finally for $A>0$, we define a smooth function $\chi_A:\R^+\mapsto [0,1]$ such that $\mb_{[0,A]}\leq \chi_A \leq \mb_{[0,A+(1\vee A)]}$.

\end{itemize}

\section{Preliminaries and main results of the paper}

\subsection{Nanbu particles system}

We follow here \cite[Section 1.5]{FGue}. For $X \in \R^3$ we define $I(X), J(X)$ two unit vectors such that $\lt(\frac{X}{|X|}, \frac{I(X)}{|X|}, \frac{J(X)}{|X|} \rt)$ is a direct orthonormal basis of $\R^3$, and for $\theta \in [0,\pi]$ and $\varphi \in [0,2\pi]$ define

\bq
\label{eq:Para}
\left\{\begin{matrix}
	 \displaystyle  \Gamma(X,\varphi):=\cos(\varphi)I(X)+\sin(\varphi)J(X)	\\ 
	 \displaystyle v':=v-\frac{1-\cos(\theta)}{2}(v-v_*)+\frac{\sin(\theta)}{2}\Gamma(v-v_*,\varphi)	\\ 
	 \displaystyle v'_*:=v_*+\frac{1-\cos(\theta)}{2}(v-v_*)-\frac{\sin(\theta)}{2}\Gamma(v-v_*,\varphi)	\\ 
	 \displaystyle a(v,v_*,\theta,\varphi)=v'-v=-(v'_*-v_*).	
\end{matrix}\right.
\eq
Throughout all the paper, we work under the assumption
\begin{itemize}
\item[\textbf{(H)}] There are $K_1>K_2>0$ such that for any $\theta \in (0,\pi]$, there holds
\begin{equation*}
K_1\theta^{-1-\nu} \geq \beta(\theta) \geq K_2\theta^{-1-\nu}.
\end{equation*}
\end{itemize}
We refer to \cite[Lemma 1.1]{FGue} for the implication of this assumption. We also fix the notation 
\[
\beta_0=\pi \int_0^\pi (1-\cos(\theta))\beta(\theta) d\theta.
\]
Define
\bq
\label{eq:c}
c_{\gamma,\nu}(v,v_*,z,\varphi)=a\lt(v,v_*,G_\nu\lt( \frac{z}{|v-v_*|^\gamma}  \rt),\varphi\rt), \quad G_\nu(z)=\inf\Bigl\{x\in (0,\pi), \ \int_x^\pi \beta(\theta) d\theta \leq z \Bigr\},
\eq
and introduce the notations for $\phi\in \CC^2(\R^3)$ and respectfully $\Phi\in \CC^2(\R^{3N})$ and $i=1,\cdots,N$
\begin{align*}
&\tilde{\Delta}\phi\lt( v,v_*,z,\varphi  \rt)=\phi\lt( v+c_{\gamma,\nu}(v,v_*,z,\varphi) \rt)-\phi\lt( v \rt)- c_{\gamma,\nu}(v,v_*,z,\varphi)\cdot \nabla \phi(v)\\
&\tilde{\Delta}_i\Phi\lt( V,v_*,z,\varphi  \rt)=\Phi\lt( V+c_{\gamma,\nu}(v,v_*,z,\varphi)\mathbf{a}_i \rt)-\Phi\lt( V \rt)- c_{\gamma,\nu}(v,v_*,z,\varphi)\cdot \nabla_i \Phi(V).
\end{align*}

Finally for any $R>1,\delta,\eta\in(0,1)$ we introduce the smooth cut off function on probability measures on $\R^3$
\bq
\label{eq:al}
\begin{aligned}
	\alpha^R_{\delta,\eta}(g)=(1-\chi_{\frac{\eta^2}{2}})\lt(\sup_{\ee \in \Sd} \int_{\R^{3}\times \R^3} \chi_R(v)\chi_R(v_*)\chi_{2\delta} \lt((v-v_*)\cdot\ee \rt)g(dv)g(dv_*)\rt)+\chi_{4\eta}\lt( \int_{\R^3} \chi_{R-1}(v)g(dv) \rt),
\end{aligned}
\eq
which heuristically values one if $g$ the mass of the ball of radius $R$ is too small (w.r.t. to the threshold $4\eta$) and if it is not the case, if a too large proportion of this mass lies closely to some plane. \newline

We can now define the interacting particles system. Consider a sequence $(\e_N)_{N\geq 2}$ converging to zero. For $N\geq 2$ we denote $(\MM_{N}^{i})_{ i=1,\cdots,  N}$, $N$ independent Poisson random measures (see for instance \cite[Chapter VI]{Cin}) on $\R^+\times \R^+\times[0,2\pi]\times\R^3\times \{1,\cdots,N\}\setminus \{i\}$ with intensity $ds\times dz \times d\varphi\times \rho_{\e_N}(dw)\times \frac{1}{N-1}\sum_{j\neq i} \delta_j$. Let also $(\BB_t^i)_{t\geq 0, i=1,\cdots,N}$ be $N$ independent Brownian motions independent of the $(\MM_{N}^{i})_{ i=1,\cdots,  N}$. Then we consider the following system of SDEs with jumps

\bq
\begin{aligned}
	\label{eq:BKnu>1}
	\VV_t^{N,i}=&\VV_0^{N,i}+ \int_{[0,t]\times [0,2\pi]\times\R^+ \times \R^3\times \{1,\cdots,N\}\setminus \{i\} }c_{\gamma,\nu}\lt( \VV^{N,i}_{s_-},\VV^{N,j}_{s_-}+w,z,\varphi \rt)\bar{\MM}^{i}_{N}(ds,d\varphi,dz,dw,dj)\\
	&+ \frac{1}{N-1}\sum_{j\neq i}^N \int_0^t \beta_0K_\gamma*\rho_{\e_N}(\VV_s^{N,i}-\VV_s^{N,j})ds + \int_0^t \sqrt{2} \alpha^R_{\delta,\eta}\lt(\frac{1}{N-1}\sum_{j\neq i}^N \delta_{\VV_s^{N,j}}*\rho_{\e_N}\rt) d\BB_s^{i}
\end{aligned}
\eq
Ito's rule for jump process yields (see for instance \cite[Theorem 4.4.7]{Ebe}) that for any smooth test function $\Phi\in \CC^2(\R^{3N})$
\begin{align*}
&\Phi(\VV_t^N)=\Phi(\VV_0^N)+\sum_{i=1}^N\int_0^t\frac{\beta_0}{N-1}\sum_{j\neq i}^NK_\gamma*\rho_{\e_N}(\VV_s^{N,i}-\VV_s^{N,j})\cdot \nabla_i\Phi(\VV_s^N)+  \lt(\alpha^R_{\delta,\eta}\lt(\frac{1}{N-1}\sum_{j\neq i}^N \delta_{\VV_s^{N,j}}*\rho_{\e_N}\rt) \rt)^2\Delta_i \Phi(\VV_s^N) ds\\
&+\sum_{i=1}^N\int_0^t\frac{1}{N-1}\sum_{j\neq i}\int_{\R^3}\int_{0}^{2\pi}\int_0^\infty  \tilde{\Delta}_i\Phi(\VV_{s^-}^N, \VV_{s_-}^{N,j}+w,z,\varphi)dz d\varphi \rho_{\e_N}(dw) ds\\
&+\sum_{i=1}^N \int_{[0,t]\times [0,2\pi]\times\R^+ \times \R^3\times \{1,\cdots,N\}\setminus \{i\}}\lt(\Phi\lt(\VV_{s_-}^N+c_{\gamma,\nu}\lt( \VV_{s_-}^{N,i}, \VV_{s_-}^{N,j}+w,z,\varphi \rt)\mathbf{a}_i\rt)-\Phi\lt(\VV_{s_-}^N\rt)\rt)\bar{\MM}^{i}_{N}(ds,d\varphi,dz,dw,dj)\\
&+\int_0^t\sqrt{2} \alpha^R_{\delta,\eta}\lt(\frac{1}{N-1}\sum_{j\neq i}^N \delta_{\VV_s^{N,j}}*\rho_{\e_N}\rt) \nabla_i \Phi(\VV_s^N) \cdot d\BB_s^{i}.
\end{align*}

Taking the expectation, we find that $(G_t^N)_{t\geq 0}$, where $G_t^N=\LL(\VV_t^N)\in \Ps$, is weak solution to the Master Equation

\bq
\label{eq:ME}
\pa_t G_t^N=\mathcal{A}^{N}G_t^N,
\eq
where the generator $\mathcal{A}^{N}$ is defined by duality as
\begin{align*}
\mathcal{A}^{N}\Phi(V)&=\frac{1}{N-1}\sum_{i=1}^N\sum_{j\neq i}^N  \int_{ \R^3}\rho_{\e_N}(dw)\lt( \beta_0K_\gamma(v_i-v_j-w) \cdot\nabla_i \Phi(V)+ \int_{0}^{2\pi} \int_{\R^+}d\varphi dz  \tilde{\Delta}_i\Phi(V,v_j+w,z,\varphi)\rt)\\
&+ \sum_{i=1}^N \lt(\alpha^R_{\delta,\eta}\lt(\frac{1}{N-1}\sum_{j\neq i}^N \delta_{v_j}*\rho_{\e_N}\rt)\rt)^2\Delta_i \Phi(V) 
\end{align*}
 Then for fixed $v_i,w$, we use the change of variable $\theta=G_\nu\lt(   \frac{z}{|v_i-w|^\gamma} \rt)$, which yields $dz=-|v_i-w|^\gamma\beta(\theta)d\theta $,
and then
\begin{align*}
\int_{0}^{2\pi}d\varphi \int_{\R^+}dz \lt(\Phi(V+c_{\gamma,\nu}(v_i,w,z,\varphi)\mathbf{a}_i)-\Phi(V)\rt)&=\int_{0}^{2\pi}d\varphi\int_{0}^{\pi}d\theta |v_i-w|^\gamma\beta(\cos(\theta)) \lt(\Phi(V+a(v_i,w,\theta,\varphi)\mathbf{a}_i)-\Phi(V)\rt)\\
&=\int_{\Sd}B(v_i-w,\sigma) \lt(\Phi(V+(v'_i-v_i)\mathbf{a}_i)-\Phi(V)\rt) d\sigma,
\end{align*}
so that, in view of point $(i)$ of Lemma \ref{lem:use}, we can rewrite $\mathcal{A}^N$ as
\bq
\label{eq:defG2}
\begin{aligned}
	\mathcal{A}^N\Phi(V)&=\frac{1}{N-1}\sum_{i=1}^N \sum_{j\neq i}^N\int_{\Sd}  \int_{ \R^3}B(v_i-(v_j+w),\sigma)  \lt(\Phi(V+(v'_i-v_i)\mathbf{a}_i)-\Phi(V)\rt)\rho_{\e_N}(w)dw d\sigma\\
	&+\sum_{i=1}^N\lt(\alpha^R_{\delta,\eta}\lt(\frac{1}{N-1}\sum_{j\neq i}^N \delta_{v_j}*\rho_{\e_N}\rt)\rt)^2\Delta_i \phi(V)\\
	&:=\NN^N\Phi(V)+\PP^N\Phi(V).
\end{aligned}
\eq
The generator $\NN^N$ is a slight mollification (w.r.t. to the colliding particle, for computational purposes) of the classical Nanbu particles, as considered in \cite{SN,Xu}. Let us roughly describe the heuristic of this particles system. For each $i=1,\cdots,N$, we attach to the $i$-th particle a Poisson point process of intensity $dt$, which roughly plays the role as a clock to count the collision times. Each time this clock rings, a particle $j$ (with $j\neq i$) is chosen with uniform probability to enter into collision with the $i$-th particle (up to a mollification of the $j$-th particle w.r.t. $\rho_{\e_N}$). Then an azimuth $\varphi$ is chosen with uniform probability, and the deviation angle is chosen according to \eqref{eq:c}. After which, the trajectory of the $i$-th particle is affected, but not the one of the $j$-th particle, which makes this system not so relevant from a physical point of view. Yet it enables to perform our analysis.\newline
The perturbation $\PP^N$ is a mathematical artifact adapted from \cite{MN}, and has no real physical meaning. It is meant to enables to use a the microscopic level, the regularizing effects of the grazing collisions, well known at the macroscopic scale \cite{A,ADVW,Vil99}. More precisely, whenever the particles are too illed configured to apply this theory, some independent Brownian diffusions act on each particles, so they provide the regularizing effects that we can no longer obtain from grazing collisions (more details are provided in the sequel). This perturbation is expected to vanish as the number of particles goes to infinity.

\subsection{Main results}

We begin this section with some comments about the literature concerning the problem of propagation of chaos for the Boltzmann and Landau equations.\newline 
\begin{figure}[h]
	\centering
	\includegraphics[scale=1.6]{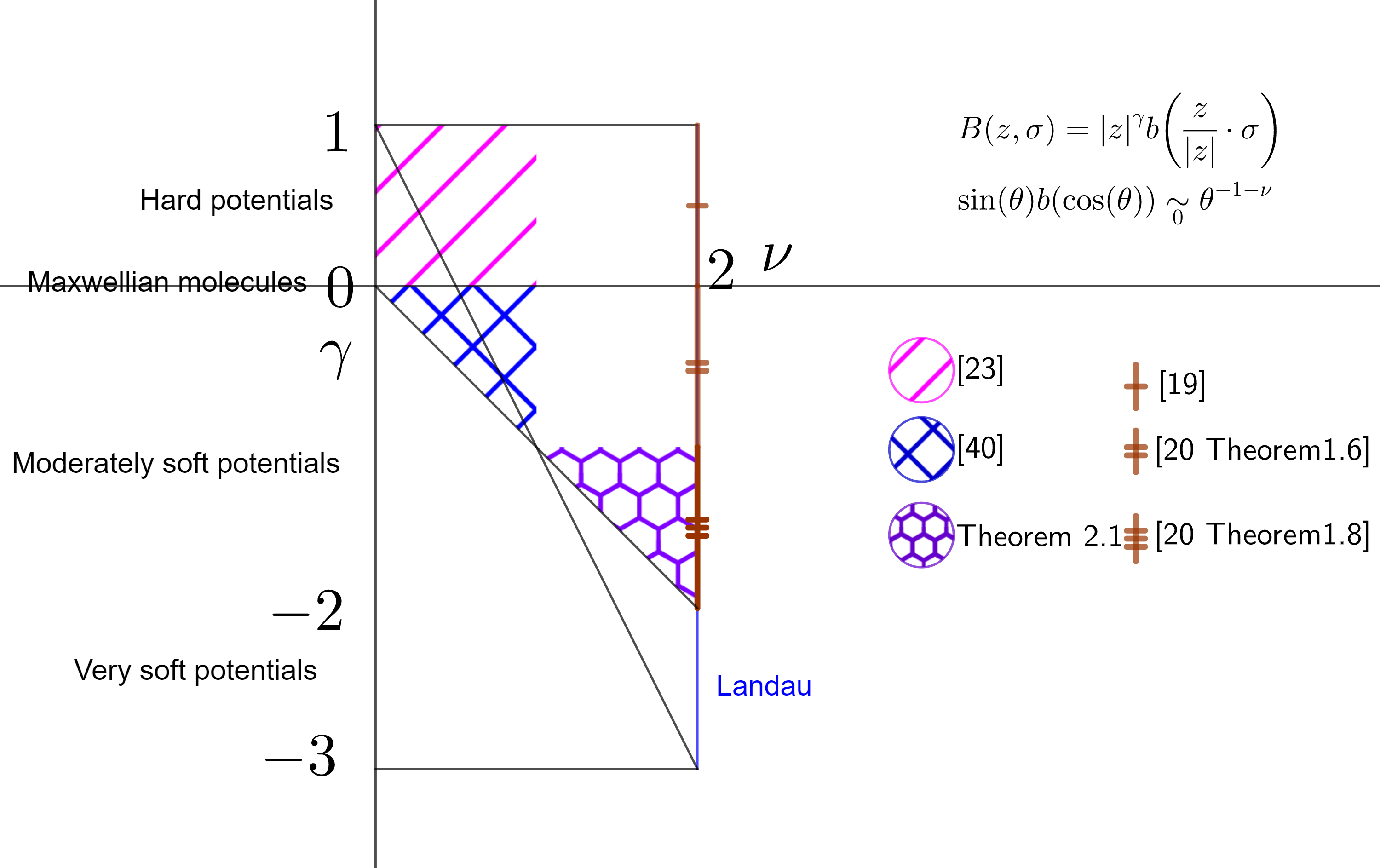}
	\caption{Non exhaustive topography of propagation of chaos results for equation \eqref{eq:Bol} in the rectangle $(\nu,\gamma)\in (0,2]\times[-3,1]$. The line $\nu=2$ represents the Landau equation. The diagonal corresponds to the inverse power law $\gamma=\frac{s-5}{s-1}, \ \nu=\frac{2}{s-1}$ for $s\in [2,\infty)$}.
	\label{fig1}
\end{figure}
We start this non exhaustive list with the seminal papers by Tanaka \cite{Tan} and Sznitman \cite{SznBol}, which concern respectfully Maxwellian molecules ($\gamma=0$ and $\int_0^\pi\theta \beta(\theta)d\theta<\int_0^\pi\beta(\theta)d\theta=+\infty$) and hard spheres ($\gamma=1$, and $b(x)=|x|$ in \eqref{eq:H1}). In the later, a convergence result from an interacting particles system to the corresponding limit equation, is qualitatively established thanks to a martingale method, which we are going to use in this paper. Both cases have been quantitatively treated in \cite{MM}, thanks to functional framework which relies on a semi-group approach. See \cite[Theorem 5.1]{MM} for the case of \textit{true Maxwellian molecules} (i.e. $\gamma=0$ and $\nu=\dem$), and \cite[Theorem 6.1]{MM} for the case of hard spheres (i.e. $\gamma=1$ and $b\equiv 1$ in \eqref{eq:H1}). A similar semi group approach is applied to the Landau equation for Maxwellian molecules (i.e. $\gamma=0$ and $\nu\rightarrow 2^{-}$) in \cite{Kle}. \newline
A probabilistic approach is developed in \cite{SN}, in the case of hard potentials $\gamma \in [0,1]$ and $\nu\in (0,1)$ (see \cite[Theorem 1.4]{SN}), which provides a rate of convergence in Wasserstein metric of the empirical measure associated to some interacting particles system (similar to \eqref{eq:BKnu>1} modulo the correction coefficient $\alpha^R_{\delta,\eta}$ and the mollification w.r.t. $\rho_{\e_N}$), toward the solution of the corresponding Boltzmann equation. A comparable coupling method is used in \cite{FGL}, for the Landau equation with hard potentials, with uniform in time rate in the special case of Maxwellian molecules (see \cite[Theorem 4]{FGL}). \newline
When one turns to the case of soft potentials, difficulties arise  from the singularity in the collision kernel, and it is one of the trending topic among the kinetic community to obtain some propagation of chaos result for singular interaction. For the Boltzmann equation,a quantitative result in Wasserstein metric is obtained in \cite{Xu} in the case of moderately soft potentials $\gamma \in (-1,0),\nu \in (0,1)$ and $\gamma+\nu>0$ (see \cite[Theorem 1.4]{Xu}, and it is to the best of the author's knowledge, the only propagation of chaos result for the Boltzmann equation with soft potentials. As for the Landau equation, the full range of moderately soft potentials $\gamma \in (-2,0)$ is treated in \cite{MN}. In the range $\gamma \in (-1,0)$, a quantitative result is obtained with a similar technique as \cite{Xu} (see \cite[Theorem 1.6]{MN}). For the full range of soft potentials $\gamma \in (-2,0)$, a qualitative convergence result is obtained thanks to an information-based approach (see \cite[Theorem 1.8]{MN}). The main result of this paper, consists in extending this approach to the Boltzmann equation, and is stated in the
\begin{theorem}
	\label{thm:main1}
	Let $\gamma \in (-2,-1)$ and $\nu\in (1,2)$ be such that $\gamma+\nu> 0$, and a collision kernel $B:\R^3\times \mathbb{S}^2\mapsto \R^+$ of the form \eqref{eq:H1}-\eqref{eq:bGraz} satisfying \textbf{(H)}. Let $p\in \lt( \frac{3}{3+\gamma},\frac{3}{3-\nu} \rt)$, $g_0\in L\ln L(\R^3)\cap \PP_k(\R^3)$ for some $k\geq |\gamma| p \frac{\frac{\nu}{3-\nu}}{\frac{3}{3-\nu}-p}$, and $(G_0^N)_{N\geq 2}\in\PP_{\textit{sym}}(\R^{3N})$ be a $g_0$-chaotic (in the sense of \cite[Definition 2.1]{Szn})  sequence satisfying
	\bq
	\label{eq:Ass}
	H_0:=\sup_{N\geq 1}\frac{1}{N}\int_{\R^{3N}} G_0^N(V)\ln (G_0^N(V)) dV <\infty, \ M_{0,k}:=\sup_{N\geq 1}\int |v_1|^k G_0^N(V)dV<\infty.
	\eq
	Let $T>0$ and for an explicit numerical constant $C>0$ let
	\[
	R >2+ \sqrt{\|g_0\|_{L^1_2}},  \ \eta<  C\lt(1-\frac{\|g_0\|_{L^1_2}}{4(R-2)^2}\rt), \ \delta <  \frac{CR^{-5}}{ \exp\lt(4\frac{C+2(H(g_0)+\|g_0\|_{L^1_2})}{\eta^2}  \rt)},
	\]
	and for each $N\geq 2$ consider $(\VV^N_t)_{t\in [0,T]}$ a $\R^{3N}$-valued process solution to \eqref{eq:BKnu>1}, for a $G_0^N$-distributed initial condition. \newline
	Then $\lt(\mei \delta_{\VV_t^{N,i}}\rt)_{t\in [0,T]}$ converges, as $N$ goes to infinity, weakly in law to the unique solution to the Boltzmann equation \eqref{eq:Bol} starting from $g_0$, $(g_t)_{t\in[0,T]}\in L^1(0,T;L^p(\R^3))\cap L^\infty (0,T;L^1_k(\R^3))$, .\newline
	In particular, the perturbation diffusion in the particles system \eqref{eq:BKnu>1} vanishes as $N$ goes to infinity.
\end{theorem}

The novelty of this result, is that to the best of the author's knowledge, it deals with a range of softness of potentials which was not covered before, for the Boltzmann equation. The two drawbacks are that it provides only a qualitative convergence, and it misses the physical cases of the inverse power law potentials.

\subsection{ Sketch of the proof }

The strategy we use to treat the propagation of chaos for equation with very singular coefficient is the one which was first established in \cite{MSN}. The authors treat the 2D Navier-Stokes equation in vortex formulation. Thanks to an entropy dissipation method, they obtain a bound on a key information quantity (namely the Fisher information) uniformly in the number of particles. This enables to deduce both the existence of a converging subsequence of the empirical measure associated to the interacting diffusions system, and the uniqueness of the limit point. Still in the case of mean-field equation, the same method has been applied to some sub-critical Keller-Segel equation in \cite{GQ}. Then the author extended this strategy to fractional diffusion in \cite{SS}.\newline
As for collisional dynamics, the Landau equation with moderately soft potentials ($\gamma \in (-2,0)$)  has been treated in\cite[Theorem 1.8]{MN}, and it is this strategy that we adapt to the Boltzmann equation. Seeing the Landau equation as the grazing collisions limit of the Boltzmann equation, the idea is thus to adapt the techniques used for Brownian-like diffusion to a L\'evy flight diffusion (as from \cite{GQ} to \cite{SS}).  \newline

In that purpose, the functional which plays an essential role, is the weighted normalized fractional Fisher information $\II^N_{\nu,\gamma}$ defined for $G^N\in \Ps$ as
	
	\bq
	\label{eq:Fishdef1}
\II^N_{\nu,\gamma}(G^N)=\frac{1}{N}\sum_{i=1}^N \int_{ \R^{d(N-1)}}\int_{\R^{d}\times \R^d} \frac{\lt(\lal v \ral^{\gamma/2}\sqrt{G^N(v_1,v_{i-1},v,v_{i+1},v_N)}-\lal v_* \ral^{\gamma/2}\sqrt{G^N(v_1,v_{i-1},v_*,v_{i+1},v_N)}\rt)^2}{|v_*-v|^{d+\nu}} dvdv_* dV^{N-1}_i,\\
	\eq
	(we define it with general dimension, as the results about this quantity, provided in this paper do not depend on the dimension). For $V^{N-1}_i\in \R^{d(N-1)}$ fixed we denote
	\[
	G^{N,N-1}_i(v)=G^N(v_1,\cdots,v_{i-1},v,v_{i+1},\cdots,v_N),
	\]
	the function defined on $\R^d$ with the variables $(v_1,\cdots,v_{i-1},v_{i+1},\cdots,v_N)$ freezed, and for $(\VV_1,\cdots,\VV_N)$ a $\R^{dN}$ random variable of law $G^N$, we denote
	\[
	g^i(\cdot|V_i^{N-1})=\LL(\VV_i|\VV_1,\cdots,\VV_{i-1},\VV_{i+1},\cdots,\VV_N), \ \text{and} \ G^{N-1}=\LL(\VV_1,\cdots,\VV_{i-1},\VV_{i+1},\cdots,\VV_N)
	\]
	the conditional law of the $i-th$ particle knowing the $N-1$ others, and the law of the $N-1$ others particles respectfully, such that there holds
	\[
	G^N(v_1,\cdots,v_{i-1},v,v_{i+1},\cdots,v_N)=g^i(v|V_i^{N-1})G^{N-1}(V_i^{N-1}). 
	\]
	With these notations, we may rewrite the fractional Fisher information both as
	\begin{align*}
\II^N_{\nu,\gamma}(G^N)&=\frac{1}{N}\sum_{i=1}^N \int_{ \R^{d(N-1)}}\lt|\lal \cdot \ral^{\gamma/2} \sqrt{G^{N,N-1}_i(\cdot)}\rt|^2_{H^{\nu/2}(\R^d)} dV^{N-1}_i\\
&=\frac{1}{N}\sum_{i=1}^N \int_{ \R^{d(N-1)}}\lt|\lal \cdot \ral^{\gamma/2} \sqrt{g^i(\cdot|V_i^{N-1})}\rt|^2_{H^{\nu/2}(\R^d)} G^{N-1}(V_i^{N-1})  dV^{N-1}_i.\\
	\end{align*}
To control this information, we have use of the theory of regularization by grazing collision (see for instance \cite{ADVW},\cite{Anoich}). It is in order to use these results as mere black boxes, that we need the perturbation diffusion in the particles system \eqref{eq:BKnu>1}, which aims to quantify whether or not the assumptions to apply this theory are satisfied. \newline
More precisely, for any $R>1$ and $\delta,\eta>0$ we say that $f\in \PP(\R^3)$ satisfies the condition $\mathbb{A}(R,\delta,\eta)$ when	
\bq
\label{eq:cond}
\begin{aligned}
\mathbb{A}(R,\delta,\eta) :& \ \int_{\B_ R}f(dv)\geq 4\eta, \ \text{and} \ \forall \mathbf{e}\in \Sd, \forall x\in \B_R \ \int_{P_\delta^{\mathbf{e},x}}  f(dv)\leq \eta\ \text{with} \ P_\delta^{\mathbf{e},x}= \{ y \in \B_R, \ |(y-x)\cdot \mathbf{e} | \leq \delta \}.
\end{aligned}
\eq
This condition is useful to state the
\begin{proposition}[Adapted from Proposition 2.1 of \cite{Anoich}]
	 	\label{prop:Alex}
	  	Assume that $\gamma+\nu>0$ and let $B$ be a collision kernel  of the form \eqref{eq:H1}-\eqref{eq:bGraz} satisfying \textbf{(H)}. Let $g\in \PP(\R^3)$ satisfying condition $\mathbb{A}(R,\delta,\eta)$ for some $R>1$ and $\delta,\eta>0$. Then there are constants $c,C>0$ depending on $R,\delta,\eta$ such that for any $f\in L^2(\R^3)$ there holds
	 	\[
	 	\int_{\R^6}\int_{\Sd}B(v-v_*,\sigma) (f(v')-f(v))^2 g(v_*)dvdv_*d\sigma\geq c \lt|\lal v\ral^{\gamma/2} f \rt|^2_{H^{\nu/2}(\R^3)}-C\|f\|^2_{L^2_{\gamma/2}(\R^3)}.
	 	\]
\end{proposition}
Usually, the assumptions of such results are expressed in terms of moment and entropy estimates in the literature. But since we aim to use it in a context where an atomic measure (or a slight moollification of it), plays the role of $g$ in the above proposition, we need the more geometric condition $\mathbb{A}(R,\delta,\eta)$ (as an atomic measure can not have a bounded entropy).\newline
Moreover due to the following implication (see Appendix \ref{App:B}) 
\bq
\label{eq:imp}
f\in \PP(\R^3) \ \text{does not satisfy} \ \mathbb{A}(R,\delta,\eta) \Longrightarrow \ \alpha^R_{\delta,\eta}(f)\geq 1,
\eq
when the particles configuration does not allow us to apply the theory regularization by grazing collisions, the additional perturbation diffusion provides the smoothing effect that we require. This enable us to state the 
\begin{theorem}
	\label{thm:Fish1}
	Let $T>0$. Assume that the collision kernel $B$ of the form \eqref{eq:H1}-\eqref{eq:bGraz} satisfies \textbf{(H)} for some $(\gamma,\nu)\in(-2,0)\times(0,2)$, with $\gamma+\nu>0$. Let $R>1,\delta,\eta\in (0,1)$, and for each $N\geq 2$, let $(G_t^N)_{t\in[0,T]}$ be the law of the solution to \eqref{eq:BKnu>1} with initial law $G_0^N\in\PP_{\textit{sym}}(\R^{3N})$, satisfying \eqref{eq:Ass}. \newline
	Then for any $t\in [0,T]$ and $N\geq 2$ ,it holds
	\begin{align}
	&\label{eq:mom}  \int_{\R^{3N}}|v_1|^kG_t^N(V)dV\leq e^{C_kT}\lt( 1+ M_{0,k} \rt), \\
	& \label{eq:Fish}  \int_0^T \II^N_{\nu,\gamma}(G_t^N)dt\leq C_{H_0,M_{0,k},T,\gamma,\nu,k,R,\delta,\eta}
	\end{align}
\end{theorem}
The key estimate \eqref{eq:Fish} is used for two purposes. First to control the sigularity in \eqref{eq:BKnu>1}, and establish the tightness of the empirical measures sequence, thanks to the
\begin{proposition}
	\label{prop:HLS}
	For any $d\geq 2$, $\lambda\in (0,\nu)$ and $r\in \lt(\frac{d}{d+\nu}\vee \frac{\nu+\lambda}{2\nu}\vee \frac{2d-(\nu-\lambda)}{2d} ,1\rt)$ there is $C_{\lambda,d,\nu,\gamma,r}>0$ such that for any $F^N\in \Ps$ it holds
	\[
	\sup_{i\neq j}\int_{\R^{dN}}|v_i-v_j|^{-\lambda}F^N(dV)\leq C_{\lambda,d,\nu,\gamma,r} \lt(\II^N_{\nu,\gamma}(F^N)+\int_{\R^{dN}} |v_1|^{-\gamma\frac{r}{1-r}} F^N(dV) +1  \rt).
	\]
\end{proposition}

This result can be seen as some Hardy-Litllewood-Sobolev inequality \cite[Theorem 4.3]{Lieb} in the case where the integrating measure is not tensorized. Since the most singular part of the trajectory of one particle, is the drift part in \eqref{eq:BKnu>1}, this bound is sufficient to deduce in Proposition \ref{prop:tig}, using some classical stochastic calculus tools, that there exists an accumulation point to the sequence $\lt(\mei \delta_{\VV_t^{N,i}}\rt)_{N\geq 2}$ in $\PP(\R^3)$. We then prove that the support of the law of this accumulation point is included in the set of probability measures with smooth density, thanks to the 

\begin{proposition}
	\label{thm:Fish2}
	The family of functionals $(\II^N_{\nu,\gamma})_{N\geq 1}$ defined in \eqref{eq:Fishdef1} are $\Gamma$-lower semi-continuous, i.e. if $(F_N)_{N\geq 1}\in\Ps$ converges to $\pi\in \PP(\PP(\R^d))$ in the sense that for any $j\geq 2$, $F_N^j \underset{N\rightarrow +\infty}{\rightharpoonup} \int_{\PP(\R^d)} \rho^{\otimes j} \pi(d\rho)$, where $F^j_N$ denotes the $j$-particles marginal of $F_N$, then
	\[
	\liminf_{N\geq 1}\II^N_{\nu,\gamma}(F_N)\geq \int_{\PP(\R^d)}\lt| \lal \cdot \ral^{\gamma/2} \sqrt{\rho}  \rt|_{H^{\nu/2}(\R^d)}^2 \pi(d\rho).
	\]
\end{proposition}
By Sobolev's embeddings, we can deduce some Lebesgue regularity on the limit point, thanks to the

\begin{lemma}
		\label{lem:LpHnu}
		Let $\gamma \in (-2,-1)$ and $\nu \in (1,2)$ such that $\gamma+\nu>0$. For any $p \in  \lt( \frac{3}{3+\gamma},\frac{3}{3-\nu} \rt)$ and $q>1$ there is a constant $C_{p,q,d,\nu}$ and $\theta_i>0$, $i=1,2,3$ with 
		$\theta_1+\theta_2+\theta_3=1$, such that for any $f\in \L^1(\R^d)$ there holds
		\[
		\|h\|_{L^p} \leq C_{p,q,d,\nu}\lt(\int_{\R^3} \lal v \ral^{-\gamma q(p-1/q)} h\rt)^{\theta_1}  \|\lal \cdot \ral^{\gamma} h \|^{\theta_2}_{L^1(\R^3)}   \lt|\lal \cdot \ral^{\gamma/2} \sqrt{h}\rt|^{2\theta_3}_{H^{\nu/2}(\R^3)}
		\]
	\end{lemma}
	\begin{proof}
		First, by Holder's inequality, we get for any $q>1$
		\begin{align*}
		\|h\|_{L^p}^p=&\int_{\R^3}h^p=\int_{\R^3}\lal v \ral^{\gamma (p-1/q)} h^{p-1/q} \lal v \ral^{-\gamma (p-1/q)} h^{1/q}\\
		&\leq \lt(\int_{\R^3} \lal v \ral^{-\gamma q(p-1/q)} h\rt)^{1/q} \lt( \int_{\R^3}(\lal v \ral^{\gamma} h)^{q'(p-1/q)}  \rt)^{1/q'}\\
		&=\lt(\int_{\R^3} \lal v \ral^{-\gamma q(p-1/q)} h\rt)^{1/q}\| \lal \cdot \ral^{\gamma} h \|_{L^{q'(p-1/q)} }^{p-1/q}.
		\end{align*}
		Using interpolation between Lebesgue spaces, and Sobolev's embedding  (see for instance \cite[Theorem 6.5]{DiNez}) we have that
		\begin{align*}
		\| \lal \cdot \ral^{\gamma} h \|_{L^{q'(p-1/q)} } &\leq \| \lal \cdot \ral^{\gamma} h \|^{1-\lt(1-\frac{1}{q'(p-1/q)}\rt)\frac{3-\nu}{3}}_{L^{1}}\| \lal \cdot \ral^{\gamma} h \|^{\lt(1-\frac{1}{q'(p-1/q)}\rt)\frac{3-\nu}{3}}_{L^{\frac{3}{3-\nu}} } \\
		&\leq  \| \lal \cdot \ral^{\gamma} h \|^{1-\lt(1-\frac{1}{q'(p-1/q)}\rt)\frac{3-\nu}{3}}_{L^{1}} \lt| \sqrt{\lal \cdot \ral^{\gamma} h } \rt|_{H^{\nu/2}}^{2\lt(1-\frac{1}{q'(p-1/q)}\rt)\frac{3-\nu}{3}}
		\end{align*}
\end{proof}

Moreover, we show in Proposition \ref{prop:mart}, that the exhibited accumulation point solves a martingale problem, as is classical in the context of qualitative propagation of chaos result (see for instance \cite{SznBol},\cite{FJ},\cite{Osa}). In fact, the solution to this martingale problem can be seen as a solution to some Boltzmann equation with an additional perturbation diffusion
 \bq
 \label{eq:BolP}
 \pa_t f_t(v)=\int_{\R^3\times\mathbb{S}^{2}}B(v-v_*,\sigma)\lt(f_t(v')f_t(v'_*)-f_t(v)f_t(v_*)\rt)dv_*d\sigma+\alpha^{R}_{\delta,\eta}(f_t)\Delta f_t,
 \eq
 for the initial condition $g_0$. Thanks to a result by Fournier and Guérin, we can show that there is at most one solution to this martingale problem, with the Lebesgue regularity that we have obtained. More precisely we have the strong-strong stability estimate in $W_2$ (Wasserstein 2) metric, in the
 \begin{theorem}[Corollary 1.5 of \cite{FGue}]
	\label{thm:Fgue}
	Let $B$ be a collision kernel of the form of \eqref{eq:H1}-\eqref{eq:bGraz} satisfying \textbf{H}. Let $(g_t)_{t\in [0,T]},(\tilde{g}_t)_{t\in [0,T]}$ be two weak solutions to the Boltzmann equation \eqref{eq:Bol}, such that $(g_t)_{t\in [0,T]},(\tilde{g}_t)_{t\in [0,T]}\in L^\infty([0,T],L^1_2(\R^3))\cap L^1([0,T],L^p(\R^3))$ for some $p>\frac{3}{3+\gamma}$. There exists $K>0$ depending on $B$ and $p$, such that, for any $t\geq 0$
	\[
	W_2(g_t,\tilde{g}_t)\leq 	W_2(g_0,\tilde{g}_0)e^{K_p\int_0^t \lt(1+\|g_s\|_{L^p(\R^3)}+\|\tilde{g}_s\|_{L^p(\R^3)}\rt)  }.
	\]
\end{theorem}
Thanks to a slight modification of the proof of this Theorem (since the nonlinear diffusion coefficient $\alpha^R_{\delta,\eta}$ is smooth, see Lemma \ref{lem:regCO}) we show that equation \eqref{eq:BolP} admits a unique solution lying in $L^1([0,T],L^p(\R^3))$, for the initial condition $g_0$.\newline 
Finally, consider $(g_t)_{t\in [0,T]}$ the unique solution to the Boltzmann equation, lying in $L^1([0,T],L^p(\R^3))$, starting from $g_0$, as given by Theorem \ref{thm:Fgue}. Choosing the parameters $R,\delta$ and $\eta$ in \eqref{eq:BKnu>1}, according to some bounds on the entropy and second order moment of $g_0$, we have by Lemma \ref{lem:regCO2}, that $\alpha_{\delta,\eta}^R(g_t)=0$ for any $t\in [0,T]$. Thus for this particular choice of parameters, $(g_t)_{t\in [0,T]}$ is also the unique solution to \eqref{eq:BolP}, and the result is proved.

\section{Proof of Theorem \ref{thm:Fish1}}

This section is devoted to the proof of the a priori bounds on the law of solutions to SDE \eqref{eq:BKnu>1}. In the first subsection we obtain the moment estimate \eqref{eq:mom}, mimicking the arguments of \cite[Proposition 4.1]{MN}. In the second we use some entropy dissipation method to prove the bounds \eqref{eq:Fish}.\newline
Before further considerations, we make the following observation. For any  functions $h$ and $g$ smooth enough, we have using Young's and Holder's inequalities	
\bq
\label{eq:|H1|}
\dem  \int_{\R^d}|h \nabla g |^2- 4 \int_{\R^d} |g \nabla h |^2 \leq |\nabla(gh)|^2 \leq 2 \int_{\R^d}|h \nabla g |^2+2\int_{\R^d} |g \nabla h |^2.
\eq
And similarly for any $\nu\in (0,2)$
\bq
\label{eq:Hnu|}
\begin{aligned}
\dem \int_{\R^{2d}}h^2(x) \frac{(g(x)-g(y))^2}{|x-y|^{d+\nu}}dxdy&-4\int_{\R^{2d}}g^2(x) \frac{(h(x)-h(y))^2}{|x-y|^{d+\nu}}dxdy \leq |h g|^2_{H^{\nu/2}}\\
&  \leq 2\int_{\R^{2d}}h^2(x) \frac{(g(x)-g(y))^2}{|x-y|^{d+\nu}}dxdy+ 2\int_{\R^{2d}}g^2(x) \frac{(h(x)-h(y))^2}{|x-y|^{d+\nu}}dxdy.
\end{aligned}
\eq
In particular if $g$ is bounded and Lipschitz
\[
\dem \int_{\R^{2d}}h^2(x) \frac{(g(x)-g(y))^2}{|x-y|^{d+\nu}}dxdy-C_{d,\nu}\|g\|^2_{W^{1,\infty}} \|h\|^2_{L^2} \leq |h g|^2_{H^{\nu/2}}  \leq 2\int_{\R^{2d}}h^2(x) \frac{(g(x)-g(y))^2}{|x-y|^{d+\nu}}dxdy+ 2 C_{d,\nu}\|g\|^2_{W^{1,\infty}}\|h\|^2_{L^2}.
\]

Finally, let us  make some comments about the well posedness of the system of SDE \eqref{eq:BKnu>1}. For fixed $N\geq 2$, the maps $v\in \R^d\mapsto K_\gamma*\rho_{\e_N}(v)\in \R^d$ and $V\in \R^{dN}\mapsto \alpha_{\delta,\eta}^R\lt( \frac{1}{N-1}\sum_{j\neq i} \delta_{v_j}*\rho_{\e_N} \rt)$, are bounded and Lipschitz, so difficulties may arise only from the integral w.r.t. the compensated Poisson random measures. More precisely, the difficulty comes from the fact that the Poisson random measures considered in \eqref{eq:BKnu>1} have an infinite measure for intensity w.r.t. the $z$ variable. This problem can be solved by a truncation argument (see for instance \cite[Proposition 4.2]{SN}), because the SDE consists then in a recursive equation. Now, using the same tightness-martingale formulation method that we are about to use to treat the limit $N\rightarrow \infty$, enables to pass to the limit in the truncation parameter for fixed $N$ (see for instance \cite[Remark 2.2]{SS}). This enables to conclude that  for each $N\geq 2$, and $\VV^N_0$ a $\R^{3N}$-valued random vector, there exists, on some suitable probability space, a solution to \eqref{eq:BKnu>1} starting from $\VV^N_0$. We do not give the specifics of the proof of this result in order to avoid redundancies with our main result, nor do we address the question of uniqueness, as it is not needed in order to state Theorem \ref{thm:main1}.

\subsection{Moments estimates}
\begin{lemma}
	\label{lem:mom}	
	Let $T>0$ and $((\VV_t^{N,i})_{t\geq \in [0,T]})_{i=1,\cdots,N}$ a solution to the particles system \eqref{eq:BKnu>1} starting from a $G_0^N$ distributed initial condition satisfying \eqref{eq:Ass}. The following hold true 
	\begin{itemize}
	\item[$(i)$] For any $k\geq 2$, there is a constant $C_k>0$ independent of $R,\delta,\eta$ and $N\geq 2$ such that there holds
	\[
	\sup_{t\in[0,T]}\E\lt[|\VV_t^{N,1}|^k\rt]\leq e^{C_kT}\lt( 1+  M_{0,k}  \rt).
	\]
	\item[$(ii)$] For any stopping time $\tau \in [0,T]$ there holds
	\[
	\E\lt[|\VV_\tau^{N,1}|^2\rt]\leq M_{2,0}+\beta_0\e_N\int_0^T\E\lt[\lt|K_\gamma\rt|*\rho_{\e_N}\lt(\VV_s^{N,1}-\VV_s^{N,2}\rt)\rt]ds+4T.
	\]
	
	\end{itemize}
\end{lemma}

\begin{proof}
	
	First, an application of Ito's rule yields for any $i=1,\cdots,N$, $k\geq 2$ and stopping time $t>0$, that   
	\bq
	\label{eq:ITOk}
	\begin{aligned}
	|\VV_t^{N,i}|^k=&|\VV_0^{N,i}|^k-k\int_0^t\frac{\beta_0}{N-1}\sum_{j\neq i}^N K_\gamma*\rho_{\e_N}(\VV_s^{N,i}-\VV_s^{N,j})\cdot |\VV_s^{N,i}|^{k-2}\VV_s^{N,i}ds\\
	&+\int_{0}^t \frac{1}{N-1}\sum_{j\neq i}^N \int_{\R^3} \int_{0}^{2\pi}\int_0^\infty  \tilde{\Delta}|\cdot|^k\lt(\VV_{s}^{N,i},\VV_{s}^{N,j}+w,z,\varphi\rt)  dzd\varphi \rho_{\e_N}(dw)ds  \\
	&+\int_{[0,t]\times \EE_i^N}\lt(\lt|\VV_{s_-}^{N,i}+c_{\gamma,\nu}\lt(\VV_{s_-}^{N,i},\VV_{s_-}^{N,j}+w,z,\varphi \rt)\rt|^k-\lt|\VV_{s_-}^{N,i}\rt|^k\rt)\bar{\MM}_N^{i}(ds,dz,d\varphi,dw,dj)\\
	&+\int_0^t  \lt( \alpha_{\delta,\eta}^R\lt( \frac{1}{N}\sum_{j\neq i}^N \delta_{\VV_u^{N,j}}*\rho_{\e_N}  \rt)  \rt)^2|\VV_s^{N,i}|^{k-2}ds+ \sqrt{2}\int_0^t   \alpha_{\delta,\eta}^R\lt( \frac{1}{N}\sum_{j \neq i}^N \delta_{\VV_u^{N,j}}*\rho_{\e_N}  \rt)  |\VV_s^{N,i}|^{k-2} \VV_s^{N,i}\cdot d\BB_s^{N,i}.
	\end{aligned}
	\eq
	
\textit{Proof of point $(i)$}\newline
	We take the expectation, average over $i=1,\cdots,N$, and denote 
	\[
	Q(t)=\mei\E\lt[|\VV_t^{N,i}|^k\rt]=\E\lt[|\VV_t^{N,1}|^k\rt],
	\]
	By symmetry. This yields	
	\begin{align*}
	Q(t)=&Q(0)-k\int_0^t\mei\E\lt[\frac{\beta_0}{N-1}\sum_{j\neq i}^N K_\gamma*\rho_{\e_N}(\VV_s^{N,i}-\VV_s^{N,j})\cdot |\VV_s^{N,i}|^{k-2}\VV_s^{N,i}\rt]ds\\
	&+\frac{1}{N(N-1)}\sum_{i\neq j} \int_{\R^3}  \E\lt[\int_{ [0,2\pi]\times\R^+}\tilde{\Delta}|\cdot|^k\lt(\VV_{s_-}^{N,i},\VV_{s_-}^{N,j}+w,z,\varphi\rt)  dzd\varphi \rt] \rho_{\e_N}(w)dw\\
	&+\int_0^t \mei \E\lt[   \lt( \alpha_{\delta,\eta}^R\lt( \frac{1}{N}\sum_{j\neq i}^N \delta_{\VV_s^{N,j}}*\rho_{\e_N}  \rt)  \rt)^2|\VV_s^{N,i}|^{k-2}\rt]ds\\
	&=:Q(0)+I_t^N+J_t^N+K_t^N.
	\end{align*}

	$\diamond$ Estimate of $I_t^N$, \newline
	
	Since the mollification kernel $\rho_\e$ is even, and the field $K_\gamma$ is odd, $K_\gamma*\rho_\e$ is also odd. Moreover $K_\gamma*\rho_\e(0)=0$. Hence by symmetry of the law of the $(\VV_t^{N,i})_{i=1,\cdots,N}$, we find that for any $(i,j)\in \{1,\cdots,N\}^2$
	\begin{align*}
	\E&\lt[ K_\gamma*\rho_{\e_N}(\VV_s^{N,i}-\VV_s^{N,j})\cdot |\VV_s^{N,i}|^{k-2}\VV_s^{N,i}\rt]=\dem \E\lt[ K_\gamma*\rho_{\e_N}(\VV_s^{N,i}-\VV_s^{N,j})\cdot \lt( |\VV_s^{N,i}|^{k-2}\VV_s^{N,i}-|\VV_s^{N,j}|^{k-2}\VV_s^{N,j}\rt)\rt]
	\end{align*}	
	Using that there is $C_k>0$ such that for any $(v,w)\in \R^3$ it holds,  $\lt|  |v|^{k-2}v-|w|^{k-2}w  \rt|\leq C_k|v-w|(1+|v|^{k-2}+|w|^{k-2})$ and that for any $\e\in (0,1)$, it holds $|K_\gamma*\rho_{\e}(v)|\leq C(1+|v|^{\gamma+1})$, we obtain that
	\begin{align*}
	&\lt|K_\gamma*\rho_{\e_N}(\VV_s^{N,i}-\VV_s^{N,j})\cdot \lt( |\VV_s^{N,i}|^{k-2}\VV_s^{N,i}-|\VV_s^{N,j}|^{k-2}\VV_s^{N,j}\rt)\rt| \\
	&\quad \quad  \leq C_k\lt(1+|\VV_s^{N,i}-\VV_s^{N,j}|^{\gamma+1}\rt)|\VV_s^{N,i}-\VV_s^{N,j}|(1+|\VV_s^{N,i}|^{k-2}+|\VV_s^{N,j}|^{k-2})\\
	& \quad \quad \leq C_{k,\gamma} \lt( 1+ |\VV_s^{N,i}|^k+|\VV_s^{N,j}|^k \rt),
	\end{align*}
	since $k>\gamma+2>0$. Finally we conclude with
	\[
	I_t^N\leq C_k{k,\gamma} \int_0^t(1+Q(s))ds.
	\]

	$\diamond$ Estimate of $J_t^N$, \newline
	
	Using point $(iii)$ of Lemma \ref{lem:use}, we obtain 
	
	\begin{align*}
	\int_{[0,2\pi]\times\R^+}\tilde{\Delta}\lt|\cdot\rt|^k\lt(\VV_{s}^{N,i},\VV_{s}^{N,j}+w,z,\varphi\rt)\leq C_k \lt(   \lt|\VV_{s}^{N,i}\rt|^{k-2}+\lt|\VV_{s}^{N,j}+w\rt|^{k-2}\rt)  \leq C_k \lt(   \lt|\VV_{s}^{N,i}\rt|^k+\lt|\VV_{s}^{N,j}\rt|^k+|w|^k+1\rt)
	\end{align*}
	
	Therefore, since there is $C_k>0$ such that for any $N\geq 2$ it holds $\int_{\R^3} |w|^k\rho_{\e(N)}(w)dw\leq C_k$, integrating the above inequality on $\R^3$ w.r.t. the density $\rho_{\e(N)}$ yields
	\[
	J_t^N\leq C_k \int_0^t(1+Q(s))ds.
	\]
	
	$\diamond$ Estimate of $K_t^N$, \newline
	
	Since for any $\mu \in \PP(\R^3)$ it holds $\alpha_{\delta,\eta}^R(\mu)\leq 2$, we easily get
	\begin{align*}
	K_t^N&\leq k\int_0^t \mei \E \lt[ \lt| \VV_s^{N,i}\rt|^{k-2}  \rt]ds\leq C_k \int_0^t \lt( Q(s)+1 \rt)ds
	\end{align*}

	Finally, gathering all these estimates yields

	\begin{align*}
	&Q(t)\leq Q(0)+ C_k \int_0^t \lt( Q(s)+1 \rt)ds,
	\end{align*}	
	and the conclusion follows by application of Gronwall's inequality. \newline
	
	\textit{Proof of point $(ii)$}\newline 
	First observe that 
	\[
	\int_{[0,2\pi]\times\R^+}\tilde{\Delta} |\cdot|^2(v,v_*,z\varphi)dzd\varphi=\int_{[0,2\pi]\times\R^+}|c_\gamma,\nu(v,v_*,z,\varphi)|^2dzd\varphi=\beta_0 |v-v_*|^{\gamma+2}
	\]
	Therefore coming back to \eqref{eq:ITOk} with $k=2$, up to some stopping time $\tau\leq T$ yields
	\begin{align*}
		|\VV_\tau^{N,i}|^2=&|\VV_0^{N,i}|^2-2\int_0^\tau\frac{\beta_0}{N-1}\sum_{j\neq i}^N \int_{\R^3}K_\gamma(\VV_s^{N,i}-\VV_s^{N,j}-w)\cdot \VV_s^{N,i}\rho_{\e_N}(w)dwds\\
		&+\int_{0}^\tau \frac{\beta_0}{N-1}\sum_{j\neq i}^N \int_{\R^3}  \lt|\VV_{s}^{N,i}-\VV_{s}^{N,j}-w\rt|^{\gamma+2} \rho_{\e_N}(dw)ds  \\
		&+\int_{[0,\tau]\times \EE_i^N}\lt(\lt|\VV_{s_-}^{N,i}+c_{\gamma,\nu}\lt(\VV_{s_-}^{N,i},\VV_{s_-}^{N,j}+w,z,\varphi \rt)\rt|^2-\lt|\VV_{s_-}^{N,i}\rt|^2\rt)\bar{\MM}_N^{i}(ds,dz,d\varphi,dw,dj)\\
		&+\int_0^\tau  \lt( \alpha_{\delta,\eta}^R\lt( \frac{1}{N}\sum_{j\neq i}^N \delta_{\VV_u^{N,j}}*\rho_{\e_N}  \rt)  \rt)^2ds+ \sqrt{2}\int_0^\tau   \alpha_{\delta,\eta}^R\lt( \frac{1}{N}\sum_{j \neq i}^N \delta_{\VV_u^{N,j}}*\rho_{\e_N}  \rt)  \VV_s^{N,i}\cdot d\BB_s^{N,i}.
	\end{align*}
	Taking the expectation and averaging over $i=1,\cdots,N$, we obtain
	\begin{align*}
	\mei\E\lt[|\VV_\tau^{N,i}|^2\rt]=&\mei\E\lt[|\VV_0^{N,i}|^2\rt]-\mei\E\lt[2\int_0^\tau\frac{\beta_0}{N-1}\sum_{j\neq i}^N \int_{\R^3}K_\gamma(\VV_s^{N,i}-\VV_s^{N,j}-w)\cdot \VV_s^{N,i}\rho_{\e_N}(w)dwds\rt]\\
	&+\mei\E\lt[\int_{0}^\tau \frac{\beta_0}{N-1}\sum_{j\neq i}^N \int_{\R^3}  \lt|\VV_{s}^{N,i}-\VV_{s}^{N,j}-w\rt|^{\gamma+2}\rho_{\e_N}(dw)ds\rt]  \\
	&+\mei\E\lt[\int_0^\tau  \lt( \alpha_{\delta,\eta}^R\lt( \frac{1}{N}\sum_{j\neq i}^N \delta_{\VV_u^{N,j}}*\rho_{\e_N}  \rt)  \rt)^2ds\rt].
	\end{align*}
	By symmetry of the law of the $(\VV_t^{N,i})_{i=1,\cdots,N}$ and $\rho_{\e_N}$
	\begin{align*}
		\mei\E\lt[|\VV_\tau^{N,i}|^2\rt]&=M_{2,0}-\mei\E\lt[\int_0^\tau\frac{\beta_0}{N-1}\sum_{j\neq i}^N \int_{\R^3}K_\gamma\lt(\VV_s^{N,i}-\VV_s^{N,j}-w\rt)\cdot w \rho_{\e_N}(w)dwds\rt]\\
	&+\mei\E\lt[\int_0^\tau  \lt( \alpha_{\delta,\eta}^R\lt( \frac{1}{N}\sum_{j\neq i}^N \delta_{\VV_u^{N,j}}*\rho_{\e_N}  \rt)  \rt)^2ds\rt].
	\end{align*}
	Hence using symmetry again and the fact that $ \alpha_{\delta,\eta}^R \leq 2$, the result is proved.

\end{proof}	

\subsection{Entropy dissipation estimates}

For each $N\geq 2$, consider $(G_t^N)_{t\in[0,T]}$ solution to \eqref{eq:ME} for the initial condition $G_0^N$. For $V\in \R^{3N}$, we denote
\[
\mu_i^{N-1,\e_N}=\frac{1}{N-1}\sum_{j\neq i} \delta_{v_j}*\rho_{\e_N}.
\]
Using the definition \eqref{eq:defG2} of $\tilde{\mathcal{A}}^N$, and dropping the $t$ in the notations for simplicity yields

	\begin{align*}
	\frac{d}{dt}\frac{1}{N}\int_{ \R^{3N} }G^N \ln G^N&=\frac{1}{N}\int_{ \R^{3N} }  \tilde{\mathcal{A}}_N G^N (1+\ln G^N ) \\
	&=\frac{1}{N}\int_{ \R^{3N} }  \NN^N G^N (1+\ln G^N )+\frac{1}{N}\int_{ \R^{3N} }  \PP^N G^N (1+\ln G^N )\\
	&:= -\DD^{N}(G^N)-\CC^{N}(G^N),\\
	\end{align*}

	$\bullet$ Estimate of $\DD^{N}$ \newline
	
	By definition of $\NN^N$ \eqref{eq:defG2} and using using for each $i=1,\cdots,N$, the unitary change of variables $(v_i,w,\sigma)\rightarrow \lt(v'_i,w',\frac{v_i-w}{|v_i-w|}\rt)$, we can rewrite

	\begin{align*}
	&\DD^{N}(G^N)=\\
	&\quad \frac{1}{N}\sum_{i=1}^N\int_{\R^{3N}}(1+\ln G^N(V))\int_{\Sd}  \int_{ \R^3}B(v_i-w,\sigma)  \lt(G^N(V)-G^N(V+(v'_i-v_i)\mathbf{a}_i)\rt)\mu_i^{N-1,\e_N}(w)dw d\sigma\\
	&=\frac{1}{N}\sum_{i=1}^N \int_{ \R^{3N} }\int_{\Sd} \int_{\R^3} B(v_i-w,\sigma) G^N(V)\ln \frac{G^N(V)}{ G^N(V+(v'_i-v_i)\mathbf{a}_i)} \mu_i^{N-1,\e(N)}(w)dwd\sigma dV\\
	&=\frac{1}{N}\sum_{i=1}^N \int_{ \R^{3N} } \int_{\Sd} \int_{\R^3} B(v_i-w,\sigma) \lt( G^N(V)\ln \frac{G^N(V)}{G^N(V+(v'_i-v_i)\mathbf{a}_i)} -G^N(V)+G^N(V+(v'_i-v_i)\mathbf{a}_i) \rt) \mu_i^{N-1,\e_N}(w)dwd\sigma dV\\
	&+ \frac{1}{N}\sum_{i=1}^N \int_{ \R^{3N} }\int_{\Sd} \int_{\R^3} B(v_i-w,\sigma) \lt( G^N(V)-G^N(V+(v'_i-v_i)\mathbf{a}_i) \rt) \mu_i^{N-1,\e_N}(w)dwd\sigma dV\\
	&:=S^N(G^N)+T^N(G^N)\\
	\end{align*}

	$\diamond$ Estimate of $T^N$ : \newline
		We rewrite $G^N$ in terms of its $N-1$ particles marginal $G^{N-1}$ and conditional law as
		\[
		G^N(v_1,\cdots,v_N)=g^i(v_i|V_i^{N-1})G^{N-1}(V_i^{N-1}),
		\]
		and we may rewrite
		\[
		T^N= \frac{1}{N}\sum_{i=1}^N \int_{ \R^{3(N-1)} }\int_{\R^3}\int_{\Sd} \int_{\R^3} B(v_i-w,\sigma) \lt( g^i(v_i|V_i^{N-1})-g^i(v'_i|V_i^{N-1})\rt)d\sigma dv_i \mu_i^{N-1,\e_N}(w)dw G^{N-1}_i(V_i^{N-1})dV_i^{N-1}.
		\]
		In view of the cancellation Lemma (see \cite[Lemma 1]{ADVW} and the remark below it), we obtain that
		\[
		\int_{\Sd} \int_{\R^3} B(v_i-w,\sigma) \lt( g^i(v_i|V_i^{N-1})-g^i(v'_i|V_i^{N-1})\rt) d\sigma dv_i = g^i(\cdot|V_i^{N-1})*S(w),
		\]
		with 
		\[
		S(z)=C_{\gamma,\nu } |z|^\gamma, \ C_{\gamma,\nu }=
		 |\mathbb{S}^1|\int_0^{\pi/2} \beta(\theta)\lt( |\cos(\theta/2)|^{-\gamma+3}-1 \rt)d\theta.
		\]
		Hence for any $w\in \R^3$ and $p\in \lt(\frac{3}{3+\gamma},\frac{3}{3-\nu}\rt)$ there holds
		\begin{align*}
		\lt|\int_{\Sd} \int_{\R^3} B(v_i-w,\sigma) \lt( g^i(v_i|V_i^{N-1})-g^i(v'_i|V_i^{N-1})\rt) d\sigma dv_i\rt|\quad = C_{\gamma,\nu}  \int_{\R^3} \lt(|v-w|^\gamma \mb_{|v-w|\leq 1}+\mb_{|v-w|>1}\rt) g^i(v|V_i^{N-1})dv&\\
		\leq C_{\gamma,\nu} \lt( \lt( \int_{\R^3}|v-w|^{\gamma p'} \mb_{|v-w|\leq 1} dv\rt)^{1/p'} \lt(\int_{\R^3} \lt(g^i(v|V_i^{N-1})\rt)^pdv\rt)^{1/p}+1\rt)&\\
		\leq C_{\gamma,p,\nu} \lt(	\|g^i(\cdot|V_i^{N-1})\|_{L^{p}}+1\rt)&,
		\end{align*}
		where we have used the fact that $g^i(\cdot|V_i^{N-1})\in \PP(\R^3)$ for any $V_i^{N-1}$ and Holder's inequality to pass from the second to the third line, and the fact that $\gamma p'>-3$ to pass to the fourth. \newline
	But by Lemma \ref{lem:LpHnu}, for each $V_i^{N-1}\in \R^{3(N-1)}$ there holds for any $\e>0$ and $q>1$ such that $k\geq |\gamma| q(p-1/q)$
			\begin{align*}
		\|g^i(\cdot|V_i^{N-1})\|_{L^{p}}& \leq C_{p,q,\nu}  \lt(\int_{\R^3} \lal v \ral^{-\gamma q(p-1/q)} g^i(\cdot|V_i^{N-1})\rt)^{\theta_1}  \|\lal \cdot \ral^{\gamma} g^i(\cdot|V_i^{N-1}) \|^{\theta_2}_{L^1(\R^3)}   \lt|\lal \cdot \ral^{\gamma/2} \sqrt{g^i(\cdot|V_i^{N-1})}\rt|^{2\theta_3}_{H^{\nu/2}(\R^3)}\\
		& \leq C_{p,q,\nu}  \e^{-\theta_3}\lt(\int_{\R^3} \lal v \ral^{-\gamma q(p-1/q)} g^i(\cdot|V_i^{N-1})\rt)^{\theta_1} \lt(\e \lt|\lal \cdot \ral^{\gamma/2} \sqrt{g^i(\cdot|V_i^{N-1})}\rt|^2_{H^{\nu/2}(\R^3)}\rt)^{\theta_3} \\
		& \leq C_{\gamma,\nu,p,q,\e} \int_{\R^3} \lal v \ral^{-\gamma q({p}-1/q)} g^i(v|V_i^{N-1})dv +\e \lt| \sqrt{ \lal \cdot \ral^{\gamma} g^i(\cdot|V_i^{N-1}) }  \rt|_{H^{\frac{\nu}{2}}(\R^3)} ^2.  
		\end{align*}
	The above inequalities combined yield, by integration w.r.t. $\mu_i^{N-1,\e_N}$ and $G^{N-1}$
	\bq
	\label{eq:TN}
	T^N(G^N)\geq -C_{k,\gamma,\nu,\e} \int_{\R^{3N}} \lal v_1 \ral^k G^N(dV)-\e \II_{\nu,\gamma}^N(G^N)  .
	\eq

		$\diamond$ Estimate of $S^{N}$ \newline 
		
	We use that for any $x,y\geq 0$ there holds $x\ln (x/y)-x+y\geq (\sqrt{x}-\sqrt{y})^2$, to obtain
	\bq
	\label{eq:DN}
	S^N(G^N)\geq \frac{1}{N}\sum_{i=1}^N \int_{ \R^{3N} } \int_{\Sd} \int_{\R^3} B(v_i-w,\sigma)  \lt(\sqrt{G^N(V+(v'_i-v_i)\mathbf{a}_i)}-\sqrt{G^N(V)}\rt)^2\mu_i^{N-1,\e_N}(w)dwd\sigma dV.
	\eq

$\bullet$ Estimate of $\CC^{N}$ \newline

For each $i=1,\cdots,N$, we integrate by parts w.r.t. $v_i$ and obtain, since $\mu_i^{N-1,\e(N)}$ does not depend on $v_i$, that

\begin{align*}
\CC^{N}(G^N)&=-\frac{1}{N}\int_{ \R^{3N} } (1+\ln G^N ) \lt(\alpha^R_{\delta,\eta}\lt(\mu_i^{N-1,\e_N}\rt) \rt)^2 \Delta_i G^N(V)dV\\
&= \frac{1}{N}\sum_{i=1}^N \int_{ \R^{3N} } \lt(\alpha^R_{\delta,\eta}\lt(\mu_i^{N-1,\e(N)}\rt) \rt)^2 \nabla_i (\ln G^N) \cdot \nabla_iG^N(V)dV,\\
&=\frac{4}{N}\sum_{i=1}^N \int_{ \R^{3(N-1)} } (\alpha^R_{\delta,\eta}(\mu_i^{N-1,\e_N}))^2 \int_{\R^3} \lt|\nabla \sqrt{G^{N,N-1}_i}\rt|^2dv_i dV^{N-1}_i\\
&=\frac{4}{N}\sum_{i=1}^N \int_{ \R^{3(N-1)} } (\alpha^R_{\delta,\eta}(\mu_i^{N-1,\e_N}))^2 \int_{\R^3} \lt|\nabla \lt( \lal \cdot \ral^{-\gamma/2}\lal \cdot \ral^{\gamma/2}\sqrt{G^{N,N-1}_i}\rt)\rt|^2dv_i dV^{N-1}_i.
\end{align*}
Using \eqref{eq:|H1|}, we obtain since $\gamma\leq 0$
\begin{align*}
\int_{\R^3} \lt|\nabla \lt( \lal \cdot \ral^{-\gamma/2}\lal  \cdot \ral^{\gamma/2}\sqrt{G^{N,N-1}_i}\rt)\rt|^2dv_i  \geq  \dem \int_{\R^3} \lt|\nabla \lt( \lal  \cdot \ral^{\gamma/2}\sqrt{G^{N,N-1}_i}\rt)\rt|^2dv_i -4\int_{\R^3}  \lt|\nabla \lt( \lal  \cdot \ral^{-\gamma/2} \rt)\rt|^2 \lal  \cdot \ral^{\gamma} G^{N,N-1}_idv_i.
\end{align*}
\[
\lt|\nabla \lt( \lal  \cdot \ral^{-\gamma/2} \rt)\rt|^2 \lal  \cdot \ral^{\gamma}=\frac{\gamma^2}{4}|v|^2\lal v \ral^{-4}\leq \frac{\gamma^2}{4}
\]
By the Fourier transform definition of the Sobolev's norm, and Parseval's identity, we get
\begin{align*}
\lt|\lal  \cdot \ral^{\gamma/2}\sqrt{G^{N,N-1}_i}\rt|^2_{H^{\nu/2}(\R^3)}&\leq \lt|\lal  \cdot \ral^{\gamma/2}\sqrt{G^{N,N-1}_i}\rt|^2_{H^{1}(\R^3)}+\lt\|\lal  \cdot \ral^{\gamma/2}\sqrt{G^{N,N-1}_i}\rt\|^2_{L^2(\R^3)}\\
&\leq \lt|\lal  \cdot \ral^{\gamma/2}\sqrt{G^{N,N-1}_i}\rt|^2_{H^{1}(\R^3)}+\lt\|G^{N,N-1}_i\rt\|_{L^1(\R^3)}.
\end{align*}

Intergating over $\R^{3(N-1)}$ yields
\bq
\label{eq:CN}
\begin{aligned}
	\CC^{N}(G^N)+(\gamma^2+1)&  \frac{1}{N}\sum_{i=1}^N \int_{ \R^{3(N-1)} }\lt\|G^{N,N-1}_i\rt\|_{L^1(\R^3)} dV^{N-1}_i= \CC^{N}(G^N)+(\gamma^2+1)\\
	&\geq  \frac{1}{N}\sum_{i=1}^N \int_{ \R^{3(N-1)} }\dem(\alpha^R_{\delta,\eta}(\mu_i^{N-1,\e_N}))^2 \lt|\lal  \cdot \ral^{\gamma/2}\sqrt{G^{N,N-1}_i}\rt|^2_{H^{\nu/2}(\R^3)} dV^{N-1}_i.
\end{aligned}
\eq

$\bullet$ Conclusion \newline

Summing \eqref{eq:CN}, \eqref{eq:DN} and \eqref{eq:TN} we obtain 
		\begin{align*}
		&\DD^{N}(G^N)+\CC^{N}(G^N)+C_{k,\gamma,\nu,\e}\int_{\R^{3N}}\lal v_1\ral ^kG^N(V)dV\geq\\
		&\frac{1}{N}\sum_{i=1}^N \int_{ \R^{3(N-1)} } \int_{\Sd} \int_{\R^6} B(v_i-w,\sigma)  \lt(\sqrt{G^{N,N-1}_i(v_i')}-\sqrt{G^{N,N-1}_i(v_i)}\rt)^2\mu_i^{N-1,\e_N}(w)dwd\sigma dV^{N-1}_i\\
		&+\frac{1}{N}\sum_{i=1}^N \int_{ \R^{3(N-1)} }    \lt|\lal \cdot \ral^{\gamma/2}  \sqrt{G^{N,N-1}_i}\rt|^2_{H^{\nu/2}(\R^3)}\lt(\alpha^R_{\delta,\eta}(\mu_i^{N-1,\e_N})\rt)^2 dV^{N-1}_i-\e\II_{\nu,\gamma}(G^N).
		\end{align*}
		Then for each $V^{N-1}_i\in \R^{3(N-1)}$, there are only two possibilities. Either $\mu_i^{N-1,\e_N}$ satisifies condition $\mathbb{A}(R,\delta,\eta)$. In which case we may apply Proposition \ref{prop:Alex} to obtain
		\begin{align*}
		S^N(G^N)&\geq \frac{1}{N}\sum_{i=1}^N \int_{ \R^{3N} } \int_{\Sd} \int_{\R^3} B(v_i-w,\sigma)  \lt(\sqrt{G^N(V+(v'_i-v_i)\mathbf{a}_i)}-\sqrt{G^N(V)}\rt)^2\mu_i^{N-1,\e_N}(w)dwd\sigma dV\\
		&\geq c_{R,\delta,\eta}\II_{\nu,\gamma}(G^N).
		\end{align*} Either it does not satisfy it, and accodring to \eqref{eq:imp}, there holds $\alpha^R_{\delta,\eta}(\mu_i^{N-1,\e_N})\geq 1$. In any case, we deduce
		\begin{align*}
		\DD^{N}(G^N)+\CC^{N}(G^N)&+C_{k,\gamma,\nu,\e}\int_{\R^{3N}}\lal v_1\ral ^kG^N(V)dV\\
		&\geq ((c_{R,\delta,\eta}\wedge 1)-\e)\II_{\nu,\gamma}^N(G^N).
		\end{align*}
		By integration over time, we obtain 
		\begin{align*}
		C_{R,\delta,\gamma,\nu,\eta,k}\int_0^T\II_{\nu,\gamma}^N(G^N_t)dt\leq &\int_0^T\DD^{N}(G^N_t)+\CC^{N}(G^N_t)+\int_{\R^3{3N}}\lal v_1 \ral^{k}G_t^N(dV)dt \\
		&\leq N^{-1}\lt(H(G_0^N)-H(G_T^N)\rt)+T\sup_{t\in [0,T]}\int_{\R^3{3N}}\lal v_1 \ral^{k}G_t^N(dV)\\
		&\leq H_0+C_{k,T}(1+M_{0,k}),
		\end{align*}
		which concludes the proof.
\section{Proof of Theorem \ref{thm:main1}}

Thanks to the estimate obtained in the previous section, we may prove our main result. 
\subsection{Tightness}

	We begin with a tightness result in the 

		\begin{proposition}
			\label{prop:tig} 
			For each $N\geq 2$, let $\VV^N_0$ be a $\R^{3N}$-valued, $G_0^N$-distributed random variable, and consider $(\VV_t^N)_{t\in [0,T]}$ solution to \eqref{eq:BKnu>1} starting from $\VV^N_0$.
			If the sequence $(G_0^N)_{N\geq 2}$ is chaotic and satisfies \eqref{eq:Ass}, then the sequence $\lt(\mei \delta_{\lt(\VV_t^{N,i}\rt)_{t\in [0,T]}}\rt)_{N\geq 2}$ is tight in $\PP( \mathbb{D}([0,T],\R^3))$.
		\end{proposition}

	\begin{proof}
		Since $\mathbb{D}([0,T],\R^3)$ is polish, in view of \cite[Proposition 2.2 point (ii)]{Szn}, it is enough to show the tightness of the process $\lt(\VV_t^{N,1}\rt)_{t\in [0,T]}$. By definition, it holds

			\begin{align*}
				\VV_t^{N,1}&=\VV_0^{N,1}+ \int_{[0,t]\times \EE_1^N}c_{\gamma,\nu}\lt( \VV^{N,1}_{s_-},\VV^{N,j}_{s_-}+w,z,\varphi \rt)\bar{\MM}^{1}_{N}(ds,d\varphi,dz,dw,dj)\\
				&+ \frac{\beta_0}{N-1}\sum_{j>1}^N \int_0^t (K_\gamma*\rho_{\e_N})(\VV_s^{N,1}-\VV_s^{N,j})ds + \int_0^t \sqrt{2} \alpha^R_{\delta,\eta}\lt(\frac{1}{N}\sum_{j\neq i} \delta_{\VV_s^{N,j}}*\rho_{\e_N}\rt) d\BB_s^{1}\\
				&:=\VV_0^{N,1}+\ZZ^{1,N}_t+\mathcal{D}^{1,N}_t+\mathcal{X}^{1,N}_t.
			\end{align*}
		
		It is then enough to show the tightness of each of the term of the sum.

		$\bullet$ Tightness of $\lt(\VV_0^{N,1}\rt)_{N\geq 1}$. \newline
	
	This is a simple consequence of the fact that the sequence of laws of the initial conditions is chaotic. \newline

	$\bullet$ Tightness of $\lt((\mathcal{D}^{1,N}_t)_{t\in [0,T]}\rt)_{N\geq 1}$. \newline

	First recall that for all $v\in \R^3$ we have $|K_\gamma*\rho_{\e_N}(v)|\leq C(1+ |v|^{\gamma+1})$. For any $0\leq s<t\leq T$ we have
	
		\begin{align*}
		\lt|\mathcal{D}^{1,N}_t-\mathcal{D}^{1,N}_s\rt|&=\lt| \frac{\beta_0}{N}\sum_{j>1}^N \int_s^t K_\gamma*\rho_{\e_N}(\VV_u^{N,1}-\VV_u^{N,j})du \rt|\leq \frac{C}{N}\sum_{j>1}^N \int_s^t\lt(1+\lt| \VV_u^{N,1}-\VV_u^{N,j} \rt|^{\gamma+1}\rt)du\\
		&\leq |t-s|^{1/q}\frac{C_q}{N}\sum_{j>1}^N \lt(\int_0^T 1+ \lt| \VV_u^{N,1}-\VV_u^{N,j} \rt|^{q'(\gamma+1)}du\rt)^{1/q'} \\
		&\leq  |t-s|^{1/q}\frac{C_q}{N}\sum_{j>1}^N \lt(1+\int_0^T\lt| \VV_u^{N,1}-\VV_u^{N,j} \rt|^{-q'(|\gamma|-1)}du\rt)  =: |t-s|^{q'}Z_{N,q}^T. 
	\end{align*}

	We choose $q>\frac{\nu}{\nu-|\gamma|+1}$, so that $q'(|\gamma|-1)<\nu$. Hence by symmetry, and since $G_t^N=\LL(\VV_t^N)\in \Ps$  it holds
	
		\begin{align*}
		\E\lt[Z_{N,q}^T\rt] &= \frac{C_q(N-1)}{N} \lt(1+\int_0^T\E\lt[ \lt| \VV_t^{N,1}-\VV_t^{N,2} \rt|^{q'(\gamma+1)}  \rt] dt \rt) =\frac{C_q(N-1)}{N} \lt(1+\int_0^T\int_{\R^{3N}}  \lt| v_1-v_2 \rt|^{q'(\gamma+1)} G_t^N(dV) dt\rt) . \end{align*}
	
Due to Proposition \ref{prop:HLS}, we find that for any $r\in \lt(\frac{d}{d+\nu}\vee \frac{\nu+q'(|\gamma|-1)}{2\nu}\vee \frac{2d-(\nu-q'(|\gamma|-1))}{2d} ,1\rt)$ there is a constant  
		\begin{align*}
	\int_{\R^{3N}}  \lt| v_1-v_2 \rt|^{-q'(|\gamma|-1)} G_t^N(dV)&\leq C_{q,\nu,\gamma,r}\lt( \II^N_{\nu,\gamma}(G_t^{N}) + \int_{\R^{3N}} |v_1|^{-\gamma \frac{r}{1-r} } G_t^N(dV) +  1 \rt),
	\end{align*}

and then using the bounds \eqref{eq:mom} and \eqref{eq:Fish}, we have
	\begin{align*}
	\E\lt[Z_{N,q}^T\rt] &\leq C_{q,\nu,\gamma,r} \lt(1+ \int_0^T \II^N_{\nu,\gamma}(G_t^{N})dt+\int_0^T \int_{\R^{3N}} |v|^k G_t^N(dV)dt \rt)\\
	&\leq C_{q,\nu,\gamma,r,T,M_{0,k},H_0,k,R,\delta,\eta}.
	\end{align*}

		Then for $A>0$ let us denote
	\[
	\mathcal{K}^A=\Bigl\{ h\in \CC([0,T],\R^d), \ h(0)=0, \ \sup_{0\leq s< t\leq T} \frac{|h(s)-h(t)|}{|s-t|^{p'}}\leq A \Bigr\},
	\]
	which is compact by Ascoli-Azerla's Theorem. Then using Markov's inequality yields
	\begin{align*}
		\sup_{N\geq 1}\mathbb{P}\lt( (\mathcal{D}^{1,N}_t)_{t\in [0,T]} \notin \mathcal{K}^A \rt)&\leq\sup_{N\geq 1}\mathbb{P}\lt( Z_{N,q}^T  \geq A\rt) \leq A^{-1}\sup_{N\geq 1}\E\lt[ Z_{N,q}^T\rt],
	\end{align*}
	and the sequence of laws of $\lt((\mathcal{D}^{1,N}_t)_{t\in [0,T]}\rt)_{N\geq 1}$ is tight (see Definition before Theorem 5.1 of \cite{Bil}), since for any $\e>0$, we can choose $A$ large enough such that it holds
	\[
	\sup_{N\geq 1}\mathbb{P}\lt( (\mathcal{D}^{1,N}_t)_{t\in [0,T]} \notin \mathcal{K}^A \rt)\leq\e.
	\]
		
	$\bullet$ Tightness of $(\mathcal{X}^{1,N}_t)_{t\in [0,T]}$. \newline
	
	By Burkholder-Davis-Gundy inequality (see \cite[Theorem 4.4.20]{Ebe}), for any $p>1$ and since for any $\mu \in \PP(\R^3)$ it holds $ \alpha^R_{\delta,\eta}(\mu)\leq 2$, for any $N\geq 2$, we have
	\begin{align*}
	\E\lt[\sup_{0\leq s<t\leq T}\lt|\mathcal{X}^{1,N}_t-\mathcal{X}^{1,N}_s\rt|^{2p}\rt]&=\E\lt[\sup_{0\leq s<t\leq T}\lt| \int_s^t \sqrt{2} \alpha^R_{\delta,\eta}\lt(\frac{1}{N}\sum_{j\neq i} \delta_{\VV_u^{N,j}}*\rho_{\e_N}\rt) d\BB^1_u \rt|^{2P}\rt]\\
	&\leq C_p\E\lt[ \lt(\int_s^t  \alpha^R_{\delta,\eta}\lt(\frac{1}{N}\sum_{j\neq i} \delta_{\VV_u^{N,j}}*\rho_{\e_N}\rt) du\rt)^p\rt]\leq C_p|t-s|^p, \\
	\end{align*}
	and the result follows by similar considerations as above.
	
\end{proof}

	$\bullet$ Tightness of $\lt((\ZZ^{1,N}_t)_{t\in [0,T]}\rt)_{N\geq 1}$. \newline
	Due to standard stochastic calculus property \cite[Theorem 4.2.3, point (2)]{Ebe}, we find that for any stopping time $\tau$, and real number $\zeta>0$, there holds
	\begin{align*}
	\E \lt[ \lt| 	\ZZ_{\tau+\zeta}^{1,N}-\ZZ_\tau^{1,N}\rt|^2\rt] &=\E\lt[ \lt|\int_{]\tau,\tau+\zeta]\times 	\EE_{1}^N}c_{\gamma,\nu}\lt( \VV^{N,1}_{u_-},\VV^{N,j}_{u_-}+w,z,\varphi \rt)\bar{\MM}^1_{N}(du,d\varphi,dz,dw,dj)\rt|^2 \rt]\\
	&=\E\lt[ \int_\tau^{\tau+\zeta} \int_{0}^\infty \int_0^{2\pi} \int_{\R^3} \frac{1}{N-1}\sum_{j\neq i}^N \lt|c_{\gamma,\nu}\lt( \VV^{N,1}_{u},\VV^{N,j}_{u}+w,z,\varphi \rt)\rt|^2 \rho_{\e_N}(w)dwd\varphi dz du \rt]\\
	&=\int_0^\zeta\E\lt[  \int_{\R^3} \frac{\beta_0}{N-1}\sum_{j\neq i}^N \lt| \VV^{N,1}_{\tau+u}-\VV^{N,j}_{\tau+u}-w\rt|^{\gamma+2} \rho_{\e_N}(w)dw \rt]du\\
	&\leq \beta_0 \int_0^\zeta C_\gamma \lt(1+\E\lt[ \lt| \VV^{N,1}_{\tau+u}\rt|^2 \rt]\rt)du,
	\end{align*}
	where we have used exchangeability, the fact that $|x|^{\gamma+2}\leq C_{\gamma}(1+|x|^2)$ and that $\int_{\R^3} |w|^2\rho_{\e_N}(w)dw\leq 1$, to pass to the last line.\newline 
	Using point $(ii)$ of Lemma \ref{lem:mom}, and since $\tau+\zeta$ is also a stopping time, we obtain 
	\begin{align*}
	\E \lt[ \lt|\ZZ_{\tau+\zeta}^{1,N}-\ZZ_\tau^{1,N}\rt|^2\rt]& \leq C_{\gamma,\beta_0}\zeta \lt(  1+ M_{2,0}+\beta_0\e_N\int_0^T\E\lt[\lt|K_\gamma\rt|*\rho_{\e_N}\lt(\VV_s^{N,1}-\VV_s^{N,2}\rt)\rt]ds+4T \rt)\\
	&\leq \zeta C_{\gamma,\beta_0,M_{2,0},T}\lt(1+\e_N\int_0^T\lt(\II_{\nu,\gamma}(G_s^N)+\int_{\R^{3N}} |v_1|^k G_s^N(dV)\rt)ds \rt)\leq \zeta C_{H_0,M_{0,k},T,\gamma,\nu,k,R,\delta,\eta}.
	\end{align*}
	Hence, by Markov's inequality, it holds
	\begin{align*}
	\mathbb{P}\lt( \lt|\ZZ^{1,N}_{\tau+\zeta}-\ZZ^{2,N}_{\tau}\rt| \geq \eta \rt)&\leq \eta^{-2}\E\lt[\lt|\ZZ^{1,N}_{\tau+\zeta}-\ZZ^{1,N}_{\tau}\rt|^2 \rt]\\
	& \leq  \eta^{-2}\zeta C_{E_0,\beta_0,\gamma}.
	\end{align*}
	In view of Aldous criterion \cite[Theorem 16.10]{Bil}, the sequence of \textit{càdlàg} processes $\lt(\lt(\ZZ_t^{1,N}\rt)_{t\in [0,T]}\rt)_{N\geq 1}$ is tight. Indeed, for any $\e,\eta>0$ we can find some $\zeta_0:= \eta^{2}C_{E_0,\beta_0,\gamma}^{-1}\e >0$ such that for any $\zeta \leq \zeta_0$ there holds
	\[
	\sup_{N\geq 1}\mathbb{P}\lt( \lt|\ZZ^{1,N}_{\tau+\zeta}-\ZZ^{2,N}_{\tau}\rt| \geq \eta \rt)\leq \e,
	\]
	for any stopping time $\tau$ such that $\tau+\zeta_0\leq T$.

\subsection{Martingale problem}

We now show that the accumulation point obtained in the previous section is unique. In that purpose we define the set
\bq
\label{eq:S}
\mathcal{S}:=\Bigl\{  Q\in \PP( \mathbb{D}([0,T],\R^3) ), \ Q \ \text{satisfies $(a)$, $(b)$ and $(c)$ }  \Bigr\},
\eq
where the conditions $(a),(b),(c)$ are defined as
\bq
\label{eq:mart}
\left\{ \begin{array}{ll}
	\displaystyle (a) \quad \mathbf{e}_0\#\mathcal{Q}=g_0,&\\[2mm]
	\displaystyle(b)\quad  \mathcal{Q}_t:=\mathbf{e}_t\#\mathcal{Q}, (\mathcal{Q}_t)_{t\in[0,T]} \quad \text{satisfies}, \int_0^T |\lal \cdot \ral^{\gamma/2}  \sqrt{Q_t}|^2_{H^{\nu/2}(\R^3)}+\int_{\R^3}|v|^kQ_t(dv) dt<\infty &\\[2mm]
	\displaystyle (c) \quad \forall 0<t_1<\cdots<t_m<s<t\leq T, \phi_1,\cdots,\phi_m\in C_b(\R^2), \phi\in C^2_b(\R^2), \mbox{it holds}&\\[2mm]
	\displaystyle \quad \quad \MM(Q)= \CC(Q)	  &\\.
		\MM(Q):=	\int_{\mathbb{D}([0,T],\R^3)^2} Q(d\varrho)Q(d\tilde{\varrho}) \prod_{k=1}^m\phi_k(\varrho_{t_k})&\\ \Bigl(\phi(\varrho_t)-\phi(\varrho_s)-\int_s^t\beta_0K_\gamma(\varrho_u-\tilde{\varrho}_u) \cdot \nabla \phi(\varrho_u)du-\int_s^t\int_{ [0,2\pi]\times \R^+ }\tilde{\Delta}\phi(\varrho_u,\tilde{\varrho}_u,\varphi,z)dzd\varphi  \,du \Bigr)&\\
	\CC(Q):=\int  Q(d\varrho) \prod_{k=1}^m\phi_k(\varrho_{t_k})\int_s^t\lt(\alpha^R_{\delta,\eta}(\mathcal{Q}_u)\rt)^2 \Delta \phi (\varrho_u) \,du . 
\end{array} \right.
\eq

We show that this accumulation point almost surely belongs to $\mathcal{S}$ in the

\begin{proposition}
	\label{prop:mart}
	For each $N\geq 2$, let $\VV^N_0$ be a $G_0^N$-distributed random variable, and consider $\lt((\VV_t^N)_{t\in [0,T]}\rt)_{N\geq 2}$ solution to \eqref{eq:BKnu>1}. Assume that $(G_0^N)_{N\geq 2}$ is $g_0$-chaotic, satisfies \eqref{eq:Ass} and that there is a subsequence of $\lt(\mei \delta_{\lt(\VV_t^{N,i}\rt)_{t\in [0,T]}}\rt)_{N\geq 2}$ converging in law to some $f\in \PP( \mathbb{D}([0,T],\R^3))$. Then $f$ almost surely belongs to $\mathcal{S}$.
\end{proposition}

\begin{proof}
We prove successively that each of the condition is fulfilled. \newline  

$\bullet$ $f$ satisfies $(a)$\newline 

It is a simple consequence form the fact that the law of the initial condition to the particle system \eqref{eq:BKnu>1} are $g_0$ chaotic.\newline

$\bullet$ $f$ satisfies $(b)$\newline

For any $t\in [0,T]$ we denote $\pi_t=\LL(f_t)$, and for any $k\geq 2$, $\pi_t^k=\int_{\PP(\R^3)}  \rho^{\otimes k}. \pi_t(d\rho) $. Recall that $G^N_t=\LL(\VV_t^N)\in \Ps$. Since $\lt(\mei \delta_{\VV_t^{N,i}}\rt)_{N\geq 2}$ converges (up to a subsequence) in law to $f_t$, for any $k\geq 2$, $(G^{N,k}_t)_{N\geq 2}$ ($k$-particles marginal of $G^N_t$) converges weakly to $\pi_t^k$. Using Proposition \ref{thm:Fish2}, Fatou Lemma and bound \ref{eq:Fish} we obtain
\begin{align*}
	\E\lt[  \int_0^T | \lal \cdot \ral^{\gamma/2} \sqrt{f_t}|_{H^{\nu/2}(\R^3)}^2 dt  \rt]&= \int_0^T \int_{\PP(\R^3)}  | \lal \cdot \ral^{\gamma/2}  \sqrt{\rho}|_{H^{\nu/2}(\R^3)}^2 \pi_t(d\rho) dt\\
	&\leq \int_0^T \liminf_{N} \II_{\nu,\gamma}^N(G^N_t) dt\leq \liminf_{N} \int_0^T  \II_{\nu,\gamma}^N(G^N_t) dt < +\infty,
\end{align*}	
and therefore
\[
\int_0^T | \lal \cdot \ral^{\gamma/2} \sqrt{f_t}|_{H^{\nu/2}(\R^3)}^2 dt < +\infty, \quad \text{a.s.}
\]
A similar argument applies to prove the almost surely bound on the $k$ moment.\newline

$\bullet$ $f$ satisfies $(c)$\newline

For any $i=1,\cdots,N$ Ito's rule yields for any test function $\phi\in\CC^2(\R^3)$
\begin{align*}
&\phi(\VV^{N,i}_t)=\phi(\VV^{N,i}_0)+ \int_0^t \frac{\beta_0}{N-1}\sum_{j\neq i} K_\gamma*\rho_{\e_N}(\VV_{s}^{N,i}-\VV_{s}^{N,j})\cdot \nabla \phi(\VV_{s}^{N,i}) + \lt(\alpha^R_{\delta,\eta}\lt( \frac{1}{N-1}\sum_{j\neq i}^N \delta_{\VV_s^{N,j}}*\rho_{\e_N} \rt)\rt)^2\Delta \phi(\VV_s^{N,i})ds \\
& +\frac{1}{N}\sum_{j\neq i}\int_0^t \int_0^{2\pi}d\varphi \int_0^\infty dz \int_{\R^3}  \tilde{\Delta} \phi (\VV_{s_-}^{N,i},\VV_{s_-}^{N,j}+w,\varphi,z) \rho_{\e_N}(dw)ds+\int_0^t \sqrt{2}\alpha^R_{\delta,\eta}\lt(\frac{1}{N-1}\sum_{j\neq i}^N \delta_{\VV_s^{N,j}}*\rho_{\e_N} \rt)\nabla\phi(\VV_s^{N,i})\cdot d\BB_s^{i}\\
&+\int_{[0,t]\times\EE_i^N }\lt( \phi(\VV_{s_-}^{N,i}+c_{\gamma,\nu}(\VV_{s_-}^{N,i},\VV_{s_-}^{N,j}+w,\varphi,z))-\phi(\VV_{s_-}^{N,i})\rt)\bar{\MM}^{i}_N(ds,dz,d\varphi,dw,dj).
\end{align*}
Then we define the processes $(O^{N,i}_t)_{t\in [0,T]}$ as
\begin{align*}
\mathcal{O}^{N,i}_t:=&\phi(\VV^{N,i}_t)-\phi(\VV^{N,i}_0)-\int_0^t \frac{\beta_0}{N-1}\sum_{j\neq i} K_\gamma*\rho_{\e_N}(\VV_{s}^{N,i}-\VV_{s}^{N,j})\cdot \nabla \phi(\VV_{s}^{N,i})  ds \\
&-\int_0^t \frac{1}{N}\sum_{j\neq i} \int_{\R^3} \int_0^{2\pi} \int_0^\infty \tilde{\Delta} \phi(\VV_s^{N,i},\VV_s^{N,j}+w,z,\varphi) d\varphi dz \rho_{\e_N}(dw)  ds
\end{align*}
In view of the above Ito's expansion, we also have
\begin{align*}
\mathcal{O}^{N,i}_t=&\int_0^t \lt(\alpha^R_{\delta,\eta}\lt( \frac{1}{N-1}\sum_{j\neq i}^N \delta_{\VV_s^{N,j}}*\rho_{\e_N} \rt)\rt)^2\Delta \phi(\VV_s^{N,i})ds+\int_0^t \sqrt{2}\alpha^R_{\delta,\eta}\lt(\frac{1}{N-1}\sum_{j\neq i}^N \delta_{\VV_s^{N,j}}*\rho_{\e_N} \rt)\nabla\phi(\VV_s^{N,i})\cdot d\BB_s^{i}\\
&+\int_{[0,t]\times\EE_i^N}\lt( \phi(\VV_{s_-}^{N,i}+c_{\gamma,\nu}(\VV_{s_-}^{N,i},\VV_{s_-}^{N,j}+w,\varphi,z))-\phi(\VV_{s_-}^{N,i})\rt)\bar{\MM}^{i}_N(ds,dz,d\varphi,dw,dj).\\
\end{align*}
We also define
\begin{align*}
\KK^N:=&\mei \prod_{k=1}^n \phi_{k}(\VV^{N,i}_{t_k}) \Bigl(  \mathcal{O}^{N,i}_t-\mathcal{O}^{N,i}_s \Bigr)\\
&=\mei \prod_{k=1}^n \phi_{k}(\VV^{N,i}_{t_k})\int_s^t \lt(\alpha^R_{\delta,\eta}\lt( \frac{1}{N-1}\sum_{j\neq i}^N \delta_{\VV_u^{N,j}}*\rho_{\e_N} \rt)\rt)^2\Delta \phi(\VV_u^{N,i})du\\
&+\mei \prod_{k=1}^n \phi_{k}(\VV^{N,i}_{t_k}) \Bigl(\int_s^t \sqrt{2}\alpha^R_{\delta,\eta}\lt(\frac{1}{N-1}\sum_{j\neq i}^N \delta_{\VV_u^{N,j}}*\rho_{\e_N} \rt)\nabla\phi(\VV_u^{N,i})\cdot d\BB_u^{i} \\
&+\int_{[s,t]\times\EE_i^N}\lt( \phi(\VV_{u_-}^{N,i}+c_{\gamma,\nu}(\VV_{u_-}^{N,i},\VV_{u_-}^{N,j}+w,\varphi,z))-\phi(\VV_{u_-}^{N,i})\rt)\bar{\MM}^{i}_N(du,dz,d\varphi,dw,dj)\Bigr).
\end{align*}
Finally, observe that by definition, it holds
\begin{align*}
	&\MM(\mu^N)=\\
	&\mei \prod_{k=1}^n \phi_{k}(\VV^{N,i}_{t_k}) \Bigl(\phi(\VV^{N,i}_t)-\phi(\VV^{N,i}_s)-\mej \int_s^t \beta_0K_\gamma(\VV_u^{N,i}-\VV_u^{N,j})\cdot \nabla \phi(\VV_u^{N,i})+ \int_0^\infty\int_0^{2\pi}\tilde{\Delta}\phi (\VV_u^{N,i},\VV_u^{N,j},\varphi,z))dz d\varphi du \Bigr)\\
\end{align*}	
For any $\e>0$ we define $K_\gamma^\e$ smooth, bounded, such that $K_\gamma^\e(x)=K_\gamma(x)$ for any $|x|\geq \e$ and $|K_\gamma^\e(x)|\leq |K_\gamma(x)|$ for any $|x|\leq \e$,  and 
\begin{align*}
	&\MM_\e(Q):=\\
	&	\int_{\mathbb{D}([0,T],\R^3)^2} Q(d\varrho)Q(d\tilde{\varrho}) \prod_{k=1}^N\phi_k(\varrho_{t_k}) \Bigl(\phi(\varrho_t)-\phi(\varrho_s)-\int_s^t\beta_0K_\gamma^\e(\varrho_u-\tilde{\varrho}_u) \cdot \nabla \phi(\varrho_u)+\int_{ [0,2\pi]\times \R^+ }dzd\varphi\tilde{\Delta}\phi(\varrho_u,\tilde{\varrho}_u,\varphi,z)  \,du  \Bigr).
\end{align*} 
With this definition, $\MM_\e$ is a continuous function on $\PP(\mathbb{D}([0,T],\R^3))$, in view  of Lemma \ref{lem:reg1}. Next introduce the following decomposition
\begin{align*}
\E\lt[\lt|\MM(f)-\CC(f)\rt|\rt]&\leq \E\lt[\lt|\MM(f)-\MM_\e(f)\rt|\rt]+\E\lt[\lt|\MM_\e(f)-\MM_\e(\mu^{N})\rt|\rt]+\E\lt[\lt|\MM_\e(\mu^{N})-\MM(\mu^{N})\rt|\rt]+\E\lt[\lt|\MM(\mu^{N})-\KK^N\rt|\rt]\\
&+\E\lt[\lt|\KK^N-\CC(f)\rt|\rt]\\
&=:  \II_1+\II_2+\II_3+\II_4+\II_5.
\end{align*}

$\Box$ Estimate of $\II_1,\II_3$\newline

We first observe that by definition of $K_\gamma^\e$ we have that
\[
\lt|K_\gamma^\e(x)-K_\gamma(x) \rt| \leq 2|x|^{\gamma+1}\mb_{|x|\leq \e},
\] 
so that for any $Q\in \PP\lt(  \DD([0,T];\R^3)\rt)$  it holds

\begin{align*}
\lt|\MM(Q)-\MM_\e(Q)\rt|&=\lt| \int_{\mathbb{D}([0,T],\R^3)^2} Q(d\varrho)Q(d\tilde{\varrho}) \prod_{k=1}^n\phi_k(\varrho_{t_k}) \Bigl(-\int_s^t\beta_0(K_\gamma^\e-K_\gamma)(\varrho_u-\tilde{\varrho}_u) \cdot \nabla \phi(\varrho_u)du\Bigr)   \rt|\\
&\leq C_{\phi,n} \int_0^T \int_{\R^3\times\R^3}   |v-v_*|^{\gamma+1}\mb_{|v-v_*|\leq \e}  Q_u(dv)Q_u(dv_*) du \\
&\leq \e^p C_{\phi,n}  \int_0^T \int_{\R^3\times\R^3}   |v-v_*|^{\gamma+1-p} Q_u(dv)Q_u(dv_*) du,
\end{align*}
for any $p>0$. Hence we have
\begin{align*}
	&\II_1\leq \e^p C_{\phi,n}  \E \lt[\int_0^T \int_{\R^3\times\R^3}   |v-v_*|^{\gamma+1-p} f_u(dv)f_u(dv_*) du \rt]\\
	&\II_3 \leq \e^p  C_{\phi,n} \meij  \int_0^T  \E \lt[  |\VV_u^{N,i}-\VV_u^{N,j}|^{\gamma+1-p}\rt] du.
\end{align*}	
Choosing $p\in (0,1+\nu+\gamma)$, using Proposition \ref{prop:HLS} and point $(i)$ of Lemma \ref{lem:aff} we obtain
\begin{align*}
\II_1&\leq \e^pC_{\phi,n,p,\gamma,\nu} \lt(1+\E \lt[\int_0^T \lt|\lal \cdot \ral^{\gamma/2} \sqrt{f_u}\rt|^2+\int_{\R^3} |v|^k f_u(dv) du \rt]\rt)\\
& \leq \e^p C_{\phi,n,p,\gamma,\nu,H_0,T,M_{0,k},R,\delta,\eta,k}
\end{align*}
We obtain a similar estimate for $\II_3$, using the same argument as in the proof of Proposition \ref{prop:tig}.\newline

$\Box$ Estimate of $\II_2$\newline

For any $\e>0$, $\II_2$ converges to $0$ as $N$ goes to infinity, since $\MM_\e$ is a continuous function on $\PP(\mathbb{D}([0,T],\R^3))$, and $\mu^N$ converges in law to $f$.\newline

$\Box$ Estimate of $\II_4$\newline

Observe that

\begin{align*}
\MM(\mu^N)&-\KK^N=\mei \prod_{k=1}^n \phi_{k}(\VV^{N,i}_{t_k}) \Bigl(\mej \int_s^t \beta_0\lt(K_\gamma-K_\gamma*\rho_{\e_N}\rt)(\VV_u^{N,i}-\VV_u^{N,j})\cdot \nabla \phi(\VV_u^{N,i})du \\
&+\int_s^t \int_{\R^3} \int_0^{2\pi} \int_0^\infty \lt(\tilde{\Delta}\phi (\VV_u^{N,i},\VV_u^{N,j},\varphi,z)-\tilde{\Delta}\phi (\VV_u^{N,i},\VV_u^{N,j}+w,\varphi,z)\rt) \rho_{\e_N}(dw) dzd\varphi du \Bigr).
\end{align*}
Since for any $v,v_*\in \R^3$, we have $$\lim_{N\rightarrow \infty}  K_\gamma-K_\gamma*\rho_{\e_N}(v-v_*)=0,$$ and for any $z,\varphi\in \R^*\times [0,2\pi]$ $$\lim_{N\rightarrow \infty}  \int_{\R^3} \lt(\tilde{\Delta}\phi (v,v_*,\varphi,z)-\tilde{\Delta}\phi (v,v_*+w,\varphi,z)\rt) \rho_{\e_N}(dw)=0,$$
using Proposition \ref{prop:HLS} and bound \eqref{eq:Fish}, we obtain that $\II_4=\E\lt[\lt|\MM(\mu^N)-\KK^N\rt|\rt]$ goes to zero as $N$ goes to infinity by the Lebesgue dominated convergence Theorem.\newline

$\Box$ Estimate of $\II_5$\newline

Denoting for $s\in [0,T]$
\[
\mu_s^N=\frac{1}{N}\sum_{j=1}^N \delta_{\VV_s^{N,j}}, \ \mu_s^{N-1,\e_N}= \frac{1}{N-1}\sum_{j\neq i}^N \delta_{\VV_s^{N,j}}*\rho_{\e_N},
\]
we rewrite
\begin{align*}
	\KK^N-\CC(f)=&\mei \prod_{k=1}^n \phi_{k}(\VV^{N,i}_{t_k})\int_s^t \lt(\lt(\alpha_{\delta,\eta}^R\lt(\mu_u^{N-1,\e_N} \rt)  \rt)^2-\lt( \alpha_{\delta,\eta}^R\lt(\mu_u^N \rt)  \rt)^2\rt)\Delta \phi(\VV_u^{N,i})du+\CC(\mu^N)-\CC(f)\\
	&+\mei \prod_{k=1}^n \phi_{k}(\VV^{N,i}_{t_k}) \Bigl(\int_s^t \sqrt{2}\alpha^R_{\delta,\eta}\lt(\mu_u^{N-1,\e_N}\rt)\nabla\phi(\VV_u^{N,i})\cdot d\BB_u^{i} \\
	&+\int_{[s,t]\times\EE_i^N}\lt( \phi(\VV_{u_-}^{N,i}+c_{\gamma,\nu}(\VV_{u_-}^{N,i},\VV_{u_-}^{N,j}+w,\varphi,z))-\phi(\VV_{u_-}^{N,i})\rt)\bar{\MM}^{i}_N(du,dz,d\varphi,dw,dj)\Bigr).
\end{align*}	
First, using Lemma \ref{lem:regCO} we obtain 
\begin{align*}
	&\lt|\lt( \alpha_{\delta,\eta}^R\lt(\mu_s^N \rt)  \rt)^2- \lt( \alpha_{\delta,\eta}^R\lt(\mu_s^{N-1,\e_N} \rt)  \rt)^2\rt|=\lt| \alpha_{\delta,\eta}^R\lt(\mu_s^N \rt)-  \alpha_{\delta,\eta}^R\lt(\mu_s^{N-1,\e_N} \rt)  \rt| \lt( \alpha_{\delta,\eta}^R\lt( \mu_s^N \rt)+ \lt( \alpha_{\delta,\eta}^R\lt(\mu_s^{N-1,\e_N} \rt)  \rt)  \rt) \\
	&\leq C_{R,\delta,\eta} W_1(\mu_s^N,\mu_s^{N-1,\e_N})\leq C_{R,\delta,\eta}\lt( W_1\lt(\mu_s^N,\frac{1}{N-1}\sum_{j\neq i}^N \delta_{\VV_s^{N,j}}\rt)  + W_1\lt(\frac{1}{N-1}\sum_{j\neq i}^N \delta_{\VV_s^{N,j}},\mu_s^{N-1,\e_N}\rt)  \rt).
\end{align*}	
By considering the transport plan which consists in splitting the atom $\delta_{\VV_s^{N,i}}$ of mass $1/N$ in $\mu_s^N$, into $N-1$ atoms of mass $1/(N(N-1))$ and transporting each of these atoms onto each of the $\delta_{\VV_s^{N,j}}$ for $j\neq i$, we obtain by definition of the Wasserstein 1 metric
\begin{align*}
	\E\lt[W_1\lt(\mu_s^N,\frac{1}{N-1}\sum_{j\neq i}^N \delta_{\VV_s^{N,j}}\rt)\rt]  \leq \frac{1}{N(N-1)}\sum_{j\neq i} \E\lt[|\VV_s^{N,i}-\VV_s^{N,j}|\rt]\leq N^{-1} C_k\lt(\E\lt[|\VV_s^{N,i}|^k\rt]+1\rt).
\end{align*}	 
Moreover since $W_1$ is convex, we have by Jensen's inequality
\begin{align*}
W_1\lt(\frac{1}{N-1}\sum_{j\neq i}^N \delta_{\VV_s^{N,j}},\mu_s^{N-1,\e_N}\rt)\leq W_1(\delta_0,\rho_{\e_N}).
\end{align*}	 
Therefore
\begin{align*}
	\E\lt[\lt|\mei \prod_{k=1}^n \phi_{k}(\VV^{N,i}_{t_k})\int_s^t \lt(\lt(\alpha_{\delta,\eta}^R\lt(\mu_u^{N-1,\e_N} \rt)  \rt)^2-\lt( \alpha_{\delta,\eta}^R\lt(\mu_u^N \rt)  \rt)^2\rt)\Delta \phi(\VV_u^{N,i})du\rt|\rt]\leq C_{\phi,n,\delta,\eta,M_{0,k},T,k}\lt(N^{-1}+W_1(\delta_0,\rho_{\e_N})\rt).
\end{align*}	

Moreover, since $\CC$ is smooth on $\PP(\mathbb{D}([0,T],\R^3))$ (due to Lemma \ref{lem:regCO}), and $\mu^N$ converges in law to $f$, $\E\lt[ \lt|\CC(\mu^N)-\CC(f)\rt| \rt]$ converges to $0$ as $N$ goes to infinity. Finally, since the $N$ independent Brownian motions $(\BB^i)_{i=1,\cdots,N}$ are themselves independent of the $N$ independent Poisson random measures $(\MM^i)_{i=1,\cdots,N}$, we deduce from classical stochastic calculus that
\begin{align*}
	\E\Bigl[\Bigl|&\mei \prod_{k=1}^n \phi_{k}(\VV^{N,i}_{t_k}) \Bigl(\int_s^t \sqrt{2}\alpha^R_{\delta,\eta}\lt(\frac{1}{N-1}\sum_{j\neq i}^N \delta_{\VV_u^{N,j}}*\rho_{\e_N} \rt)\nabla\phi(\VV_u^{N,i})\cdot d\BB_u^{i} \\
	&+\int_{[s,t]\times\EE_i^N}\lt( \phi(\VV_{u_-}^{N,i}+c_{\gamma,\nu}(\VV_{u_-}^{N,i},\VV_{u_-}^{N,j}+w,\varphi,z))-\phi(\VV_{u_-}^{N,i})\rt)\bar{\MM}^{i}_N(du,dz,d\varphi,dw,dj)\Bigr) \Bigr|^2 \Bigr]\\
	&\leq C_{(\phi_k)_{k=1,\cdots,n},\phi,t,s} N^{-1}.
\end{align*}	 

Combining all the estimates obtained in this step yields

\begin{align*}
	\lim_{N\rightarrow \infty}\E\lt[ \lt| \KK^N-\CC(f)  \rt| \rt]=0.
\end{align*}	

$\Box$ Conclusion. \newline 

Gathering all the estimates obtained so far letting $N$ go to infinity and $\e$ to $0$, we have
\[
\E\lt[ \lt| \MM(f)-\CC(f)  \rt|  \rt]=0.
\]
Therefore $\MM(f)=\CC(f)$ almost surely, and the result is proved.
\end{proof}

\begin{lemma}
	\label{lem:sing}
	Whenever $\gamma+\nu>0$, the set $\mathcal{S}$ defined in \eqref{eq:S}, is a singleton.
\end{lemma}

\begin{proof}
If $f\in \mathcal{S}$, by definition $f$ satisfies condition $(b)$ of \eqref{eq:mart}. It follows from Lemma \ref{lem:LpHnu}, that $f\in L^1(0,T;L^p)$ for $p\in \lt(\frac{3}{3+\gamma}, \frac{3}{3-\nu} \rt)$. Indeed for $q>\frac{\frac{\nu}{d-\nu}}{\frac{d}{d-\nu}-p}$ there holds $k\geq |\gamma| q(p-1/q)$ and we obtain
\begin{align*}
\int_0^T\|f_t\|_{L^p}dt&\leq  C_{p,q,d,\nu}\int_0^T\lt(\int_{\R^3} \lal v \ral^{-\gamma q(p-1/q)} f_t\rt)^{\theta_1}  \|\lal \cdot \ral^{\gamma} f_t \|^{\theta_2}_{L^1(\R^3)}   \lt|\lal  \cdot \ral^{\gamma/2} \sqrt{f_t}\rt|^{2\theta_3}_{H^{\nu/2}(\R^3)}dt\\
&\leq  C_{p,q,d,\nu}\int_0^T\lt(\int_{\R^3} \lal v \ral^{k} f_t+  \lt|\lal \cdot \ral^{\gamma/2} \sqrt{f_t}\rt|^{2}_{H^{\nu/2}(\R^3)}+1 \rt)dt.
\end{align*}

	Then the result is a simple consequence \cite[Theorem 1.3]{FGue}. We provide a short sketch of proof, summarizing \cite[Section 3]{FGue}. \newline	
	Let $Q,\tilde{Q}\in \PP( \mathbb{D}([0,T],\R^3))$ be two solutions to the martingale problem \eqref{eq:mart}.  Classical stochastic calculus tools provide the existence, on some suitable probability space, of a $g_0$-distributed random variable $\VV_0$, a Brownian motion $(\BB_t)_{t\in [0,T]}$, a Poisson random measure $\MM$ on $[0,T]\times \R^3\times\R^3\times [0,2\pi]\times \R^+$ with the intensity $ds\times R_s(dv,dv_*)\times d\varphi\times dz$ where for any $s\in[0,T]$, $R_s\in \PP(\R^3\times\R^3)$ is a coupling plan between $Q_s$ and $\tilde{Q}_s$, and a $Q$-distributed process $(\VV_t)_{t\in[0,T]}$ (resp. a $\tilde{Q}$-distributed process $(\mathcal{W}_t)_{t\in[0,T]}$ ) such that 
	\begin{align*}
	&\VV_t=\VV_0+\int_{[0,t]\times \R^6\times [0,2\pi]\times \R^+ } c_{\gamma,\nu}(\VV_{s_-},v,\varphi,z) \bar{\MM}(ds,dv,dv_*,d\varphi,dz) + \int_0^t \int_{\R^3} K_\gamma(\VV_s-v) Q_s(dv) ds+ \int_0^t \sqrt{2} \alpha_{\delta,\eta}^R(Q_s)d\BB_s\\
	&\mathcal{W}_t=\VV_0+\int_{[0,t]\times \R^6\times [0,2\pi]\times \R^+ } c_{\gamma,\nu}( \mathcal{W}_{s_-},v_*,\varphi+\varphi_0(\VV_{s_-}-v,\mathcal{W}_{s_-}-v_*),z) \bar{\MM}(ds,dv,dv_*,d\varphi,dz) \\
	&\quad \quad \quad + \int_0^t \int_{\R^3} K_\gamma( \mathcal{W}_s-v) \tilde{Q}_s(dv) ds+ \int_0^t \sqrt{2}\alpha_{\delta,\eta}^R(\tilde{Q}_s)d\BB_s,
	\end{align*}
	where $\varphi_0:\R^3\times\R^3\mapsto [0,2\pi)$ is the measurable function of \cite[Lemma 3.2]{FGue}. Using some Ito's expansion and taking the expectation yields
	\begin{align*}
	&\E\lt[|\VV_t-\WW_t|^2\rt]=\int_0^t \lt(\alpha_{\delta,\eta}^R(Q_s)-\alpha_{\delta,\eta}^R(\tilde{Q}_s)\rt)^2ds\\
	+&\int_0^t\int_{\R^6}\E\lt[\int_0^\infty\int_0^{2\pi} \lt|c_{\gamma,\nu}(\VV_{s_-},v,\varphi,z)-c_{\gamma,\nu}( \mathcal{W}_{s_-},v_*,\varphi+\varphi_0(\VV_{s_-}-v,\mathcal{W}_{s_-}-v_*),z)\rt|^2  dzd\varphi  \rt]R_s(dv,dv_*)ds \\
	+&2\int_0^t\int_{\R^6}\E\lt[\int_0^\infty\int_0^{2\pi}(\VV_{s^-}-\WW_{s^-}) \lt(c_{\gamma,\nu}(\VV_{s_-},v,\varphi,z)-c_{\gamma,\nu}( \mathcal{W}_{s_-},v,\varphi+\varphi_0(\VV_{s_-}-v,\mathcal{W}_{s_-})-v_*,z)\rt)  dzd\varphi  \rt]R_s(dv,dv_*)ds \\
	& :=I_t+J_t+K_t.
	\end{align*}
	Using \cite[(3.20)]{FGue}, we claim that for any $p>\frac{3}{3+\gamma}$
	\[
	I_t+J_t\leq \int_0^t C_p\lt(1+\|Q_s\|_{L^p}+\|\tilde{Q}_s\|_{L^p}\rt)\E\lt[|\VV_s-\WW_s|^2\rt]ds.
	\]
	Due to Lemma \ref{lem:regCO} and classical Wasserstein metric property, we easily get
	\[
	\lt(\alpha_{\delta,\eta}^R(Q_s)-\alpha_{\delta,\eta}^R(\tilde{Q}_s)\rt)^2\leq C_{R,\delta,\eta}W^2_1(Q_s,\tilde{Q}_s)\leq C_{R,\delta,\eta}W^2_2(Q_s,\tilde{Q}_s)\leq \E\lt[|\VV_s-\WW_s|^2\rt].
	\]
	Finally we obtain by Gronwall's inequality
	\[
	W^2_2(Q_t,\tilde{Q}_t)\leq\E\lt[|\VV_t-\WW_t|^2\rt]\leq C_{p,R,\delta,\eta}\int_0^t\lt(1+\|Q_s\|_{L^p}+\|\tilde{Q}_s\|_{L^p}\rt)\E\lt[|\VV_s-\WW_s|^2\rt]ds,
	\]
	which yields by application of Gronwall's inequality, $W^2_2(Q_t,\tilde{Q}_t)=0$ for any $t\in [0,T]$ and concludes the proof.
\end{proof}	
 
\subsection{Conclusion}

We are now in position to prove Theorem \ref{thm:main1}. Set $T>0$, and
$$R >2+ M_{2,0}^{1/2},  \ \eta<  C\lt(1-\frac{M_{2,0}}{4(R-2)^2}\rt), \ \delta <  \frac{CR^{-5}}{ \exp\lt(4\frac{C+2(H+M_{2,0})}{\eta^2}  \rt)}.$$
For each $N\geq 2$, consider $\VV^N_0$ a $G^N_0$-distributed random vector, and let $(\VV_t^N)_{t \in [0,T]} \in  \mathbb{D}([0,T],\R^{3N})$ be a weak solution to \eqref{eq:BKnu>1} starting from $\VV^N_0$\newline
Due to Proposition \ref{prop:tig}, there is a subsequence of $\lt(\mei \delta_{\lt(\VV_t^{N,i}\rt)_{t\in [0,T]}}\rt)_{N\geq 2}$ which converges in law to some random variable $f \in \PP( \mathbb{D}([0,T],\R^3))$, of law $\pi \in \PP(\PP( \mathbb{D}([0,T],\R^3)))$.\newline
Moreover we know form Proposition \ref{prop:mart} and Lemma \ref{lem:sing}, that the support of $\pi$ is reduced to the singleton which unique element consists in the solution to \eqref{eq:mart}. Hence it follows that the full sequence $\lt(\mei \delta_{\lt(\VV_t^{N,i}\rt)_{t\in [0,T]}}\rt)_{N\geq 2}$ converges in law to $f$, as $N$ goes to infinity, since it has a unique accumulation point. Note that in particular, the time marginals of $f$, solves the perturbed Boltzmann equation \eqref{eq:BolP} for the initial condition $g_0$, and satisfies $f\in L^1(0,T; L^p(\R^3))$ for any $p\in \lt(\frac{3}{3+\gamma},\frac{3}{3-\nu}\rt)$.\newline
We conclude by remarking that for our particular choice of $R,\delta,\eta$, the unique solution to the perturbed Boltzmann equation starting from $g_0$, is actually the unique solution to the Boltzmann equation \eqref{eq:Bol} (with the exhibited regularity), denoted $(g_t)_{t\in [0,T]}$. Indeed, by the classical H Theorem and conservation of kinetic energy, there holds for any $t\geq 0$
\[
H(g_t)\leq H(g_0), \int_{\R^3} |v|^2g_t(v)dv=\int_{\R^3} |v|^2g_0(v)dv.
\]  
Hence due to Lemma \ref{lem:regCO2}, for any $t\geq 0$
\[
\alpha_{\delta,\eta}^R(g_t)=0,
\]
which concludes the proof.

\section*{Acknowledgements}
The author was supported by the Fondation Mathématique Jacques Hadamard, and warmly thanks Maxime Hauray, Nicolas Fournier and Nicolas Rougerie for many advices, comments and discussions which have made this work possible.

\appendix

\section{Properties of the collision parametrization }
In this appendix, we prove some useful properties of the coefficient $c_{\gamma,\nu}$ defined in \eqref{eq:c}. We beign with the
\begin{lemma}
	\label{lem:use}
	For any $k\geq 2$ there are $C_k,C_\nu>0$ such that for $v,v_*\in \R^3\times\R^3$ it holds
	\begin{itemize}
		\item[$(i)$]
		\[
		\int_{\R^+\times [0,2\pi]} c_{\gamma,\nu}(v,v_*,z,\varphi)dzd\varphi=\beta_0K_\gamma(v-v_*).
		\]
		\item[$(ii)$]
		\[
		\int_{\R^+\times [0,2\pi]} |c_{\gamma,\nu}(v,v_*,z,\varphi)|^2dzd\varphi=\beta _0 |v-v_*|^{\gamma+2}.
		\]
		\item[$(iii)$]
		\[
		\lt|\int_{\R^+\times [0,2\pi]}\tilde{\Delta}(|\cdot|^k)(v,v_*,z,\varphi)dzd\varphi\rt| \leq C_k(|v|^k+|v_*|^k+1).
		\]
	
	\end{itemize}
\end{lemma}

\begin{proof}
	The first two points are obtain thanks to the simple substitution $\theta=G_\nu\lt(   \frac{z}{|v_i-w|^\gamma} \rt)$ and definition \eqref{eq:c} (see for instance \cite[Lemma 2.1]{FGue} for the first point and \cite[Lemma 2.3]{FGue} for the second).\newline
	For $s\in (0,1)$ we set
	\[
	v_s=v+sc_{\gamma,\nu}\lt(v,v_*,z,\varphi \rt).
	\]
	By Taylor's expansion we find that
	\begin{align*}
		\tilde{\Delta}(|\cdot|^k)(v,v_*,z,\varphi)&=k\int_0^1 |v_s|^{k-2}\lal c_{\gamma,\nu}\lt(v,v_*,z,\varphi \rt) ,  I_3+(k-2)\frac{v_s\otimes v_s}{|v_s|^2}  , c_{\gamma,\nu}\lt(v,v_*,z,\varphi \rt) \ral ds\\
		&\leq k\int_0^1 |v_s|^{k-2} \lt( \lt|c_{\gamma,\nu}\lt(v,v_*,z,\varphi \rt)\rt|^2 +(k-2)\frac{ \lt(c_{\gamma,\nu}\lt(v,v_*,z,\varphi \rt)\cdot v_s \rt)^2}{|v_s|^2}  \rt)ds\\
		&\leq 3k(k-1) \int_0^1 |v_s|^{k-2}  \lt|c_{\gamma,\nu}\lt(v,v_*,z,\varphi \rt)\rt|^2 ds\\
		&\leq 3k(k-1) \lt(|v|+|c_{\gamma,\nu}\lt(v,v_*,z,\varphi \rt)|\rt)^{k-2}  \lt|c_{\gamma,\nu}\lt(v,v_*,z,\varphi \rt)\rt|^2\\
		&\leq C_k\lt( |v|^{k-2}\lt|c_{\gamma,\nu}\lt(v,v_*,z,\varphi \rt)\rt|^2+ \lt|c_{\gamma,\nu}\lt(v,v_*,z,\varphi \rt)\rt|^k \rt).
	\end{align*}
	
	Therefore
	
	\begin{align*}
		\int_{\R^+\times \Sd} \tilde{\Delta}(|\cdot|^k)(v,v_*,z,\varphi)dz d\varphi&\leq C_k\lt(    \int_{\R^+\times [0,2\pi]}   |v|^{k-2} \lt|c_{\gamma,\nu}\lt(v,v_*,z,\varphi \rt)\rt|^2+ \lt|c_{\gamma,\nu}\lt(v,v_*,z,\varphi \rt)\rt|^k dz d\varphi \rt)\\
		&\leq C_k \lt( |v|^{k-2}|v-v_*|^{\gamma+2}+|v-v_*|^{\gamma+k}\rt)\leq C_k\lt( |v|^k+|v_*|^k +1\rt).
	\end{align*}
	since $\gamma+2>0$. Which concludes the proof. 
\end{proof}	

Then we need the

\begin{lemma}
	\label{lem:reg1}
	For any $\phi \in \CC^2(\R^3)$ the function
	\[
	(v,v_*)\mapsto \int_{ [0,2\pi]\times \R^+ }\tilde{\Delta}\phi(v,v_*,z,\varphi)d\varphi dz,
	\]
	is continuous.
\end{lemma}

\begin{proof}
	By Taylor's expansion and change of variables we have 
	\begin{align*}
\int_{ [0,2\pi]\times \R^+ }\tilde{\Delta}\phi(v,v_*,z,\varphi)d\varphi dz&=\int_0^1\int_{ [0,2\pi]\times \R^+ }c_{\gamma,\nu}(v,v_*,\varphi,z)\cdot \nabla^2 \phi( v+sc_{\gamma,\nu}(v,v_*,\varphi,z)) \cdot c_{\gamma,\nu}(v,v_*,\varphi,z)dzd\varphi ds\\
	&=|v-v_*|^{\gamma+2}\int_0^1\int_{ [0,\pi]\times [0,2\pi] }\beta(\theta) \frac{v'-v}{|v-v_*|}`\cdot \nabla^2 \phi( v+s(v'-v) )\cdot \frac{v'-v}{|v-v_*|}d\varphi d\theta ds.
	\end{align*}
Then recall from \eqref{eq:c}, that $\frac{v'-v}{|v-v_*|}=\begin{pmatrix}
-\frac{1-\cos(\theta)}{2}\\ 
\frac{\sin(\theta)}{2}\cos(\varphi)\\ 
\frac{\sin(\theta)}{2}\sin(\varphi)
\end{pmatrix}$ in the orthonormal basis $\lt(\frac{v-v_*}{|v-v_*|},\frac{I(v-v_*)}{|v-v_*|},\frac{J(v-v_*)}{|v-v_*|}\rt)$.
Therefore, since $\nabla^2 \phi$ is bounded
\[
\int_{ [0,\pi]\times [0,2\pi] }\beta(\theta)\lt| \frac{v'-v}{|v-v_*|}\rt|^2\lt|\nabla^2 \phi( v+s(v'-v) )\rt|d\varphi d\theta < \infty.
\]
We conclude by observing that for any $s\in [0,1],\theta\in [0,\pi]$ and $\varphi \in [0,2\pi]$ the function
\[
(v,v_*)\mapsto |v-v_*|^{\gamma+2} \beta(\theta) \frac{v'-v}{|v-v_*|}`\cdot \nabla^2 \phi( v+s(v'-v) )\cdot \frac{v'-v}{|v-v_*|}
\]
is continuous (since $\phi\in \CC^2$ and $\gamma+2>0$).
\end{proof}

\section{Regularizing effects of grazing collisions }
\label{App:B}

In this appendix, we gather some properties of the coefficient $\alpha^R_{\delta,\eta}$ defined in \eqref{eq:al}. We begin with the
\begin{lemma}
	\label{lem:regCO}
	For any $R,\delta,\eta>0$, there is $ C_{R,\delta,\eta}$ such that for any $f,g\in \PP_1(\R^3)$
	\begin{align*}
		\lt|\alpha^R_{\delta,\eta}(f)-\alpha^R_{\delta,\eta}(g)\rt|\leq C_{R,\delta,\eta} W_1(f,g).
	\end{align*}
\end{lemma}
\begin{proof}
	First observe that for fixed $\ee\in \Sd$, the function 
	\[
	(v,v_*)\in \R^3\times\R^3\mapsto \chi_R(v)\chi_R(v_*)\chi_{2\delta} \lt((v-v_*)\cdot\ee \rt)\in [0,1],
	\]
	is $(\|\chi_R\|_{Lip}+\|\chi_{2\delta}\|_{Lip})$-Lipschitz. Then using the reversed triangular inequality and Kantorovitch-Rubinstein duality, we obtain
	\begin{align*}
		&\lt|\sup_{\ee \in \Sd} \int_{\R^{3}\times \R^3} \chi_R(v)\chi_R(v_*)\chi_{2\delta} \lt((v-v_*)\cdot\ee \rt)f(dv)f(dv_*)-\sup_{\ee \in \Sd} \int_{\R^{3}\times \R^3} \chi_R(v)\chi_R(v_*)\chi_{2\delta} \lt((v-v_*)\cdot\ee \rt)g(dv)g(dv_*)\rt|\\
		&\leq \sup_{\ee \in \Sd}\lt| \int_{\R^{3}\times \R^3} \chi_R(v)\chi_R(v_*)\chi_{2\delta} \lt((v-v_*)\cdot\ee \rt)\lt(f(dv)f(dv_*)-g(dv)g(dv_*)\rt)\rt|\\
		&\leq (\|\chi_R\|_{Lip}+\|\chi_{2\delta}\|_{Lip})W_1(f\otimes f,g\otimes g)\leq 2(\|\chi_R\|_{Lip}+\|\chi_\delta\|_{Lip})W_1( f,g),
	\end{align*}	
	so that
	\begin{align*}
		&\lt|\alpha^R_{\delta,\eta}(f)-\alpha^R_{\delta,\eta}(g)\rt|\leq \|\chi_{4\eta}\|_{Lip} \lt|  \int_{\R^3} \chi_{R-1}(v)f(dv) -\int_{\R^3} \chi_{R-1}(v)g(dv) \rt|+\\
		&\|1-\chi_{\frac{\eta^2}{2}}\|_{Lip} \lt|\sup_{\ee \in \Sd} \int_{\R^{3}\times \R^3} \chi_R(v)\chi_R(v_*)\chi_{\delta} \lt((v-v_*)\cdot\ee \rt)f(dv)f(dv_*)-\sup_{\ee \in \Sd} \int_{\R^{3}\times \R^3} \chi_R(v)\chi_R(v_*)\chi_{\delta} \lt((v-v_*)\cdot\ee \rt)g(dv)g(dv_*)\rt|\\
		&\leq C_{R,\eta,\delta}W_1(f,g).
	\end{align*}
\end{proof}	

We then introduce a kind of decomposition of \cite[Proposition 3]{ADVW}. We begin with a generalization of some classical equiintegrability result \cite[Lemma 6]{DesVil}, in the

\begin{lemma}
	\label{lem:regCO2}
	Let $f\in L\ln L(\R^3)\cap \PP_2(\R^3)$ and $H,M\in \R$ be such that $M\geq \|f\|_{L^2_1}$ and $H\geq H(f)$. When $R,\delta,\eta$ satisfy
	\begin{align*}
		&R >2+ M^{1/2},  \ \eta<  C\lt(1-\frac{M}{4(R-2)^2}\rt), \ \delta <  \frac{CR^{-5}}{ \exp\lt(4\frac{C+2(H+M)}{\eta^2}  \rt)}, \\
	\end{align*}
	for an explicit numerical constant $C>0$, it holds 
	\[
	\alpha^R_{\delta,\eta}(f)=0.
	\]
\end{lemma}

\begin{proof}
	We first choose $R>2+ M^{1/2}$ and $\eta \in (0,1/4)$ such that $1-\frac{M}{(R-2)^2}> 8\eta$, so that $\chi_{4\eta}\lt( \int_{\R^3} \chi_{R-1}(v)g(dv)\rt)=0$.\newline
	Then define for any $\ee\in \Sd$
	\[
	K^{R}_{\delta,\eta}(\mathbf{e}):=\Bigl\{(v,v_*), \ \text{s.t.} \ |v|\leq R+1,|v_*|\leq R+1, |(v-v_*)\cdot\ee|\leq 4\delta \Bigr\}\subset\R^3\times\R^3,
	\]
	and observe that  it holds
	\begin{align*}
		\int_{\R^3\times \R^3} \mb_{K^{R}_{\Delta,\eta}(\mathbf{e})}dvdv_*&\leq \int_{\R^3}\mb_{|v_*|\leq R+1}\lt(\int_{\R^3} \mb_{|v|\leq R+1} \mb_{|(v-v_*)\cdot\ee|\leq 4\delta}dv\rt) dv_*\leq C R^5\delta.
	\end{align*}
	Following \cite[point $(ii)$ of Lemma C.1]{MN} we obtain for some constant $C>1$ 
	\[
	\sup_{\mathbf{e}\in \Sd} f\otimes f\lt(K^{R}_{\Delta,\eta}(\mathbf{e})\rt) \leq \frac{C+H(f\otimes f)+ \int_{\R^6} (|v|^k+|v_*|^k) f(dv)f(dv_*) }{-\ln( C R^5\delta) }.
	\]
	We then choose $\delta>0$ such that
	\[
	\delta < \frac{CR^{-5}}{ \exp\lt(4\frac{C+2(H+M)}{\eta^2}  \rt)},
	\]
	so that for any $\ee\in \Sd$
	\begin{align*}
		\int_{\R^{3}\times \R^3} \chi_R(v)\chi_R(v_*)\chi_{\delta} \lt((v-v_*)\cdot\ee \rt)f(dv)f(dv_*)\leq 	\int_{\R^3\times \R^3} \mb_{K^{R}_{\Delta,\eta}(\mathbf{e})}f(dv)f(dv_*) \leq \eta^2/4,
	\end{align*}
	and then $(1-\chi_{\eta^2/2})\lt(\sup_{\ee \in \Sd} \int_{\R^{3}\times \R^3} \chi_R(v)\chi_R(v_*)\chi_{\delta} \lt((v-v_*)\cdot\ee \rt)f(dv)f(dv_*)\rt)=0$, and the result is proved.
\end{proof}

\textit{Proof of claim \eqref{eq:imp}}

Note that if, foo some $R,\delta,\eta>0$, $f\in \PP(\R^3)$ does not satisfy $\mathbb{A}(R,\delta,\eta)$, then it satsfies
\begin{align*}
	\overline{\mathbb{A}(R,\delta,\eta)} :& \ \int_{\B_ R}f(dv)< 4\eta, \ \text{or} \ \exists \mathbf{e}\in \Sd, \exists x\in \B_R \ \int_{P_\delta^{\mathbf{e},x}}  f(dv)> \eta\ \text{with} \ P_\delta^{\mathbf{e},x}= \{ y \in \B_R, \ |(y-x)\cdot \mathbf{e} | \leq \delta \}.
\end{align*}

	We proceed by disjunction of the cases. \newline

$\bullet$ Case 1 : $\int_{B_R}f< 4\eta$.\newline
In this case we clearly have
\begin{align*}
\int_{\R^3}\chi_{R-1}(v)f(dv)< 4\eta,
\end{align*}
and therefore
\begin{align*}
\chi_{4\eta}\lt(\int_{\R^3}\chi_{R-1}(v)f(dv)\rt)=1.
\end{align*}

$\bullet$ Case 2 : $\exists \mathbf{e} \in \Sd$ and $x\in \B_R$ such that  $\int_{P_\delta^{\mathbf{e},x}}  f(dv)> \eta$.\newline
For any $v,v_*\in P_\delta^{\mathbf{e},x}$ it holds
\begin{align*}
\lt|(v-v_*)\cdot \mathbf{e}\rt|&\leq 2 \delta.
\end{align*} 
and then
\begin{align*}
\eta ^2 \leq \lt(\int_{P_\delta^{\mathbf{e},x}}f(dv)\rt)^2&=\int_{\R^3\times \R^3} \mb_{P_\delta^{\mathbf{e},x}}(v)\mb_{P_\delta^{\mathbf{e},x}}(v_*)f(dv)f(dv_*)\\
&	\leq  \int_{\R^{3}\times \R^3} \chi_R(v)\chi_R(v_*)\chi_{2\delta} \lt((v-v_*)\cdot \mathbf{e} \rt)f(dv)f(dv_*)\\
&	\leq \sup_{\ee \in \Sd} \int_{\R^{3}\times \R^3} \chi_R(v)\chi_R(v_*)\chi_{2\delta} \lt((v-v_*)\cdot\ee \rt)f(dv)f(dv_*)
\end{align*}
and then
\[
\alpha_{\delta,\eta}^R(\mu^N)\geq 1.
\]
\qedsymbol

\textit{Sketch of proof of Proposition \ref{prop:Alex}}\newline

We follow \cite{Anoich} and define for $\gamma\in (-2,0)$
\begin{align*}
\CC_\gamma\lt(g,f\rt):&=\int_{\R^6}\int_{\Sd}B(v-v_*,\sigma) (f(v')-f(v))^2 g(v_*)dvdv_*d\sigma\\
&=\int_{\R^6}\int_{\Sd} |v-v_*|^\gamma b(\cos(\theta))(f(v')-f(v))^2 g(v_*)dvdv_*d\sigma.
\end{align*}
The main tool of the proof is the
\begin{lemma}[Corollary 2.1 of \cite{ADVW}]
	\label{lem:Bob}
	\[
	\CC_0(g,f)=\int_{\R^6}\int_{\Sd}b(\cos(\theta))g(v_*)\lt(f(v')-f(v)\rt)^2dvdv_*d\sigma\geq \int_{\R^3} \lt|\FF(f)(\xi)\rt|^2 \int_{\Sd} b\lt( \frac{\xi}{|\xi|}\cdot \sigma  \rt)\lt(\FF(g)(0)-\lt|\FF(g)(\xi^-)\rt|\rt)d\sigma d\xi
	\]
	where $\xi^{-}=\dem(\xi-|\xi|\sigma)$. 
\end{lemma}	

	Let $\phi_R$ be a smooth function satisfying $ \mb_{|v|\geq 2R} \geq \phi_R(v)\geq \mb_{|v|\geq 4R}$ and  $(v_j)_{j\in J}\subset B_{4R}$ with $\text{card}(J)=:n\leq C R^3/\delta^3$ such that $B_{4R}\subset \bigcup_{j\in J} B^{v_j}_{\delta/4}$.\newline
	For each $j\in J$ we choose $\phi_j$ a smooth function such that $\mb_{(B^{v_j}_{\delta/4})^c} \geq \phi_j \geq \mb_{(B^{v_j}_{\delta/2})^c}$ and define $\chi_j(v)=\chi_{\delta/4}(|v-v_j|)$ (for simplicity $\phi_{n+1}=\phi_R$, $\chi_{n+1}=\chi_{4R}$). \newline 
	Since $\gamma+\nu>0$, following the argument of \cite[ (2.7) and (2.8) of the proof of Proposition 2.1]{Anoich}  we have
\bq
\label{eq:estAlex}\CC_\gamma\lt(g,f\rt)\geq  C_{R,\delta,\gamma} \sum_{j=1}^{n+1} \CC_0\lt(\phi_{j}  g, \chi_{j}\lal \cdot \ral^{\gamma/2}f \rt)-C_{R,\delta} \|g\|_{L^1}\|f\|^2_{L^2_{\gamma/2}}.
\eq
Then for each $\xi,\sigma$ there is $\theta_{\xi^-}$ such that
\[
\lt|\FF\lt(\phi g\rt)(\xi^-)\rt|=\FF\lt(\phi g\rt)(\xi^-)e^{-i\theta_{\xi_-}}=\int_{\R^3}\phi g(w)e^{-i(w\cdot \xi_-+\theta_{\xi_-})}dw.
\]
Therefore since the latter quantity is real, there holds
\begin{align*}
\FF\lt(\phi g\rt)(0)-\lt|\FF\lt(\phi g\rt)(\xi^-)\rt|&= \int_{\R^3} \phi g(w)\lt( 1- \cos(w\cdot \xi_-+\theta_{\xi_-})\rt)dw \\
&=2\int_{\R^3}\sin^2\lt(\frac{w\cdot \xi_-+\theta_{\xi_-}}{2}\rt) \phi(w)g(dw)\\
&\geq 2\mb_{|\xi^{-}|\leq \pi R^{-1}}\int_{\R^3}\sin^2\lt(\frac{w\cdot \xi_-+\theta_{\xi_-}}{2}\rt) \phi(w)g(dw),
\end{align*}
Denoting
\[
\KK_{\delta,\xi_-}:=\Bigl\{ v \in \B_R, \ \exists \ell \in \mathbb{Z}, \ \lt| v\cdot \xi_-+ \theta_{\xi_-}+2\ell\pi \rt|  \leq \delta  \Bigr\},
\]	
we have
\begin{align*}
\FF\lt(\phi g\rt)(0)-\lt|\FF\lt(\phi g\rt)(\xi^-)\rt|&\geq \mb_{|\xi^{-}|\leq \pi R^{-1}}\int_{\R^3\setminus \KK_{\delta|\xi^{-1}|,\xi_-} }\sin^2\lt(\frac{w\cdot \xi_-+\theta_{\xi_-}}{2}\rt) \phi(w)g(dw)\\
&\geq \mb_{|\xi^{-}|\leq \pi R^{-1}}\sin^2\lt(\frac{\delta |\xi_-|}{2}\rt)\int_{\B_R\setminus \KK_{\delta|\xi^{-1}|,\xi_-} } \phi(w)g(dw).
\end{align*}
Next observe that since $|\xi^{-}|\leq \pi R^{-1}$, there is at most one $k\in \mathbb{Z}$ such that
\[
\KK_{\delta|\xi^{-1}|,\xi_-}\Bigl\{ |v|\leq R, \ \ \lt| v\cdot \frac{\xi_-}{|\xi_-|}+ \frac{\theta_{\xi_-}+2k\pi}{|\xi_-|} \rt|  \leq \delta \Bigr\}.
\]
Hence this set is reduced to $P^{x, \delta}_{ \frac{\xi_-}{|\xi_-|}}$ as defined in \eqref{eq:cond}, for some $x\in \B_R$.
So that for a numeric constant $C>0$, and each $j=1,\cdots,n+1$ there holds
\begin{align*}
\sin^2\lt(\frac{\delta |\xi_-|}{2}\rt)\int_{\B_R\setminus \KK_{\delta|\xi^{-1}|,\xi_-} } \phi_j(w)g(dw) &\geq C\delta^2|\xi_-|^2 \int_{\B_R\setminus \KK_{\delta|\xi^{-1}|,\xi_-} }  \mb_{(B_{\delta/2}^{v_j})^c}(w)  g(dw)\\
&\geq C\delta^2|\xi_-|^2  \lt( \int_{\B_R} g(w)dw-\int_{B_{\delta/2}^{v_j}} g(w)dw-\int_{\KK_{\delta|\xi^{-1}|,\xi_-}} g(w)dw\rt)\\
&\geq C\delta^2|\xi_-|^2  \lt( 4\eta-\eta-\eta\rt)=C\delta^2\eta |\xi_-|^2, 
\end{align*}
since $g$ satisfies $\mathbb{A}(R,\delta,\eta)$ by assumption.\newline
Coming back to \eqref{eq:estAlex} and using Lemma \ref{lem:Bob} and the fact that $\|g\|_{L^1}=1$, we deduce that
\begin{align*}
\CC_\gamma\lt(g,f\rt)&\geq  \delta^2 \eta C_{R,\delta,\gamma} \int_{\R^3}\lt( \sum_{j=1}^{n+1} \lt|\FF\lt(\chi_{j}\lal \cdot \ral^{\gamma/2}f \rt)(\xi)\rt|^2\rt)\int_{\Sd} b\lt( \frac{\xi}{|\xi|}\cdot \sigma  \rt)\lt|\xi^-\rt|^2\mb_{|\xi^{-}|\leq \pi R^{-1}}d\sigma d\xi-C_{R,\delta} \|f\|^2_{L^2_{\gamma/2}}\\
&.
\end{align*}
But since $|\xi_-|^2=|\xi|^2 \lt(1-\frac{\xi}{|\xi|}\cdot \sigma\rt)$, using assumption $\mathbf{H}$ we have
\begin{align*}
\int_{\Sd} \mb_{|\xi_-|\leq \pi R^{-1}}b\lt(\frac{\xi}{|\xi|}\cdot \sigma \rt)|\xi_-|^2 d\sigma
&= \int_{0}^{\pi} \beta\lt(\cos(\theta) \rt)\mb_{|\xi|\sqrt{1-\cos(\theta)}\leq \pi R^{-1}}|\xi|^2(1-\cos(\theta)) d\theta\\
&\geq  K_2 |\xi|^2 \int_{0}^{\pi} \mb_{|\xi|\theta \leq \pi R^{-1}}\theta^2\frac{d\theta}{\theta^{1+\nu}}=K_2|\xi|^2\lt[(2-\nu)^{-1}\theta^{2-\nu}\rt]_{0}^{\pi R^{-1}|\xi|^{-1}}=C_\nu|\xi|^\nu.
\end{align*}
Therefore
\begin{align*}
\CC_\gamma\lt(g,f\rt)&\geq   C_{R,\delta,\gamma,\eta} \sum_{j=1}^{n+1} \lt|\chi_{j}\lal \cdot \ral^{\gamma/2}f \rt|^2_{H^{\nu/2}(\R^3)}-C_{R,\delta} \|f\|^2_{L^2_{\gamma/2}}.
\end{align*}	
But using that $ \sum_{j=1}^{n+1} \chi^2_j  \geq 1$ and observation \eqref{eq:Hnu|} we find that
\begin{align*}
&\sum_{j=1}^{n+1}  \lt|\chi_j \lal \cdot \ral^{\gamma/2}f\rt|_{H^{\nu/2}(\R^3)}^2+4\int_{\R^{6}} \lal x \ral^{\gamma}f^2(x)\frac{\sum_{j=1}^{n+1} \lt( \chi_j(x)-\chi_j(y)\rt)^2 }{|x-y|^{3+\nu}}dxdy\\
&\quad \geq C\int_{\R^{6}}\lt( \sum_{j=1}^{n+1} \chi^2_j(x) \rt)\frac{  \lt( \lal x \ral^{\gamma/2}f(x) -\lal y \ral^{\gamma/2}f(y)\rt)^2 }{|x-y|^{3+\nu}}dxdy  \geq  C\lt| \lal \cdot \ral^{\gamma/2}f\rt|_{H^{\nu/2}(\R^3)}^2,
\end{align*}
and the result is proved, since by \eqref{eq:Hnu|}
\[
4\int_{\R^{6}} \lal x \ral^{\gamma}f^2(x)\frac{\sum_{j=1}^{n+1} \lt( \chi_j(x)-\chi_j(y)\rt)^2 }{|x-y|^{3+\nu}}dxdy \leq C_{R,\delta} \|\lal \cdot \ral^{\gamma/2}  f \|^2_{L^2_{\gamma/2}}.
\]
\qedsymbol

We can now give the result which enables to provide the key estimate \eqref{eq:Fish} in the

\section{Proof of Proposition \ref{prop:HLS}}
In this appendix, we prove Proposition \ref{prop:HLS}. Note that in the tensorized case $F^N=f^{\otimes N}$, we directly obtain the result,  thanks to the Hardy-Littlewood-Sobolev inequality \cite[Theorem 4.3]{Lieb}, and point $(i)$ of Lemma \ref{lem:aff} below. We split the proof in four Steps \newline 

$\diamond$ Step one : A preliminary computation. \newline 

We denote $F_2\in \PP_{\textit{sym}}(\R^{2d})$ the two particles marginal of $F^N$. By symmetry we have
\[
\sup_{i\neq j}\int_{\R^{dN}}|v_i-v_j|^{-\lambda}F^N(dV)=	\int_{\R^{2d}}|v_1-v_2|^{-\lambda}F_2(dv_1,dv_2).
\]

Let $(V_1,V_2)$ be a couple of symmetric random variables of joint law $F_2\in \PP_{\textit{sym}}(\R^d\times \R^d)$. We introduce 
$\Phi(x,y)=\lt( x-y  ,y\rt)$, with	$\Phi^{-1}(x,y)=\lt(  x+y ,y\rt)$. Denote $\tilde{F}_2=F_2\circ \Phi^{-1}$, so that $$\tilde{F}_2=\LL(V_1-V_2,V_2).$$
Let us denote $f$ the first marginal of $\tilde{F}_2$, i.e. $f=\LL(V_1-V_2)$. By unitary changes of variables, we obtain

\begin{align*}
\II^2_{\nu,\gamma}&(F_2)= \int_{\R^d}  \int_{ \R^{2d}} \frac{\lt(\lal v_* \ral^{\gamma/2}\sqrt{F_2(v_*,v_1)}-\lal v \ral^{\gamma/2}\sqrt{F_2(v,v_1)}\rt)^2}{|v_*-v|^{d+\nu}}dvdv_* dv_1  \\
&= \int_{\R^d}  \int_{ \R^{2d}} \frac{\lt(\lal v_* \ral^{\gamma/2}\sqrt{\tilde{F}_2\lt(  v_*-v_1  ,v_1\rt)}-\lal v \ral^{\gamma/2}\sqrt{\tilde{F}_2(v-v_1,v_1)}\rt)^2}{|v_*-v_1-(v-v_1)|^{d+\nu}}dvdv_* dv_1 \\	
&=	
\int_{\R^d}  \int_{ \R^{2d}} \frac{\lt(\lal v_*+v_1 \ral^{\gamma/2}\sqrt{\tilde{F}_2\lt(v_*,v_1\rt)}-\lal v+v_1 \ral^{\gamma/2}\sqrt{\tilde{F}_2(v,v_1)}\rt)^2}{|v_*-v|^{d+\nu}}dvdv_* dv_1\\
&=\int_{\R^d} \lal v_1\ral^\gamma \int_{ \R^{2d}} \frac{\lt(a(v,v_1)\sqrt{\tilde{F}_2\lt(v_*,v_1\rt)}-a(v_*,v_1)\sqrt{\tilde{F}_2(v,v_1)}\rt)^2}{|v_*-v|^{d+\nu}}dvdv_* dv_1=\int_{\R^3}\lal v_1 \ral^\gamma\lt|a(\cdot,v_1)\sqrt{\tilde{F}_2(\cdot,v_{1})}\rt|^2_{H^{\nu/2}}dv_1.
\end{align*} 

with 
\[
a(v,v_1)=\lt( \frac{1+|v+v_1|^2}{1+|v_1|^2} \rt)^{\gamma/4}
\]

$\diamond$ Step two : An intermediary information. \newline

Let be $S>0$ and $r\in \lt(\dem,1\rt)$. We define
\bq
\label{eq:Psir}
\Psi_r:(x,y)\in \R^+\times \R^+ \mapsto (\sqrt{x}-\sqrt{y})^{2r}\lt(\frac{x+y}{2}\rt)^{1-r},
\eq
and for $G^2\in \PP(\R^d\times\R^d)$
\begin{align*}
&\II^2_{r,S}(G^2)= \int_{ \R^{d}}\int_{\R^{d}\times \R^d} \chi_{2S}(v-v_*) \frac{\Psi_r( \chi^2_S(v)G^2(v,v_{1}),\chi^2_S(v_*)G^2(v_*,v_1) )}{|v_*-v|^{r(d+\nu)}} dvdv_* dv_1\\
&=\int_{ \R^{d}}\int_{\R^{d}\times \R^d}\chi_{2S}(v-v_*) \frac{\lt(\chi_S(v)\sqrt{G^2(v,v_{1})}-\chi_S(v_*)\sqrt{G^2(v_*,v_1)}\rt)^{2r}}{|v_*-v|^{r(d+\nu)}}\lt(\frac{\chi_S^2(v)G^2(v,v_{1})+\chi_S^2(v_*)G^2(v_*,v_1)}{2}\rt)^{1-r} dvdv_* dv_1\\
\end{align*}

We claim in this step that there is a constant such that

\bq
\label{eq:claimHol}
\II^2_{r,S}(\tilde{F}_2)\leq C_{\gamma,S,d,\nu}\lt(\II^2_{\nu,\gamma}(F_2)+1\rt)^r \lt( \int_{\R^{2d}} |w|^{-\gamma\frac{r}{1-r}} F_2(dv,dw) +1 \rt)^{1-r}.  \\
\eq

Indeed by Holder's inequality we obtain

\begin{align*}
\II^2_{r,S}(\tilde{F}_2)\leq & \lt(\int_{ \R^{d}}\lal v_1\ral ^{\gamma}\int_{\R^{d}\times \R^d}\chi_{2S}(v-v_*) \frac{\lt(\chi_S(v)\sqrt{\tilde{F}_2(v,v_{1})}-\chi_S(v_*)\sqrt{\tilde{F}_2(v_*,v_1)}\rt)^{2}}{|v_*-v|^{d+\nu}}dvdv_* dv_1\rt)^r\\
&\times \lt(\int_{ \R^{d}}\lal v_1\ral^{-\gamma\frac{r}{1-r}}\int_{\R^{d}\times \R^d}\chi_{2S}(v-v_*) \lt(\frac{\chi_S(v)\tilde{F}_2(v,v_{1})+\chi_S(v_*)\tilde{F}_2(v_*,v_1)}{2}\rt)dvdv_* dv_1\rt)^{1-r}\\
&:=(K_1)^r(K_2)^{1-r}
\end{align*}

By symmetry we easily find
\begin{align*}
K_2&\leq \int_{\R^{d}\times \R^d} \lal v_1\ral^{-\gamma\frac{r}{1-r}}  \chi_S(v)\tilde{F}_2(v,v_{1}) \lt(\int_{ \R^{d}} \chi_{2S}(v-v_*) dv_* \rt) dv dv_1\\
&\leq C_{S,d} \int_{\R^{d}\times \R^d} \lal v_1\ral^{-\gamma\frac{r}{1-r}}\tilde{F}_2(v,v_{1})  dv dv_1 = C_{S,d}  \int_{\R^{d}\times \R^d} \lal v_1\ral^{-\gamma\frac{r}{1-r}}F_2(v,v_{1})  dv dv_1,
\end{align*}
since, by definition $\tilde{F}_2$ and $F_2$ have the same second marginal. Then using \eqref{eq:Hnu|}, we have that for any $v_1\in \R^d$

\begin{align*}
\lt|\chi_S(\cdot)\sqrt{\tilde{F}_2(\cdot,v_{1})}\rt|^2_{H^{\nu/2}}\leq 2\int_{\R^d\times \R^d}\chi_S^2(v_*)\frac{\lt(\sqrt{\tilde{F}_2(v,v_{1})}-\sqrt{\tilde{F}_2(v_*,v_1)}\rt)^2}{|v-v_*|^{d+\nu}}dvdv_*+2\int_{\R^d\times \R^d}\frac{\lt(\chi_S(v)-\chi_S(v_*)\rt)^2}{|v-v_*|^{d+\nu}} \tilde{F}_2(v,v_{1})dvdv_*
\end{align*}

and similarly

\begin{align*}
\lt|a(\cdot,v_1)\sqrt{\tilde{F}_2(\cdot,v_{1})}\rt|^2_{H^{\nu/2}}&\geq \dem \int_{\R^d\times \R^d}a^2(v_*,v_1)\frac{\lt(\sqrt{\tilde{F}_2(v,v_{1})}-\sqrt{\tilde{F}_2(v_*,v_1)}\rt)^2}{|v-v_*|^{d+\nu}}dvdv_* \\
&- 2\int_{\R^d\times \R^d}\frac{\lt(a(v,v_1)-a(v_*,v_1)\rt)^2}{|v-v_*|^{d+\nu}} \tilde{F}_2(v,v_{1})dvdv_*
\end{align*}

Then observe that for any $v_1\in \R^d$ and $|v|\leq S$, there holds $a(v,v_1)\geq C_{S,\gamma}$.
Indeed	for $|v_1|\leq 2S$ we obviously have $\frac{1+|v+v_1|^2}{1+|v_1|^2}\leq 1+9S^2$
and $a(v,v_1)\geq \lt( 1+9S^2  \rt)^{\gamma/4}$. On the other hand  $|v_1|\geq 2S$ we have $1 + \frac{2v\cdot v_1}{1+|v_1|^2} + \frac{|v|^2}{1+|v_1|^2}\leq 4$ and therefore 
$a(v,v_1)\geq 4^{\gamma/4}$.

Gathering all the above inequalities, and multiplying by $\lal v_1 \ral^\gamma$ and integrating w.r.t the Lebesgue measure on $v_1$ yields

\begin{align*}
K_1\leq& \int_{\R^3}\lal v_1 \ral^\gamma\lt|\chi_S(\cdot)\sqrt{\tilde{F}_2(\cdot,v_{1})}\rt|^2_{H^{\nu/2}}dv_1 \leq C_{S,\gamma}\int_{\R^3}\lal v_1 \ral^\gamma\lt|a(\cdot,v_1)(\cdot)\sqrt{\tilde{F}_2(\cdot,v_{1})}\rt|^2_{H^{\nu/2}}dv_1\\
& + C_{S,\gamma}\int_{\R^{2d}} \lt(\int_{\R^d}\frac{\lt(\lal v+v_1\ral ^{\gamma/2}-\lal v_*+v_1\ral ^{\gamma/2}\rt)^2+\lt(\chi_S(v)-\chi_S(v_*)\rt)^2}{|v-v_*|^{d+\nu}}dv_*\rt) \tilde{F}_2(v,v_{1}) dvdv_1 
\end{align*}

and since $\chi_S$ is bounded and Lipschitz, and so is $v\mapsto \lal v+v_1 \ral^{\gamma/2}$ for any $v_1$ (uniformly in $v_1$), it follows that
\begin{align*}
&C_{S,\gamma,d,\nu}\lt(\II^2_{\nu,\gamma}(F_2)+1\rt) \geq  K_1,
\end{align*}	
which conclude the step. \newline

$\diamond$ Step three : Jensen's inequality  \newline 

We leave the reader check that the function $\Psi_r$ defined in \eqref{eq:Psir} is convex $r\in \lt(\dem,1\rt)$. As a consequence, we obtain that for any $G^2\in \PP(\R^d\times\R^d)$, there holds
\bq
\label{eq:claimJen}
\begin{aligned}
\II^2_{r,S}(G^2)&= \int_{ \R^{d}}\int_{\R^{d}\times \R^d} \chi_{2S}(v-v_*) \frac{\Psi_r( \chi^2_S(v)G^2(v,v_{1}),\chi^2_S(v_*)G^2(v_*,v_1) )}{|v_*-v|^{r(d+\nu)}} dvdv_* dv_1\\
&\geq \int_{\R^{d}\times \R^d} \chi_{2S}(v-v_*) \frac{\Psi_r\lt(\chi^2_S(v)\int_{ \R^{d}} G^2(v,v_{1})dv_1,\chi^2_S(v_*)\int_{ \R^{d}}G^2(v_*,v_1)dv_1 \rt)}{|v_*-v|^{r(d+\nu)}} dvdv_*. 
\end{aligned}
\eq
Indeed  due to Jensen's inequality we have for any $L>0$
\begin{align*}
&(C_d L^d)^{-1}\int_{ \R^{d}}\mb_{|v_1|\leq L}\Psi_r( \chi^2_S(v)G^2(v,v_{1}),\chi^2_S(v_*)G^2(v_*,v_1) )dv_1\\
&\geq \Psi_r\lt(\chi^2_S(v)(C_d L^d)^{-1}\int_{ \R^{d}}\mb_{|v_1|\leq L} G^2(v,v_{1})dv_1,\chi^2_S(v_*)(C_d L^d)^{-1}\int_{ \R^{d}}\mb_{|v_1|\leq L}G^2(v_*,v_1)dv_1 \rt)\\
& = (C_d L^d)^{-1}\Psi_r\lt(\chi^2_S(v)\int_{ \R^{d}}\mb_{|v_1|\leq L} G^2(v,v_{1})dv_1,\chi^2_S(v_*)\int_{ \R^{d}}\mb_{|v_1|\leq L}G^2(v_*,v_1)dv_1 \rt),
\end{align*} 
since $\Psi_r$ is 1-homogeneous. Therefore letting $L$ go to infinity, multiplying by $\chi_{2S}(v-v_*)|v-v_*|^{-r(d+\nu)}$ and integrating w.r.t. $v,v_*$ yields the desired bound. \newline

$\diamond$ Step four : A Sobolev's like inequality. \newline 

For any  $\e \in \lt(\frac{1}{r}-1, \frac{\nu}{d}\rt)$, we set $s=\frac{\nu-d\e}{2}$. In this step we claim that there is a constant $C_{d,s,\nu,\e,S}>0$ such that for any $h\in L^1$
\bq
\label{eq:claimSob}
\|  h\|_{L^{\frac{d}{d-2rs}}}\leq  C_{d,s,\nu,r,\e,S} \lt(\lt(\int_{\R^{2d}}\chi_{2S}(x-y)  \frac{\Psi_r(h(v),h(v_*))}{|v-v_*|^{r(d+\nu)}}dvdv_*\rt)^{1/2r}+\|h\|^{1/2r}_{L^1}\rt).
\eq
Indeed, rewriting
\begin{align*}
|h|^{p}_{W^{s,p}(\R^d)}&=\int_{\R^{2d}}\chi_{2S}(x-y)  \frac{(h(x)-h(y))^{p}}{|x-y|^{d+sp}}dxdy+\int_{\R^{2d}}(1-\chi_{2S}(x-y))  \frac{(h(x)-h(y))^{p}}{|x-y|^{d+sp}}dxdy,
&:=K_1+K_2.
\end{align*}
First observe that
\begin{align*}
&K_2\leq C_{p} \int_{\R^{2d}}(1-\chi_{2S}(x-y))  \frac{h^{p}(x)+h^{p}(y)}{|x-y|^{d+sp'}}dxdy\\
&=2C_{p} \int_{\R^d} h^{p}(x) \lt(\int_{\R^d} (1-\chi_{2S}(x-y)) |x-y|^{-(d+sp)}dy\rt)d\leq C_{S,d,s,p} \|h\|^{p}_{L^{p}}.
\end{align*}
Then since
\begin{align*}
d+sp= \frac{p}{2r}r(d+\nu)+\frac{2-(1+\e)p}{2}d,
\end{align*}
we have
\begin{align*}
&K_1=\int_{\R^{2d}}\chi_{2S}(x-y)  \frac{(\sqrt{h(x)}-\sqrt{h(y)})^{p}}{|x-y|^{d+sp}}(\sqrt{h(x)}+\sqrt{h(y)})^{p}dxdy\\
&\leq C\int_{\R^{2d}}\chi_{2S}(x-y)  \frac{(\sqrt{h(x)}-\sqrt{h(y)})^{p}}{|x-y|^{d+sp}}(h(x)+h(y))^{p/2}dxdy\\
&=C\int_{\R^{2d}} \chi_{2S}(x-y) \lt( \frac{(\sqrt{h(x)}-\sqrt{h(y)})^{2r}}{|x-y|^{r(d+\nu)}}(h(x)+h(y))^{1-r} \rt)^{p/2r}\frac{(h(x)+h(y))^{p(2r-1)/2r}}{|x-y|^{\frac{2-(1+\e)p}{2}d}}dxdy\\
&\leq C \lt(\int_{\R^{2d}} \chi_{2S}(x-y)  \frac{\Psi_r(h(x),h(y))}{|x-y|^{r(d+\nu)}}dxdy\rt)^{p/2r}\lt(\int_{\R^{2d}} \chi_{2S}(x-y)\frac{(h(x)+h(y))^{p\frac{2r-1}{2r-p}}}{|x-y|^{\frac{r(2-(1+\e)p)}{2r-p}d}} dxdy \rt)^{\frac{2r-p}{2r}}.
\end{align*}	
But by symmetry, and since $\frac{r(2-(1+\e)p)}{2r-p}<1$, we have
\begin{align*}
\int_{\R^{2d}} \chi_{2S}(x-y)\frac{(h(x)+h(y))^{p\frac{2r-1}{2r-p}}}{|x-y|^{\frac{r(2-(1+\e)p)}{2r-p}d}} dxdy &\leq C_{p,r} \int_{\R^d} h^{p\frac{2r-1}{2r-p}}(x) \lt( \int_{\R^d}  \chi_{2S}(x-y) |x-y|^{-\frac{r(2-(1+\e)p)}{2r-p}d} dy \rt) dx \\
&\leq C_{p,r,\e,d,S} \|h\|^{p\frac{2r-1}{2r-p}}_{L^{p\frac{2r-1}{2r-p}}}.
\end{align*}
Hence 
\[
|h|^{p}_{W^{s,p}(\R^d)} \leq C_{p,r,\e,d,S} \lt(\int_{\R^{2d}} \chi_{2S}(x-y)  \frac{\Psi_r(h(x),h(y))}{|x-y|^{r(d+\nu)}}dxdy\rt)^{p/2r} \|h\|_{L^{p\frac{2r-1}{2r-p}}}^{p\frac{2r-1}{2r}}+ C_{S,d,s,p} \|h\|^{p}_{L^{p}}.
\]
By interpolation between Lebesgue spaces we have
\[
\|h\|_{L^{p}}\leq \|h\|_{L^1}^{\frac{1}{2r}}\|h\|_{L^{p\frac{2r-1}{2r-p}}}^{\frac{2r-1}{2r}},
\]
therefore we find that
\[
|h|^{p}_{W^{s,p}(\R^d)} \leq C_{p,r,\e,d,S} \lt(\lt(\int_{\R^{2d}} \chi_{2S}(x-y)  \frac{\Psi_r(h(x),h(y))}{|x-y|^{r(d+\nu)}}dxdy\rt)^{p/2r} + C_{S,d,s,p}\|h\|_{L^1}^{\frac{p}{2r}}\rt) \|h\|_{L^{p\frac{2r-1}{2r-p}}}^{p\frac{2r-1}{2r}}.
\]
Then due to Sobolev's embedding (see for instance \cite[Theorem 6.5]{DiNez}) 
\[
\|h\|_{L^{p*}}\leq C_{\text{Sob}} 	|h|_{W^{s,p}(\R^d)}, 
\]
for $p^*=\frac{dp}{d-sp}$. Hence 
\[
\|h\|_{L^{p*}}\leq C_{p,r,\e,d,S} \lt(\lt(\int_{\R^{2d}} \chi_{2S}(x-y)  \frac{\Psi_r(h(x),h(y))}{|x-y|^{r(d+\nu)}}dxdy\rt)^{1/2r} + \|h\|_{L^1}^{\frac{1}{2r}}\rt) \|h\|_{L^{p\frac{2r-1}{2r-p}}}^{\frac{2r-1}{2r}}
\]
We then choose $p$ such that
\[
\frac{dp}{d-sp}=p\frac{2r-1}{2r-p}, \ \text{i.e.} p=\frac{d}{d-s(2r-1)}, 
\]
which leads to $p^*=\frac{d}{d-s2r}$
to conclude the step. \newline

$\diamond$ Final step : \newline

We now fix $r\in \lt(\frac{d}{d+\nu}\vee \frac{\nu+\lambda}{2\nu}\vee \frac{2d-(\nu-\lambda)}{2d} ,1\rt)$, and $\e \in \lt(\frac{1}{r}-1, \frac{\nu}{d} \wedge\frac{\nu-\lambda}{2dr}\rt)$ so that $r(\nu-\e d)>\lambda$, and set $p=\frac{d}{d-2r\frac{\nu-\e d}{2}}$. Using the fact that $F_2$ is a probability measure, and unitary change of variables, we obtain

\begin{align*}
\int_{\R^{2d}}|v_1-v_2|^{-\lambda}F_2(dv_1,dv_2)=&\int_{\R^{2d}}|v_1-v_2|^{-\lambda}\chi^2_{S}(v_1-v_2)+|v_1-v_2|^{-\lambda}(1-\chi^2_{S}(v_1-v_2))F_2(dv_1,dv_2)\\
&\leq \int_{\R^{2d}}|v_1-v_2|^{-\lambda}\chi^2_{S}(v_1-v_2)F_2(dv_1,dv_2)+S^{-\lambda}\\
&= \int_{\R^d}|v|^{-\lambda}\chi^2_{S}(v)f(v)dv+S^{-\lambda}\\
&\leq \lt(\int_{\R^d}\mb_{|v|\leq S+1}|v|^{-\lambda p'}dv\rt)^{1/p'}\|\chi^2_{S}f\|_{L^p}+S^{-\lambda}\leq C_{d,S,\lambda,r,\e,p}, \lt(\|\chi^2_{S}f\|_{L^p}+1\rt).
\end{align*}
since $\lambda p'= d \frac{\lambda }{r(\nu-\e d)}<d$. Using successively \eqref{eq:claimSob}, \eqref{eq:claimJen} and \eqref{eq:claimHol} we obtain
\begin{align*}
\|\chi^2_{S}f\|_{L^p} &\leq   C_{d,s,\nu,r,\e,S} \lt(\lt(\int_{\R^{2d}}\chi_{2S}(x-y)  \frac{\Psi_r(\chi^2_{R}f(v),\chi^2_{S}f(v_*))}{|v-v_*|^{r(d+\nu)}}dvdv_*\rt)^{1/2r}+\|\chi^2_{R}f\|^{1/2r}_{L^1}\rt)\\
&\leq  C_{d,s,\nu,r,\e,S} \lt(\lt(\II^2_{r,S}(\tilde{F}_2)\rt)^{1/2r}+1\rt) \\
&\leq C_{d,s,\nu,r,\e,S,\gamma}\lt(\lt(\lt(\II^2_{\nu,\gamma}(F_2)+1\rt)^r \lt( \int_{\R^{2d}} |w|^{-\gamma\frac{r}{1-r}} F_2(dv,dw) +1 \rt)^{1-r}\rt)^{1/2r}+1\rt), 
\end{align*} 
which concludes the proof, thanks to point $(iii)$ of Lemma \ref{lem:aff}.

\section{Proof of Proposition \ref{thm:Fish2}}

In this appendix, we prove Proposition \ref{thm:Fish2}. We begin with stating some properties satisfied by the functionals defined in \eqref{eq:Fishdef1}, and which allows to call these functionals \text{information}, in the

\begin{lemma}
	\label{lem:aff}
	For any $\gamma,\nu$, the family of functionals $(\II_{\nu,\gamma}^N)_{N\geq 1}$ satisfies the following properties.
	\begin{itemize}
		\item[$(i)$] For any $N\geq 2$, $\II_{\nu,\gamma}^N$ is convex and lower semi continuous w.r.t. the weak convergence on $\PP(\R^{dN})$ and for any $f\in \PP(\R^d)$
 		it holds
 		\[
		\II^N_{\nu,\gamma}(f^{\otimes N})=\lt|\sqrt{\lal \cdot \ral ^\gamma f}\rt|^2_{H^{\nu/2}(\R^d)}.
		\]
		
		\item[$(ii)$] For any $G^N\in \PP(\R^{dN})$, denote $G^k$ and $G^{N-k}$ its marginals on $\R^{dk}$ and $\R^{d(N-k)}$, then
		\[
		N\II^N_{\nu,\gamma}(G^N)\geq k\II^k_{\nu,\gamma}(G^k)+(N-k)\II^{N-k}_{\nu,\gamma}(G^{N-k}),
		\]
		with equality if and only if $G^N=G^k\otimes G^{N-k}$.
		\item[$(iii)$] For any $G^N\in \Ps$ and $G^k$ its marginals on $\R^{dk}$, it holds
		\[
		\II^k_{\nu,\gamma}(G^k)\leq \II^{N}_{\nu,\gamma}(G^N).
		\]
		\end{itemize}
\end{lemma}

\begin{proof}
	
	Proof of $(i)$ \newline
	We obtain the lower semi continuity, invoking \cite[Theorem 3.4.3]{But}, following the argument of \cite[Lemma 2.4]{Erb1} or \cite[Proposition 2.5]{Erb2}, since the function $(x,y)\mapsto (\sqrt{x}-\sqrt{y})^2$ is jointly convex, continuous and 1-homogeneous. \newline
	The last point is straightforward by definition \eqref{eq:Fishdef1}.\newline

	Proof of $(ii)$ \newline	
	
	First observe that for any $\alpha,\beta,a,b\geq 0$	there holds
	\bq
	\label{eq:id}
	(\sqrt{a\alpha}-\sqrt{b\beta})^2=(\sqrt{a}-\sqrt{b})^2\frac{\alpha+\beta}{2}+(a-b)\frac{(\alpha-\beta)}{2}+2\sqrt{ab}\lt(\frac{\alpha+\beta}{2}-\sqrt{\alpha\beta}\rt),
	\eq
	
	Let be $G_N\in \PP(\R^{dN})$ and $(V_1,\cdots,V_N)$ a random vector of law $G_N$, fix $k=1,\cdots,N$ and $ i\in\{1,\cdots,k\}$ and denote 
	\[
	g_{k}(v_1,\cdots,v_k|v_{k+1},\cdots,v_N)=\LL(V_1,\cdots,V_k\, | \, V_{k+1},\cdots,V_N),
	\]
	then 
	\[
	G_N(v_1,\cdots,v_N)=g_k(v_1,\cdots,v_k|v_{k+1},\cdots,x_N) G_{N-k}(v_{k+1},\cdots,v_N)=g_{N-k}(v_{k+1},\cdots,v_N|v_{1},\cdots,v_k)G_{k}(v_{1},\cdots,v_k).
	\]
	Then by definition we have 
	\begin{align*}
		&	N\II^N_{\nu,\gamma}(G^N)=\\
		&\sum_{i=1 }^N\int_{\R^{d(N-1)}}\int_{\R^{2d}} \frac{\lt(\lal v \ral^{\frac{\gamma}{2}}\sqrt{G^N(v_1,v_{i-1},v,v_{i+1},v_N)}-\lal v_* \ral^{\frac{\gamma}{2}}\sqrt{G^N(v_1,v_{i-1},v_*,v_{i+1},v_N)}\rt)^2}{|v_*-v|^{d+\nu}} dvdv_* dV^{N-1}_i\\
		&=\lt(\sum_{i=1}^k+\sum_{i=k+1}^N\rt)\int_{\R^{d(N-1)}}\int_{\R^{2d}} \frac{\lt(\lal v \ral^{\frac{\gamma}{2}}\sqrt{G^N(v_1,v_{i-1},v,v_{i+1},v_N)}-\lal v_* \ral^{\frac{\gamma}{2}}\sqrt{G^N(v_1,v_{i-1},v_*,v_{i+1},v_N)}\rt)^2}{|v_*-v|^{d+\nu}} dvdv_* dV^{N-1}_i.\\
		&:=K_1+K_2.
		\end{align*}	
		Then due to \eqref{eq:id}, we have
		
			\begin{align*}		
		K_1\geq & \sum_{i=1}^k \int_{\R^{d(N-1)}}\int_{\R^{2d}} \frac{\lt(\lal v \ral^{\frac{\gamma}{2}}\sqrt{G^k(V^k_v)}-\lal v_* \ral^{\frac{\gamma}{2}}\sqrt{G^k(V^k_{v_*}))}\rt)^2}{|v_*-v|^{d+\nu}}\frac{g^{N-k}(V^{N-k}|V^k_{v})+g^{N-k}(V^{N-k}|V^k_{v_*})}{2} dvdv_* dV^{N-1}_i\\
		&+\sum_{i=1}^k \int_{\R^{d(N-1)}}\int_{\R^{2d}} \frac{\lt(\lal v \ral^{\gamma}G^k(V^k_v)-\lal v_* \ral^{\gamma}G^k(V^k_{v_*})\rt)}{|v_*-v|^{d+\nu}}\frac{\lt(g^{N-k}(V^{N-k}|V^k_{v})-g^{N-k}(V^{N-k}|V^k_{v_*})\rt)}{2} dvdv_* dV^{N-1}_i\\
		&= \sum_{i=1}^k \int_{\R^{d(k-1)}}\int_{\R^{2d}} \frac{\lt(\lal v \ral^{\frac{\gamma}{2}}\sqrt{G^k(V^k_v)}-\lal v_* \ral^{\frac{\gamma}{2}}\sqrt{G^k(V^k_{v_*}))}\rt)^2}{|v_*-v|^{d+\nu}} dvdv_* dV^{k-1}_i=k\II^k_{\nu,\gamma}(G^k).
		\end{align*}
		By similar considerations, we obtain 
		\[
		K_2\geq (N-k)\II^{N-k}_{\nu,\gamma}(G^{N-k}).
		\]

	Moreover, in view of \eqref{eq:id}, equality holds if and only if for any $(v,v_*)\in \R^d\times \R^d$ and $i=1,\cdots,k$ it holds
	\[
	g^{N-k}(V^{N-k}|V^k_{v})=g^{N-k}(V^{N-k}|V^k_{v_*}),
	\]
	which implies that for any $(v,v_*)\in \R^d\times \R^d$ and $i=1,\cdots,k$.
	\[
	g^{N-k}(V^{N-k}|V^k_{v})=g^{N-k}(V^{N-k}|V^k_{v_*})=G^{N-k}(V^{N-k}),
	\]
	and concludes the point.\newline
		
	Proof of point $(iii)$\newline 
	
	By symmetry we can rewrite 
	
	\begin{align*}
		&	\II^N_{\nu,\gamma}(G^N)=\int_{\R^{d(N-1)}}\int_{\R^{2d}} \frac{\lt(\lal v \ral^{\frac{\gamma}{2}}\sqrt{G^N(v,v_{2},\cdots,v_N)}-\lal v_* \ral^{\frac{\gamma}{2}}\sqrt{G^N(v_*,v_{2},\cdots,v_N)}\rt)^2}{|v_*-v|^{d+\nu}} dvdv_* dV^{N-1}_1\\
	\end{align*}	
Using once again \eqref{eq:id}, and integrating over $(v_{k+1},\cdots,v_N)$ yields
	\begin{align*}
	&	\II^N_{\nu,\gamma}(G^N)\geq \\
	&\int_{\R^{d(N-1)}}\int_{\R^{2d}} \frac{\lt(\lal v \ral^{\frac{\gamma}{2}}\sqrt{G^k(v,v_{2},\cdots,v_k)}-\lal v_* \ral^{\frac{\gamma}{2}}\sqrt{G^k(v_*,v_{2},\cdots,v_k)}\rt)^2}{|v_*-v|^{d+\nu}}\frac{g^{N-k}(V^{N-k}|V^k_{v})+g^{N-k}(V^{N-k}|V^k_{v_*})}{2} dvdv_* dV^{N-1}_1\\
	&=\int_{\R^{d(k-1)}}\int_{\R^{2d}} \frac{\lt(\lal v \ral^{\frac{\gamma}{2}}\sqrt{G^k(v,v_{2},\cdots,v_k)}-\lal v_* \ral^{\frac{\gamma}{2}}\sqrt{G^k(v_*,v_{2},\cdots,v_k)}\rt)^2}{|v_*-v|^{d+\nu}}dvdv_* dV^{k-1}_1=	\II^k_{\nu,\gamma}(G^k).
\end{align*}
	
	\end{proof}

For $\e>0$, define $\psi_\e$ on $\R^d$ as
\[
\psi_\e(x)=\e^{-2}e^{-\sqrt{1+\lt(\frac{|x|}{\e}\rt)^2}},
\]
which is borrowed from \cite{HM}. 

\begin{lemma}
	\label{lem:reg}

	Let be $\pi\in \PP(\PP(\R^d))$ and for $N\geq 1$ define
	\[
	\pi^{N,\e}=\int_{\PP(\R^d)}(\rho * \psi_\e)^{\otimes N}\pi(d\rho).
	\]
	Then for any $x_2,\cdots,x_N:=X^{N-1}\in \R^{d(N-1)}$, 
	define $p^\e(\cdot|X^{N-1})$ the conditional law knowing $X^{N-1}$ under $\pi^{N,\e}$. Then it holds
	\begin{itemize}
		\item[$(i)$] 
		\[
		\|\nabla \ln p^\e(\cdot|X^{N-1})\|_{L^\infty}	\leq \e^{-1}.
		\]
		\item[$(ii)$] For any $R>0$, there exist constants $C_{\e,R},c_{\e,R}>0$ such that for any $x\in \R^d,\mathbf{e}\in \mathbb{S}^{d-1}$ and $u\in [0,R]$
		\[
		c_{\e,R} p^\e(x|X^{N-1}) \leq p^\e(x+u\mathbf{e}|X^{N-1})\leq C_{\e,R} p^\e(x|X^{N-1}) .
		\]
	\end{itemize}

\end{lemma}

The proof of this technical lemma can be found in \cite[Lemma Appendix C.1]{SS}. We also need an affinity result on the mean fractional Fisher information. The question of affinity of mean information functional, is frequent particles system context (see \cite{RD} for the entropy,  \cite{Ki},\cite{HM} for the classical Fisher information, and \cite[Appendix A]{Rou} for the fractional Fisher information in the case $\gamma=0$). More precisely, we have the
\begin{lemma}
	\label{lem:aff2}
	Assume that $\gamma \leq 0 $ and $\nu \in (0,2)$. Define $\tilde{\II}_{\nu,\gamma}$, for $\pi\in\PP(\PP(\R^d))$ as
	\[
	\tilde{\II}_{\nu,\gamma}(\pi)=\limsup_{N \geq 1} \II^N_{\nu,\gamma}(\pi^N), \ \pi^N=\int_{\PP(\R^d)} \rho^{\otimes N} \pi(d\rho)
	\]
	Let $\pi\in\PP(\PP(\R^d))$ be such that $\int_{\PP(\R^d)}\|\rho\|_{L^1_\kappa}\pi(d\rho)<\infty$ for some $\kappa\geq 1$, and $f_0\in \PP_\kappa(\R^{d})$ and $r>0$. Denote $\B_r:=\{ \rho \in \PP_\kappa(\R^d), \ s.t. \ W_1(\rho,f_0) <r \}$, $F=(\pi(\B_r))^{-1}\pi_{|\B_r}$ and $G=(\pi(\B_r^c))^{-1}\pi_{|\B_r^c}$. Then for any $\theta\in (0,1)$ it holds
	\[
	\tilde{\II}_{\nu,\gamma}(\theta F+(1-\theta)G)=\theta \tilde{\II}_{\nu,\gamma}( F)+(1-\theta)\tilde{\II}_{\nu,\gamma}(G)
	\]
\end{lemma}

	\begin{proof}
		We divide the proof in four steps \newline
		
		$\diamond$ Step one : A preliminary computation :\newline
		
	Let $(F^N)_{N\geq 1},(G^N)_{N \geq 1}$ be two compatible sequences of symmetric probability densities. For $(v_2,\cdots,v_N)\in \R^{d(N-1)}$ fixed we denote $V_N^v=(v,v_2,\cdots,v_N)$, and define

	\[
	\begin{split}
	&\mathcal{D}^N(v,v_*):=  \lt(  \lal v \ral^{\frac{\gamma}{2}}\sqrt{\theta F^N(V^v_N)+(1-\theta)G^N(V^{v}_N)}- \lal v_* \ral^{\frac{\gamma}{2}}\sqrt{\theta F^N(V^{v_*}_N)+(1-\theta)G^N(V^{v_*}_N)} \rt)^2\\
	&-\theta \lt(  \lal v \ral^{\frac{\gamma}{2}}\sqrt{ F^N(V^v_N)}- \lal v_* \ral^{\frac{\gamma}{2}}\sqrt{F^N(V^{v_*}_N)} \rt)^2-(1-\theta)\lt(  \lal v \ral^{\frac{\gamma}{2}}\sqrt{ G^N(V^v_N)}- \lal v_* \ral^{\frac{\gamma}{2}}\sqrt{G^N(V^{v_*}_N)} \rt)^2,
	\end{split}
	\]

	After expanding the squares, we find
	\begin{align*}
	\mathcal{D}^N(v,v_*)=2	\lal v \ral^{\frac{\gamma}{2}}\lal v_* \ral^{\frac{\gamma}{2}}	\Bigl(    \theta  \sqrt{ F^N(V^v_N)} \sqrt{F^N(V^{v_*}_N)}&+ (1- \theta)  \sqrt{ G^N(V^v_N)} \sqrt{G^N(V^{v_*}_N)}\\
	&- \sqrt{\theta F^N(V^v_N)+(1-\theta)G^N(V^{v}_N)}\sqrt{\theta F^N(V^{v_*}_N)+(1-\theta)G^N(V^{v_*}_N)} \Bigr)
	\end{align*}	
	We first obtain the straightforward upper bound
	\bq
	\label{eq:B1}
	\lt|\mathcal{D}^N(v,v_*)\rt|\leq C_\theta\lt(F^{N}(V^{v}_N)+F^{N}(V^{v}_N)+G^{N}(V^{v}_N)+G^{N}(V^{v}_N)\rt)
	\eq

	Then observe that
	\begin{align*}
		\lt(\theta F^N(V^v_N)+(1-\theta)G^N(V^{v}_N)\rt)&\lt(\theta F^N(V^{v_*}_N)+(1-\theta)G^N(V^{v_*}_N)\rt)=\theta^2F^N(V^v_N)F^N(V^{v_*}_N)+(1-\theta)^2G^N(V^v_N)G^N(V^{v_*}_N)\\
		&+\theta(1-\theta) \lt( F^N(V^v_N)G^N(V^{v_*}_N)+ F^N(V^{v_*}_N)G^N(V^{v}_N)  \rt)  \\
		&=\lt(\theta \sqrt{F^N(V^v_N)}\sqrt{F^N(V^{v_*}_N)}+(1-\theta)\sqrt{G^N(V^v_N)}\sqrt{G^N(V^{v_*}_N)}\rt)^2\\
		&+\theta(1-\theta)\lt(F^N(V^v_N)G^N(V^{v_*}_N)+ F^N(V^{v_*}_N)G^N(V^{v}_N)-2\sqrt{F^N(V^v_N)G^N(V^{v_*}_N) F^N(V^{v_*}_N)G^N(V^{v}_N)}\rt).
	\end{align*}

Therefore

\begin{align*}
	\lt|\mathcal{D}^N(v,v_*)\rt|\leq &2		\lt(   \theta \sqrt{F^N(V^v_N)}\sqrt{F^N(V^{v_*}_N)}+(1-\theta)\sqrt{G^N(V^v_N)}\sqrt{G^N(V^{v_*}_N)}\rt)\\
	&\times \lt(1-\sqrt{1+\theta(1-\theta)\frac{F^N(V^v_N)G^N(V^{v_*}_N)+ F^N(V^{v_*}_N)G^N(V^{v}_N)-2\sqrt{F^N(V^v_N)G^N(V^{v_*}_N) F^N(V^{v_*}_N)G^N(V^{v}_N)}}{\lt(\theta \sqrt{F^N(V^v_N)}\sqrt{F^N(V^{v_*}_N)}+(1-\theta)\sqrt{G^N(V^v_N)}\sqrt{G^N(V^{v_*}_N)}\rt)^2}}\rt)\\
	&\leq 2\theta(1-\theta)\frac{F^N(V^v_N)G^N(V^{v_*}_N)+ F^N(V^{v_*}_N)G^N(V^{v}_N)-2\sqrt{F^N(V^v_N)G^N(V^{v_*}_N) F^N(V^{v_*}_N)G^N(V^{v}_N)}}{\theta \sqrt{F^N(V^v_N)}\sqrt{F^N(V^{v_*}_N)}+(1-\theta)\sqrt{G^N(V^v_N)}\sqrt{G^N(V^{v_*}_N)}}
\end{align*}
Then since the sequence $(F^N)_{N\geq 1}$ is compatible
\[
F^{N}(V^{v}_N)=F^{N-1}(V_{N-1})f(v|V_{N-1}),
\]
with $f(\cdot|V_{N-1})$ is the density of the conditional law of the first component of some random vector of law $F^N$, knowing the $N-1$ last components. We then obtain
\bq
\label{eq:B2}
\begin{aligned}
	\lt|\mathcal{D}^N(v,v_*)\rt|\leq & 2\theta(1-\theta)F^{N-1}(V_{N-1})G^{N-1}(V_{N-1})\frac{f(v)g(v_*)+ f(v_*)g(v)-2\sqrt{f(v_*)g(v)f(v)g(v_*)}}{\theta F^{N-1}(V_{N-1}) \sqrt{f(v)f(v_*)}+(1-\theta)G^{N-1}(V_{N-1})\sqrt{g(v)g(v_*)}}\\
	&=2	\frac{\theta(1-\theta)F^{N-1}(V_{N-1})G^{N-1}(V_{N-1})}{\theta F^{N-1}(V_{N-1}) \sqrt{f(v)f(v_*)}+(1-\theta)G^{N-1}(V_{N-1})\sqrt{g(v)g(v_*)}}\lt(\sqrt{f(v)g(v_*)}-\sqrt{f(v_*)g(v)}\rt)^2.\\
\end{aligned}
\eq	
	
By integration, this yields for any $L>0$ 
\begin{align*}
	&\lt| \II^N_{\nu,\gamma} ( \theta F^N+ (1-\theta) G^N)- \theta \II^N_{\nu,\gamma}(F^N)-(1-\theta)\II^N_{\nu,\gamma}(G^N) \rt|\leq   \int_{\R^{(d-1)N}} d V_{N-1} \int_{\R^{2d}}dvdv_* \frac{\lt|\mathcal{D}^N(v,v_*)\rt|}{|v-v_*|^{d+\nu}}\\
	&=\int_{\R^{(d-1)N}} d V_{N-1} \int_{\R^{2d}}dvdv_* \mb_{|v-v_*|<L} \frac{\lt|\mathcal{D}^N(v,v_*)\rt|}{|v-v_*|^{d+\nu}}+\mb_{|v-v_*|\geq L} \frac{\lt|\mathcal{D}^N(v,v_*)\rt|}{|v-v_*|^{d+\nu}} :=K_1+K_2. \\
\end{align*}

$\diamond$ Estimate of $K_1$ \newline

By symmetry and using bound \eqref{eq:B1}, we straightforwardly have

\begin{align*}
K_1&\leq C \int_{\R^{dN}} (F^{N}(V) +G^N(V)) \lt(\int_{\R^d} |v_1-v_*|^{-d-\nu}\mb_{|v_1-v_*|>L} dv_*  \rt) dV \leq C_{d,\nu} L^{-\nu}.
\end{align*}

$\diamond$ Estimate of $K_2$ \newline

We use bound \eqref{eq:B2}, to obtain

\begin{align*}
	K_2 \leq   C\int_{\R^{(d-1)N}} d V_{N-1} \int_{\R^{2d}}&	\mb_{|v-v_*|< L}\frac{\theta(1-\theta)F^{N-1}(V_{N-1})G^{N-1}(V_{N-1})}{\theta F^{N-1}(V_{N-1}) \sqrt{f(v)f(v_*)}+(1-\theta)G^{N-1}(V_{N-1})\sqrt{g(v)g(v_*)}}\\
&\frac{g(v_*)\lt(\sqrt{f(v)}-\sqrt{f(v_*)}\rt)^2+f(v_*)\lt(\sqrt{g(v)}-\sqrt{g(v_*)}\rt)^2}{|v-v_*|^{d+\nu}}dvdv_*. 
\end{align*}

$\diamond$ Step two : Regularization \newline

Fix $\eta>0$ and for $s\in (0,r)$ define
\[
F'=\mb_{\B_s}F, \quad F''=F-F'=\mb_{\B_r\setminus \B_s}F.
\]
Let be $\e>0$ fixed and define $F^\e$ as the push-forward of $F$ by the application $\rho\in \PP(\R^d)\mapsto \rho*\psi_\e \in \PP(\R^d)$ (and similarly $G^\e$ ) the sequence
\[
F'^{N}_\e=\int_{\PP(\R^d)}(\rho * \psi_\e)^{\otimes N}F'(d\rho),
\]
and similarly $F^{N}_\e, F''^{N}_\e$ and $G^{N}_\e$. Since these sequences are compatible we obtain that
\[
F^{N}_\e(V^{v}_N)=F^{N-1}_\e(V_{N-1})f_\e(v|V_{N-1}), \ G^{N}_\e(V^{v}_N)=G^{N-1}_\e(V_{N-1})g_\e(v|V_{N-1}).
\]
We first fix $L>0$ such that $L^{-\nu}<\eta$, and in view of the preivous step, and since $F^{'N,\e}+F^{''N,\e}=F^N_\e$ we have
\begin{align*}
	&\lt| \II^N_{\nu,\gamma} ( \theta F^{N}_\e+ (1-\theta) G^{N}_\e)- \theta \II^N_{\nu,\gamma}(F^{N}_\e)-(1-\theta)\II^N_{\nu,\gamma}(G^{N}_\e) \rt|\\
	&\leq \eta +  \int_{\R^{(d-1)N}} d V_{N-1} \int_{\R^{2d}}\frac{\theta(1-\theta)F_\e^{N-1}(V_{N-1})G_\e^{N-1}(V_{N-1})}{\theta F_\e^{N-1}(V_{N-1}) \sqrt{f_\e(v)f_\e(v_*)}+(1-\theta)G_\e^{N-1}(V_{N-1})\sqrt{g_\e(v)g_\e(v_*)}}\\
	& \quad \quad \quad \quad \quad \quad \quad \quad \quad \quad \quad \mb_{|v-v_*|< L} \frac{g_\e(v_*)\lt(\sqrt{f_\e(v)}-\sqrt{f_\e(v_*)}\rt)^2+f_\e(v_*)\lt(\sqrt{g_\e(v)}-\sqrt{g_\e(v_*)}\rt)^2}{|v-v_*|^{d+\nu}}dvdv_* \\
	&\leq \eta + C_\theta \int_{\R^{(d-1)N}} F^{''N-1}_\e(d V_{N-1}) \int_{\R^{2d}}\mb_{|v-v_*|< L} \frac{g_\e(v_*)\lt(\sqrt{f_\e(v)}-\sqrt{f_\e(v_*)}\rt)^2+f_\e(v_*)\lt(\sqrt{g_\e(v)}-\sqrt{g_\e(v_*)}\rt)^2}{\sqrt{g_\e(v)g_\e(v_*)}|v-v_*|^{d+\nu}}dvdv_* \\
		&+ C_\theta\int_{\R^{(d-1)N}} d V_{N-1} \int_{\R^{2d}}\frac{\theta(1-\theta)F_\e^{'N-1}(V_{N-1})G_\e^{N-1}(V_{N-1})}{\theta F_\e^{'N-1}(V_{N-1}) \sqrt{f_\e(v)f_\e(v_*)}+(1-\theta)G_\e^{N-1}(V_{N-1})\sqrt{g_\e(v)g_\e(v_*)}}\\
	& \quad \quad \quad \quad \quad \quad \quad \quad \quad \quad \quad \mb_{|v-v_*|< L} \frac{g_\e(v_*)\lt(\sqrt{f_\e(v)}-\sqrt{f_\e(v_*)}\rt)^2+f_\e(v_*)\lt(\sqrt{g_\e(v)}-\sqrt{g_\e(v_*)}\rt)^2}{|v-v_*|^{d+\nu}}dvdv_*.
\end{align*}

	Set now $u=\frac{r+s}{2}$ and $\delta=\frac{r-s}{2}$, and denote
\[
\B_u^{N-1}:=\lt\{(v_2,\cdots,v_N)\in \R^{d(N-1)}\, | \, \frac{1}{N-1}\sum_{i=2}^N\delta_{v_i}\in \B_u \rt\},
\]

Using that
\[
\frac{F_\e^{'N-1}(V_{N-1})G_\e^{N-1}(V_{N-1})}{\theta F_\e^{'N-1}(V_{N-1}) \sqrt{f_\e(v)f_\e(v_*)}+(1-\theta)G_\e^{N-1}(V_{N-1})\sqrt{g_\e(v)g_\e(v_*)}}\leq \frac{F_\e^{'N-1}(V_{N-1})}{(1-\theta)\sqrt{g_\e(v)g_\e(v_*)}}\mb_{(\B^{N-1}_u)^c} +\frac{G_\e^{N-1}(V_{N-1})}{\theta\sqrt{f_\e(v)f_\e(v_*)}}\mb_{\B^{N-1}_u},
\]	

we deduce 

\begin{align*}
&\lt| \II^N_{\nu,\gamma} ( \theta F^{N}_\e+ (1-\theta) G^{N}_\e)- \theta \II^N_{\nu,\gamma}(F^{N}_\e)-(1-\theta)\II^N_{\nu,\gamma}(G^{N}_\e) \rt|\\
&\leq \eta + C_\theta \int_{\R^{(d-1)N}} F_\e^{''N-1}(d V_{N-1}) \int_{\R^{2d}}\mb_{|v-v_*|< L} \frac{g_\e(v_*)\lt(\sqrt{f_\e(v)}-\sqrt{f_\e(v_*)}\rt)^2+f_\e(v_*)\lt(\sqrt{g_\e(v)}-\sqrt{g_\e(v_*)}\rt)^2}{\sqrt{g_\e(v)g_\e(v_*)}|v-v_*|^{d+\nu}}dvdv_*\\
&+  \int_{(\B^{N-1}_u)^c}F_\e^{'N-1}(d V_{N-1}) \int_{\R^{2d}}\mb_{|v-v_*|< L} \frac{g_\e(v_*)\lt(\sqrt{f_\e(v)}-\sqrt{f_\e(v_*)}\rt)^2+f_\e(v_*)\lt(\sqrt{g_\e(v)}-\sqrt{g_\e(v_*)}\rt)^2}{\sqrt{g_\e(v)g_\e(v_*)}|v-v_*|^{d+\nu}}dvdv_*\\
&+  \int_{\B^{N-1}_u}G_\e^{N-1}(d V_{N-1}) \int_{\R^{2d}}\mb_{|v-v_*|< L} \frac{g_\e(v_*)\lt(\sqrt{f_\e(v)}-\sqrt{f_\e(v_*)}\rt)^2+f_\e(v_*)\lt(\sqrt{g_\e(v)}-\sqrt{g_\e(v_*)}\rt)^2}{\sqrt{f_\e(v)f_\e(v_*)}|v-v_*|^{d+\nu}}dvdv_*
\end{align*}
In view of Lemma \ref{lem:reg}, we have for any $v,v_*$ such that $|v-v_*|\leq L$  
	\begin{align*}
		\frac{g(v_*)}{\sqrt{g(v)g(v_*)}}&\leq (c_{\e,L})^{-1}  \\
	\end{align*}	
	And using Taylor's expansion we have that
	\begin{align*}
		\frac{\lt(\sqrt{g(v)}-\sqrt{g(v_*)}\rt)^2}{\sqrt{g(v)g(v_*)}}&= \frac{1}{\sqrt{g(v)g(v_*)}}\lt( \int_0^1 \nabla \sqrt{g} (v+s(v_*-v))\cdot( v-v_*)ds\rt)^2\\
		&\leq (c_{\e,R})^{-1}\lt(\int_0^1 \frac{\nabla \sqrt{g}}{ \sqrt{g}} (v+s(v_*-v)) \cdot(v-v_*)ds\rt)^2\\
		&=(c_{\e,R})^{-2}\lt(\int_0^1 \nabla (\ln \sqrt{g}) (v+s(v_*-v)) \cdot(v-v_*)ds\rt)^2\\
		&\leq (c_{\e,R})^{-2} \|\nabla \ln g\|^2_{L^\infty} |v-v_*|^2.
	\end{align*}

Finally, if $g\in L^\infty$ and $\nabla \ln g\in L^\infty$, it holds $\nabla \sqrt{g}=\dem \sqrt{g}\nabla \ln g\in L^\infty$, and we conclude this step with 

\begin{align*}
&\lt| \II^N_{\nu,\gamma} ( \theta F^{N,\e}+ (1-\theta) G^{N,\e})- \theta \II^N_{\nu,\gamma}(F^{N,\e})-(1-\theta)\II^N_{\nu,\gamma}(G^{N,\e}) \rt|\\
&\leq \eta +C_{\e,L,\theta} \int_{\R^{(d-1)N}} F^{''N-1}(d V_{N-1})+  C_{\e,\theta} \int_{(\B^{N-1}_u)^c}F_\e^{'N-1}(d V_{N-1}) + C_{\e,\theta} \int_{\B^{N-1}_u}G_\e^{N-1}(d V_{N-1})
\end{align*}

 $\diamond$ Step three : Concentration \newline
 
 The end of the proof is then exactly taken from \cite[Lemma 5.10]{HM}. Nevertheless we reproduce it here for the sake of completeness. First we treat the third trerm in the above r.h.s. by observing that
 $F^{''}=\mb_{\mathcal{B}_r\setminus \mathcal{B}_s}F$. Therefore
 \[
 \begin{aligned}
 \int_{\R^{d(N-1)}} F_\e''^{N-1}(dV^{N-1})=\int_{\PP(\R^d)} \mb_{\mathcal{B}_r\setminus \mathcal{B}_s}(\rho_\e)F(d\rho).
 \end{aligned}
 \]
 Due to Lebesgue's dominated convergence Theorem, the r.h.s. in the above identity converges to $0$. Therefore one can chose some $s<r$ such that 
 \[
 C_{\e,\theta}\int_{\R^{d(N-1)}} F_\e''^{N-1}(dV^{N-1}) <\eta,
 \]
 uniformly in $N$ and $\e$. Then for $V^{N-1}\notin \tilde{\B}_u^{N-1}$ and $\rho\in \B_s$ we find that
 \[
 \begin{aligned}
 W_1\lt(\frac{1}{N-1}\sum_{i=2}^N\delta_{v_i},\rho*\psi_\e\rt)&\geq 	W_1\lt(\frac{1}{N-1}\sum_{i=2}^N\delta_{v_i},f_1\rt)-	W_1\lt(f_1,\rho\rt)-	W_1\lt(\rho,\rho*\psi_\e\rt)\\
 &\quad \geq u-s-c\e\geq \frac{\delta}{2},
 \end{aligned}
 \]
 for any $\e>0$ small enough. Therefore using a Chebychev-like argument it holds
 \[
 \begin{aligned}
 \int_{\tilde{\B}_u^{N-1,c}}F_\e'^{N-1}(dV^{N-1})&=\int_{\PP(\R^d)}\lt(\int_{\R^{d(N-1)}}\mb_{\tilde{\B}_u^{N-1,c}} (\rho*\psi_\e)^{\otimes(N-1)}\rt) F'(d\rho)\\
 &\quad \leq \frac{2}{\delta}\int_{\PP(\R^d)}\lt(\int_{\R^{d(N-1)}} W_1\lt(\frac{1}{N-1}\sum_{i=2}^N\delta_{x_i},\rho*\psi_\e\rt)(\rho*\psi_\e)^{\otimes(N-1)}(dV^{N-1})\rt) F'(d\rho)
 \end{aligned}
 \]
 We claim that there is a constant $C$ depending only on $\kappa$ (see \cite[Theroem 1]{FG} in case $d=3,p=1,q=\kappa<2$) such that it holds
 \bq
 \label{eq:concHM}
 \int_{\R^{d(N-1)}} W_1\lt(\frac{1}{N-1}\sum_{i=2}^N\delta_{x_i},\rho*\psi_\e\rt)(\rho*\psi_\e)^{\otimes(N-1)}(dX^{N-1})\leq C \|\rho*\psi_\e\|_{L^1_\kappa}^{\frac{1}{\kappa}} (N-1)^{-\lt(1-\frac{1}{\kappa}\rt)}.
 \eq
 Note that \cite[Remark 2.12]{HM} provides the same result with the exponent $1-\frac{1}{\kappa}$ replaced with $\gamma \in \lt(0,\frac{1}{3+\frac{2}{\kappa}}\rt)$, but the rate of convergence does not play any role in the proof. Summing up \eqref{eq:concHM} w.r.t. $F'$, yields
 \[
 \begin{aligned}
 \int_{\tilde{\B}_u^{N-1,c}}F'^{N-1}(dV^{N-1})&\leq \frac{C}{\delta(N-1)^{\lt(1-\frac{1}{\kappa}\rt)}}\int_{\PP(\R^d)}  \|\rho*\psi_\e\|_{L^1_\kappa}^{\frac{1}{\kappa}} F'(d\rho)\\
 &\quad \leq \frac{C}{\delta(N-1)^{\lt(1-\frac{1}{\kappa}\rt)}}\lt(\int_{\PP(\R^d)}  \|\rho\|_{L^1_\kappa}\pi(d\rho)+ \|\psi_\e\|_{L^1_\kappa}\rt)^{\frac{1}{\kappa}},
 \end{aligned}
 \]
 since
 \[
 \begin{aligned}
  \|\rho*\psi_\e\|_{L^1_\kappa}&=\int_{\R^d}\int_{\R^d} \lal x \ral^{\kappa}\rho(x-y)\psi_\e(y)dxdy = \int_{\R^d}\int_{\R^d} \lal x+y \ral^{\kappa}\rho(x)\psi_\e(y)dxdy\\
 &\quad \leq 2^\kappa \lt(\int_{\R^d}\lal x\ral^\kappa\rho(x)dx+\int_{\R^d}\lal y\ral^\kappa\psi_\e(y)dy\rt)
 \end{aligned}
 \]
 
	Treating in the exact same fashion the integral w.r.t. $G^{N-1}$ enables to conclude that for any $\e>0$ we have
\[
\forall \eta>0,\, \exists N_\eta, \, \mbox{s.t.} \, \forall N \geq N_{\eta}, \, \lt| \II^N_{\nu,\gamma} ( \theta F^{N}_\e+ (1-\theta) G^{N}_\e)- \theta \II^N_{\nu,\gamma}(F^{N}_\e)-(1-\theta)\II^N_{\nu,\gamma}(G^{N}_\e) \rt|\leq 4\eta.
\]
$\diamond$ Final step\newline
Gathering all the estimates obtained in the previous steps yields for any $\e>0$
\[
\lim_{N\rightarrow +\infty }\lt|\II^N_{\nu,\gamma} ( \theta F^{N}_\e+ (1-\theta) G^{N}_\e)- \theta \II^N_{\nu,\gamma}(F^{N}_\e)-(1-\theta)\II^N_{\nu,\gamma}(G^{N}_\e)\rt|=0.
\]
Hence we deduce
\[
\begin{aligned}
\tilde{\II}_{\nu,\gamma}( \theta F^{\e}+ (1-\theta) G^{\e})&=\sup_{N\geq 1}\II^N_{\nu,\gamma}( \theta F^{N}_\e+ (1-\theta) G^{N}_\e)=\lim_{N\rightarrow+\infty}\II^N_{\nu,\gamma}( \theta F^{N}_\e+ (1-\theta) G^{N}_\e)\\
&\quad = \theta \lim_{N\rightarrow+\infty}\II^N_{\nu,\gamma}(F^{N}_\e)+(1-\theta)\lim_{N\rightarrow+\infty}\II^N_{\nu,\gamma}(G^{N}_\e)\\
&\quad =\theta \sup_{N\geq 1}\II^N_{\nu,\gamma}(F^{N}_\e)+(1-\theta)\sup_{N\geq 1}\II^N_{\nu,\gamma}(G^{N}_\e)\\
&\quad = \theta \tilde{\II}_{\nu,\gamma}(F^\e)+(1-\theta)\tilde{\II}_{\nu,\gamma}(G^\e).
\end{aligned}
\]

But using the convexity of the functional $(x,y)\mapsto \lt(\sqrt{x}-\sqrt{y}\rt)^2$ and Jensen's inequality yields 
\[
\begin{aligned}
\tilde{\II}_{\nu,\gamma}(\theta F^{\e}+ (1-\theta) G^{\e})=\sup_{N\geq 1}\II_{\nu,\gamma}^N(\theta F^{N}_\e+ (1-\theta) G^{N}_\e)\leq \sup_{N\geq 1}\II^N(\theta F^{N}+ (1-\theta) G^{N})=\tilde{\II}_{\nu,\gamma}(\theta F+ (1-\theta) G). 	
\end{aligned}
\] 
Morever it is clear from the fact that for each $N\geq 2$, the functional $\II^N_{\nu,\gamma}$ is l.s.c.  w.r.t. the weak convergence in $\PP(\R^{dN})$, that $\tilde{\II}_{\nu,\gamma}$ is l.s.c.  w.r.t. the weak convergence in $\PP(\PP(\R^d))$. But since $\theta F^{\e}+ (1-\theta) G^{\e}{\overset{*}{\rightharpoonup}}\theta F+ (1-\theta) G$ in $\PP(\PP(\R^d))$ we get that
\[
\lim_{\e\rightarrow 0} \tilde{\II}_{\nu,\gamma}(\theta F^{\e}+ (1-\theta) G^{\e})=\tilde{\II}_{\nu,\gamma}(\theta F+ (1-\theta) G).
\]
Therefore 
\[
\tilde{\II}_{\nu,\gamma}(\theta F+ (1-\theta) G)=\theta\tilde{\II}_{\nu,\gamma}(F)+(1-\theta)\tilde{\II}_{\nu,\gamma}(G),
\]
which concludes the proof.

\end{proof}

Proposition \ref{thm:Fish2} then follows from \cite[Lemma 5.6]{HM}. We leave the reader check that points $\textit{(i)-(ii)}$ of Lemma \ref{lem:aff} and Lemma \ref{lem:aff2} consist in checking that the functionals $(\II^N_{\nu,\gamma})_{N\geq 2}$ satisfy the technical assumptions of \cite[Lemma 5.6]{HM}.

\end{document}